\documentclass[10pt,reqno]{amsart}
\RequirePackage[OT1]{fontenc}
\RequirePackage{amsthm,amsmath}
\usepackage{mathtools}
\RequirePackage[numbers]{natbib}
\RequirePackage[colorlinks,linkcolor=blue,citecolor=blue,urlcolor=blue]{hyperref}
\usepackage{amssymb}
\usepackage{anysize}
\usepackage{hyperref}
\usepackage{extarrows}
\usepackage{multirow}
\usepackage{natbib}
\usepackage{stmaryrd}
\usepackage{color}
\usepackage{amsmath}
\usepackage{enumitem}
\usepackage{mathtools} 
\usepackage{dsfont} 
\usepackage{comment}
\usepackage{soul}
\usepackage{amsmath,amsthm,amssymb, bbm, array, graphicx, tikz, tikz-cd}
\usepackage{pgfplots}
\usepackage{bm}
\usepackage{extarrows}

\usepackage{bm}

\numberwithin{equation}{section}
\theoremstyle{plain} 
\newtheorem{theorem}{Theorem}[section]
\newtheorem{lemma}[theorem]{Lemma}
\newtheorem{corollary}[theorem]{Corollary}
\newtheorem{proposition}[theorem]{Proposition}

\newtheorem{assumption}[theorem]{Assumption}

\newtheorem{remark}[theorem]{Remark}

\renewcommand{\Re}{\mathrm{Re}\,}
\renewcommand{\Im}{\mathrm{Im}\,}

\newcommand{\E}{{\mathbf E }}

\newcommand{\R}{{\mathbb R }}
\newcommand{\N}{{\mathbb N}}

\renewcommand{\P}{{\mathbf P}}
\newcommand{\C}{{\mathbb C}}

\newcommand{\B}{{\mathcal B}}

\newcommand{\ov}{\overline}

\DeclareMathOperator{\sgn}{sgn}

\newcommand{\dif}{\operatorname{d}\!{}}

\newcommand{\xx}{{\bm x}}
\newcommand{\yy}{{\bm y}}
\newcommand{\ea}{{\bm e}_a}
\newcommand{\eB}{{\bm e}_B}

\newcommand{\ii}{\mathrm{i}}

\newcommand{\dd}{\mathrm{d}}
\newcommand{\ie}{\emph{i.e., }}
\newcommand{\eg}{\emph{e.g., }}

\newcommand{\wt}{\widetilde}

\newcommand{\wh}{\widehat}

\newcommand{\uu}{\mathbf{u}}
\newcommand{\vv}{\mathbf{v}}

\newcommand{\nc}{\normalcolor}

\newcommand{\bs}{\boldsymbol}

\def\ga{G^{z_1}}
\def\gb{G^{z_2}}

\def\one{\mathds{1}}

\def\<{\langle}
\def\>{\rangle}

\renewcommand{\mathbf}[1]{\bs{#1}}
 
\marginsize{25mm}{25mm}{25mm}{26mm}

\allowdisplaybreaks

\begin{document}

	\begin{minipage}{0.85\textwidth}
		\vspace{2.5cm}
	\end{minipage}
	\begin{center}
		\large\bf 	Optimal decay of eigenvector overlap for non-Hermitian random matrices
		
	\end{center}

	\renewcommand*{\thefootnote}{\fnsymbol{footnote}}	
	\vspace{0.5cm}
	
	\begin{center}
		\begin{minipage}{1.0\textwidth}
			\begin{minipage}{0.33\textwidth}
				\begin{center}
					Giorgio Cipolloni\\
					\footnotesize 
					{University of Arizona}\\
					{\it gcipolloni@arizona.edu}
				\end{center}
			\end{minipage}
			\begin{minipage}{0.33\textwidth}
				\begin{center}
					L\'aszl\'o Erd\H{o}s\footnotemark[2]\\
					\footnotesize 
					{IST Austria}\\
					{\it lerdos@ist.ac.at}
				\end{center}
			\end{minipage}
			\begin{minipage}{0.33\textwidth}
				\begin{center}
					Yuanyuan Xu\footnotemark[2]\\
					\footnotesize 
					{AMSS,~CAS}\\
					{\it yyxu2023@amss.ac.cn}
				\end{center}
			\end{minipage}
		\end{minipage}
	\end{center}
	
	\bigskip

	\footnotetext[2]{\footnotesize{Partially supported by ERC Advanced Grant "RMTBeyond" No.~101020331.}}

	\renewcommand*{\thefootnote}{\arabic{footnote}}
	
	\vspace{5mm}

	\begin{center}
		\begin{minipage}{0.91\textwidth}\small{
				{\bf Abstract.}}
			We consider the standard overlap 
			 $\mathcal{O}_{ij}: =\langle \mathbf{r}_j, \mathbf{r}_i\rangle\langle \mathbf{l}_j, \mathbf{l}_i\rangle$ of 
			 any bi-orthogonal family of left and right eigenvectors of a large random matrix $X$ 
			 with centred i.i.d. entries and
			 we  prove that it decays as an inverse second power of the distance
			 between the corresponding eigenvalues. This extends similar results  for 
			 the complex Gaussian ensemble  from Bourgade and Dubach \cite{Dubach}, as well as
			 Benaych-Georges and Zeitouni \cite{Zeitouni},
			 to any i.i.d. matrix ensemble in both symmetry classes. As a main
			 tool, we prove a two-resolvent local law for the Hermitisation of $X$ uniformly in the spectrum
			 with  optimal decay rate and optimal dependence on the density near the spectral edge.
		\end{minipage}
	\end{center}

	\vspace{5mm}
	
		{\footnotesize
		{\noindent\textit{Date}:  \today \\
		}
		
		\vspace{2mm}

		\thispagestyle{headings}

		\bigskip

		\normalsize

\section{Introduction}

We consider large $n\times n$ non--Hermitian matrices $X$ with independent, identically distributed (i.i.d.) centered entries, with the standard normalization $\E |X_{ij}|^2=\frac{1}{n}$ (\emph{i.i.d. matrices}).
 While the eigenvalue statistics of fairly general large Hermitian random matrices have been very well understood in the last fifteen years, 
  those of i.i.d. matrices have only been understood very recently both in the bulk \cite{DY24, MO24, Osman24} and at the edge \cite{edge} of the spectrum; see also \cite{Afa24, CCEJ24, CEJ24, Liu23, Liu24, Zhang24} for recent results for 
 \emph{deformed} models of the form $A+X$, where $A$ is a deterministic matrix.  
 One of the main difficulties in  understanding the eigenvalue statistics in the non--Hermitian setting is  the high instability of the spectrum, which is partly manifested by the behavior of the their eigenvectors. While in the Hermitian case, e.g. for GOE/GUE\footnote{We say that an $n\times n$ random matrix $W$ is a GUE (resp. GOE) matrix if it is Hermitian $W=W^*$, the entries above the diagonal $W_{ij}$ are such that $\sqrt{n}W_{ij}$ are 
 independent standard complex (real) Gaussian random variables, and the diagonal entries are real Gaussian random variables of variance $1/n$ (resp. $2/n$).} matrices for simplicity, the eigenvectors are distributed according to the Haar measure, for Ginibre\footnote{We say that an $n\times n$ random matrix is a complex (resp. real) Ginibre matrix if its entries are centered complex (resp. real) Gaussian random variables with variance $1/n$.} matrices the left/right eigenvectors are strongly correlated and difficult to grasp even for the Gaussian case (the general i.i.d. case is expected to be even much harder). Recall that in general, a non--Hermitian matrix $X$ has a two sets of eigenvectors $\mathbf{l}_i, \mathbf{r}_i$, which, assuming that each eigenvalue $\sigma_i$ of $X$ is simple, are defined by
\begin{equation}
X\mathbf{r}_i=\sigma_i \mathbf{r}_i, \qquad\quad \mathbf{l}_i^tX=\sigma_i\mathbf{l}_i^t, \qquad i=1,2,\ldots, n.
\end{equation}
We normalize the left/right eigenvectors to form a bi-orthogonal basis, i.e. 
$\langle \overline{\mathbf{l}}_j, \mathbf{r}_i\rangle=\mathbf{l}_j^t\mathbf{r}_i=\delta_{ij}$. Instead of studying
the eigenvectors directly, one usually considers the
scale-invariant quantity
\begin{equation}
\label{eq:odo}
\mathcal{O}_{ij}:= \langle \mathbf{r}_j, \mathbf{r}_i\rangle\langle \mathbf{l}_j, \mathbf{l}_i\rangle,
\end{equation}
called \emph{eigenvector overlaps}.
For random matrices, overlaps
emerge in many problems both in mathematics \cite{Zeitouni, Dubach, Fyo18, Cra18, MC00, Wal15} 
and  physics \cite{Ake20, Bel17, CM98, Cipent23, Cipneth24, Ghosh23, Fyo22, Fyo23, Fyo02, Fyo12, JNNPZ99, Roy23}. Two of their most prominent motivations are:
\begin{enumerate}

\item  $\mathcal{O}_{ij}$ determine the correlation of the evolution of the eigenvalues $\sigma_i$ under a \emph{non--Hermitian} analog of the \emph{Dyson Brownian motion}, which was 
 rigorously \nc derived in \cite[Appendix A]{Dubach} (see also \cite{Burda14, Grela18}).

\item For $i=j$, the quantity $\sqrt{\mathcal{O}_{ii}}$ is known as the \emph{eigenvalue condition number}
in numerical analysis, as it determines how sensitive $\sigma_i$ is to small perturbation
 (especially relevant  in the context of \emph{smoothed analysis} \cite{San06} via random matrices,  see e.g. \cite{Banks20, CES22con}).

\end{enumerate}

Despite the central role of the overlaps,
 rigorous results about their size and distribution have appeared only recently. Inspired by the seminal work of Chalker and Mehlig \cite{CM98} computing the expectations of the $\mathcal{O}_{ij}$ for the  complex Ginibre ensemble,
  Bourgade and Dubach \cite{Dubach}
 also computed the second moment of the off-diagonal overlap and diagonal overlap,
\begin{equation}
\label{eq:varcomp}
\E\big[|\mathcal{O}_{ij}|^2\big| \sigma_i=z_1, \,\, \sigma_j=z_2\big]\approx\frac{(1-|z_1|^2)(1-|z_2|^2)}{|z_1-z_2|^4},
\end{equation}
 \begin{equation}
	\label{eq:varcomp_diag}
	\E\big[ \mathcal{O}_{ii} \mathcal{O}_{jj} \big| \sigma_i=z_1, \,\, \sigma_j=z_2\big]\approx n^2(1-|z_1|^2)(1-|z_2|^2),
\end{equation}
identifying a power law (quadratic) decay of $\mathcal{O}_{ij}$ for far away eigenvalues.
Beyond Ginibre ensembles,  Benaych-Georges and Zeitouni \cite{Zeitouni} showed that the closely related quantity 
 $\sqrt{n}(\sigma_i-\sigma_j)\langle \mathbf{r}_i,\mathbf{r}_j\rangle/\|\mathbf{r}_i\|\|\mathbf{r}_j\|$
    is sub-Gaussian (uniformly in $n$) for more general non-normal matrices defined by an invariant distribution,
   and explicitly compute its distribution in the complex Ginibre case. Both
   results identify a quadratic decay, however their proofs
    rely on the invariance (or even Gaussianity) of the underlying matrix ensemble.
     
   Is this quadratic decay ubiquitous, in particular does it extend to 
    general real and complex i.i.d. matrix ensembles without invariance properties? In this paper we answer affirmatively
    by proving
    that the optimal bound (modulo the $n^\xi$ factor) as indicated by (\ref{eq:varcomp})-(\ref{eq:varcomp_diag}) 
\begin{equation}
\label{eq:off-diagover}
\frac{|\mathcal{O}_{ij}|}{\sqrt{\mathcal{O}_{ii}\mathcal{O}_{jj}}}\le \frac{n^\xi}{n|\sigma_i-\sigma_j|^2+1}
\end{equation}
holds uniformly in the spectrum
with very high probability for any small $\xi>0$. 
 Previously, only two non-optimal bounds were available. In \cite[Theorem 3.1]{mesoCLT} it was
 shown that $|\mathcal{O}_{ij}|/\sqrt{\mathcal{O}_{ii}\mathcal{O}_{jj}}\le n^{-\epsilon}$ for bulk eigenvectors,  if 
 $|\sigma_i-\sigma_j|\ge n^{-1/2+\epsilon’}$
 which is off by almost an entire factor $n$ for far away eigenvalues.
  Close to the edge of the spectrum, $|\sigma_i|, 
 |\sigma_j|\approx 1$, half of the optimal decay, $1/(\sqrt{n}|\sigma_i-\sigma_j|)$, was obtained in \cite[Corollary 3.4]{gumbel}.  We remark that the quadratic decay bound obtained in (\ref{eq:off-diagover}) is better than $1/(\sqrt{n}|\sigma_i-\sigma_j|)$ if $|\sigma_i-\sigma_j|\geq n^{-1/2+\xi}$. 
 Neither proof can be directly improved and  the optimal bound \eqref{eq:off-diagover} requires to resolve two new 
 challenges.

We now explain  the novelties in our proof of  \eqref{eq:off-diagover}.
The main inputs are the so-called \emph{multi--resolvent local laws} for products of resolvents of the Hermitisation of $X-z$.
Since the non-Hermitian eigenvectors $\mathbf{l}_i, \mathbf{r}_i$ are not accessible directly, we relate them to the eigenvectors $\mathbf{w}_i^z=((\mathbf{u}_i^z)^*,(\mathbf{v}_i^z)^*)^*\in \C^{2n}$ of the $(2n)\times (2n)$ Hermitized matrix
\begin{equation}
H^z:=\left(\begin{matrix}
0 & X-z \\
(X-z)^*& 0
\end{matrix}\right).
\end{equation}
Note that
 $\mathbf{u}_i^z,\mathbf{v}_i^z\in \C^n$ are also the left/right singular vectors corresponding to the singular value $\lambda_i^z$ of a family of matrices $X-z$ parametrized by $z\in \C$. 
We assume that the singular values $\lambda_i^z$ are labelled in an increasing order. 

Eigenvectors of $X$ and singular vectors of $X-z$ are not  related in general, except in the
special case when $z$ is an eigenvalue. Simple algebra shows that
\begin{equation}\label{eq:evevrel}
\lambda^z_1=0 \Longleftrightarrow
0\in\mathrm{Spec}(H^z) \Longleftrightarrow 0\in \mathrm{Spec}(X-z)\Longleftrightarrow z\in\mathrm{Spec}(X),
\end{equation}
and in this case the eigenvectors of $X$ with eigenvalue $z$ and the singular vectors of $X-z$ with singular value zero coincide. This fundamental relation will be our starting point
 to study the eigenvalues and the eigenvectors of the non--Hermitian $X$ 
 via the study of the spectrum of the Hermitized matrix $H^z$. 
 Our goal is thus to estimate quantities of the form $|\langle \mathbf{u}_1^{z_1}, \mathbf{u}_1^{z_2}\rangle|, |\langle \mathbf{v}_1^{z_1}, \mathbf{v}_1^{z_2}\rangle|$ with high probability simultaneously for every $|z_1|,|z_2|\le 1$. In fact, if we have the desired bound on these quantities, by choosing $z_1=\sigma_i$, $z_2=\sigma_j$, we can readily conclude \eqref{eq:off-diagover}. By spectral decomposition, one can easily see that for the desired  bound on
$|\langle \mathbf{u}_1^{z_1}, \mathbf{u}_1^{z_2}\rangle|, |\langle \mathbf{v}_1^{z_1}, \mathbf{v}_1^{z_2}\rangle|$, it is actually enough to estimate
\begin{equation}
\label{eq:2gllaw}
\langle \Im G^{z_1}(\ii\eta_1) \Im G^{z_2}(\ii\eta_2)\rangle\lesssim \frac{\rho_1\rho_2}{|z_1-z_2|^2}, \qquad\qquad\quad G^z(\ii\eta):= (H^z-\ii\eta)^{-1},
\end{equation}
for appropriate choices of $\eta_1,\eta_2$
(eventually $\eta_i$ is  chosen a bit above  the local fluctuation scale of the eigenvalues of $H^{z_i}$ close to zero).
Here $\rho_i$ denotes the local density of the eigenvalues of $H^{z_i}$. 

Although we only need the upper bound  \eqref{eq:2gllaw} for  \eqref{eq:off-diagover}, along the way
in Theorem~\ref{thm:2G_edge}
we will  prove a stronger result, called \emph{two-resolvent local law},
  that even  identifies the leading term in the relevant regime of $\eta_i$’s:
\begin{equation}
\label{eq:2gllaw1}
\Big| \langle \Im G^{z_1}(\ii\eta_1) \Im G^{z_2}(\ii\eta_2) - M_{12}\rangle\Big| = o\Big( \frac{\rho_1\rho_2}{|z_1-z_2|^2}\Big),
\end{equation}
where $M_{12}$ is a deterministic matrix computed from the underlying \emph{Dyson equation}.

To prove \eqref{eq:2gllaw1}, we rely on the \emph{zigzag strategy} \cite{mesoCLT, CEH23}, which involves three main steps: i) controlling $\langle \Im G^{z_1}(\ii\eta_1) \Im G^{z_2}(\ii\eta_2) - M_{12}\rangle$ for large values
of $\eta_1,\eta_2\sim 1$, ii) [\emph{zig-step}] using the \emph{method of characteristics} (introduced in this form in \cite{AH20, B21, HL19} for single resolvents, see \cite[Section 1.4]{CEHK24} for more history)
 to show that this bound can be propagated down to optimally small spectral parameters $\eta_i$
 at the price of adding a Gaussian component to the matrix $X$, iii) [\emph{zag-step}] 
 using a version of the Green's function comparison theorem (GFT)
  to remove the Gaussian component. While on a high level this approach has  been used quite extensively, in the current paper we face a novel challenge: we need to extract an additional small factor of order $\rho$
   coming from $\Im G^z$ (compared to $G^z$) close to the edge of the spectrum in the context of two different spectral parameters $z_1,z_2$. This requires to study the evolution of $\langle \Im G_1 \Im G_2\rangle$, with $G_i:=G^{z_i}(\ii\eta_i)$, along the characteristics flow, which is more delicate than the one for $\langle G_1 G_2\rangle$ investigated in \cite{mesoCLT, CEH23} and also requires a more delicate analysis at the GFT level.
   A bound on this quantity was estimated in \cite[Part B of Theorem 3.3]{gumbel} close to the edge of the spectrum with a
   non-optimal decay $1/|z_1-z_2|$.
   This non-optimality is a consequence of bounding $\langle G_1 G_2\rangle$ and then writing $2\ii\Im G=G-G^*$, rather than analyzing $\langle \Im G_1 \Im G_2\rangle$ directly and extract a cancellation between $G$ and $G^*$
   near the edge, as we do instead in the current paper. \nc 
   
   The second challenge we resolve in this paper is to estimate optimally the deterministic $M$-terms, such as $M_{12}$
   in ~\eqref{eq:2gllaw1}, even for longer resolvent chains that naturally arise along the zig-flow
   (we need three-resolvent chains).
   In fact,  the precision of the zigzag method ultimately depends only on such estimates on $M$-terms.
   While an explicit recursion is available for expressing the $M$-terms, the formula is inherently unstable 
   in the small $\eta$ regime as it contains delicate cancellations. In \cite{gumbel} these cancellations were 
   unravelled but only outside of the spectrum and without the extra $\rho$-factors.
   Now we find the optimal estimate uniformly in the spectrum
   and with the correct $\rho$ factors.

 We remark that the analogue of 
 the quadratic decay in  $z_1-z_2$ in \eqref{eq:2gllaw}, and hence in \eqref{eq:off-diagover}, appears
 in a related Hermitian model as well. In \cite{CEHK24} the authors study a Wigner matrix $W$ with two different (deterministic) deformations,  $W+D_1$ and $W+D_2$. They show that the product of their resolvents $\langle G_1 G_2\rangle$, 
 hence their eigenvector overlap, 
 exhibits essentially a quadratic decay $1/\langle (D_1-D_2)^2\rangle$, although the precise statement is
 modulated by the difference in the spectral parameters as well.  However, 
  the results therein are optimal only within the bulk of the spectrum, 
  again because the smallness coming from $\Im G^z$ was not exploited.

     \bigskip

 The central role that  \eqref{eq:2gllaw} plays in studying overlaps  underlines the  importance of multi-resolvent local 
 laws in the theory of random matrices. For many years, single resolvent local laws have been the technical workhorses 
 behind the proofs of Wigner-Dyson-Mehta and related universality results on local eigenvalue statistics
 since they provided almost optimal a priori bounds on individual  eigenvalues  and eigenvectors.
 These are the standard \emph{rigidity} and \emph{delocalisation} results used regularly in arguments based 
 on the Dyson Brownian motion (DBM) and the Green function comparison theorems (GFT).
 A lot of relevant information on random matrices, however,  are not accessible via 
 a single resolvent. Key physical concepts such as
 eigenvector overlaps,  thermalisation phenomena, out-of-time-order correlations etc. 
 are inherently connected with multi-resolvent chains of the form $G_1A_1G_2A_2\ldots,$ where
 $G_i$ are random resolvents and $A_i$’s are deterministic matrices.  Besides these obvious motivations, 
 the recent  results on bulk universality of eigenvalues and eigenvectors for i.i.d. matrices,
 initiated by the breakthrough papers \cite{MO24, Osman24} using the partial Schur decomposition,
 also heavily rely on multi-resolvent local laws as a priori technical inputs.  
 The Schur decomposition method is powerful to prove the universality of the distribution of
 the diagonal overlaps $\mathcal{O}_{ii}$ \cite{Osmevectors}, but it cannot yet access to 
 off-diagonal overlaps. Our result gives \eqref{eq:off-diagover} an optimal a priori bound on $\mathcal{O}_{ij}$
 and we also expect that  multi-resolvent local laws will play an important role
 in resolving the open conjecture on the universality of $\mathcal{O}_{ij}$ for i.i.d. matrices.

We close this introduction with a quick overview of
 the recent vast literature about overlaps \eqref{eq:odo} for various random matrix ensembles. 
 With the exception of \cite{Dubach}, 
 all these works concern only diagonal overlaps $\mathcal{O}_{ii}$. 
  Long after   the classical paper \cite{CM98}, the two main works that revived the interest in overlaps of non--Hermitian eigenvectors in the context of random matrix theory are \cite{Dubach, Fyo18} on Ginibre ensemble.
 Besides the results on the off-diagonal overlaps in \cite{Dubach} that we already described, 
 the distribution of $\mathcal{O}_{ii}$
 and  the correlation of $\mathcal{O}_{ii}\mathcal{O}_{jj}$ for bulk eigenvalues in the complex case have also been computed
 in that paper.
 In \cite{Fyo18} the author computes the distribution of $\mathcal{O}_{ii}$ both in the bulk and at the edge of the spectrum for complex matrices, as well as for overlaps corresponding to real eigenvalues of real Ginibre matrices. 
 Expectation of overlaps for complex eigenvalues was covered in \cite{Wur23}. \nc
Beyond Gaussian ensembles,   in the context of general i.i.d. matrices, the precise (in terms of $n$--dependence) size of $\mathcal{O}_{ii}$ was identified in \cite{CEHS23, HC23} (see also for previous non-optimal results in terms of the $n$--dependence \cite{Banks20, Jain21, Tik20}), and more recently even the universality of its distribution 
has been established \cite{Osmevectors}. Two very recent results \cite{DYYY24, Osmsingvec} also prove 
 the universality of the Gaussian fluctuations directly for  the eigenvectors' entries of i.i.d. matrices. 
 Eigenvector overlaps have also been studied for several other models even beyond i.i.d. matrices,
 including deformed Ginibre matrices \cite{Banks21, San06, Zhang24e}, symplectic Ginibre matrices \cite{Ake20, Ake24}, non--Hermitian Brownian motion \cite{BCH24, Esaki23}, the elliptic ensemble \cite{Crump24, Crump24bis, Fyo21, Tar24}, spherical and truncated unitary ensembles \cite{Dubach21, Dubach23, Noda23}, and non-normal dynamical systems \cite{Morim24}.
 
\subsection*{Acknowledgment} We thank Yan Fyodorov for pointing out the reference \cite{JNNPZ99} and for useful comments. We also thank the anonymous referee for their several comments that substantially improved the presentation of this manuscript.

\subsection*{Conventions and notations}

For integers $k,l\in\N$, with $k<l$, we use the notations $[k,l]:= \{k,\dots, l\}$ and $[k]:=[1,k]$. For positive quantities $f,g$ we write $f\lesssim g$ and $f\sim g$ if $f \le C g$ or $c g\le f\le Cg$, 
respectively, for some constants $c,C>0$ which depend only on the constants appearing 
in~\eqref{eq:hmb}. We set $f\wedge g= \min\{ f, g\}$ and $f \vee  g= \max\{ f, g\}$. 
Furthermore, for $n$--dependent positive
quantities $f_n,g_n$ we use the notation $f_n\ll g_n$ to denote that $\lim_{n\to\infty} (f_n/g_n)=0$. 
Throughout the paper $c,C>0$ denote small and large constants, respectively, which may change from line to line. We denote vectors by bold-faced lower case Roman letters $\bm{x},\bm{y}\in\C^d$, for some $d\in\N$. Vector and matrix norms, $\lVert\bm{x}\rVert$ and $\lVert A\rVert$, indicate the usual Euclidean norm and the corresponding induced matrix norm. For any $d \times d$ matrix $A$ we use $\langle A\rangle:= d^{-1}\mathrm{Tr} A$ to denote the normalized trace of $A$. Moreover, for vectors ${\bm x}, {\bm y}\in\C^d$ we define the usual scalar product 
$\langle \bm{x},\bm{y}\rangle:= \sum_{i=1}^d \overline{x_i} y_i$, and use $(A)_{\xx \yy}$ to denote the inner product $\<\xx, A \yy\>$.

\noindent

We define the following special $2n\times 2n$ matrices 
\begin{equation}\label{eq:defF}
	E_1:=
	\begin{pmatrix}
		1  &  0 \\
		0   & 0
	\end{pmatrix}, 
\qquad E_2=
	\begin{pmatrix}
		0  &  0 \\
		0   & 1
	\end{pmatrix}, 
\qquad
	E_\pm:=E_1\pm E_2=\begin{pmatrix}
		1  &  0 \\
		0   & \pm 1
	\end{pmatrix}, 
\qquad
F:=
\begin{pmatrix}
0  &  1 \\
0   & 0
\end{pmatrix},
\end{equation}
in particular, $E_+$ is  the $2n$--dimensional identity matrix.

\noindent

We will often use the concept of ``with very high probability'' for an $n$-dependent event,
meaning that for any fixed $D>0$ the probability of the event is bigger than $1-n^{-D}$ if $n\ge n_0(D)$. Moreover, we use the convention that $\xi>0$ denotes an arbitrary small positive exponent 
which is independent of $n$. We recall the standard notion of \emph{stochastic domination}:  
given two families of non-negative random variables
\[
X=\left(X^{(n)}(u) \,:\, n\in\N, u\in U^{(n)}\right)\quad\text{and}\quad Y=\left(Y^{(n)}(u) \,:\, n\in\N, u\in U^{(n)}\right)
\] 
indexed by $n$ (and possibly some parameter $u$  in some parameter space $U^{(n)}$), 
we say that $X$ is {\it stochastically dominated} by $Y$, if for any $\xi, D>0$ we have \begin{equation}
	\label{stochdom}
	\sup_{u\in U^{(n)}} \P\left[X^{(n)}(u)>n^\xi  Y^{(n)}(u)\right]\le n^{-D}
\end{equation}
for large enough $n\geq n_0(\xi,D)$. In this case we use the notation $X\prec Y$ or $X= O_\prec(Y)$
and we say that $X(u)\prec Y(u)$ {\it holds uniformly} in $u\in U^{(n)}$.

\section{Main results}

We consider \emph{i.i.d.  matrices} $X \in \C^{n\times n}$ (complex case) or $X \in \R^{n\times n}$ (real case),
 i.e. $n\times n$ matrices with independent and identically distributed (i.i.d.) entries $x_{ab}\stackrel{\mathrm{d}}{=}n^{-1/2}\chi$, where the distribution of $\chi$ satisfies the following assumptions:
\begin{assumption}
	\label{ass:mainass}
	We assume that $\E \chi=0$, $\E |\chi|^2=1$, additionally $\E \chi^2=0$ in the complex case. Furthermore, we require that for any $p\in\N$, there exists a constant $C_p>0$ such that
	\begin{equation}
		\label{eq:hmb}
		\E|\chi|^p\le C_p.
	\end{equation}
\end{assumption}

Temporarily assuming that the spectrum $\{ \sigma_i\}_{i=1}^n$ of $X$ is simple,  $X$ can be diagonalized with left eigenvectors $\{\mathbf{l}_i\}_{i=1}^n$ and right eigenvectors $\{\mathbf{r}_i\}_{i=1}^n$. For example, if $\chi$ 
is a continuous random variable, then the spectrum is simple  with probability one.
We normalize the left and right  eigenvectors so that they form a bi-orthogonal basis $\<\overline{\mathbf{l}}_i, \mathbf{r}_j \>=\delta_{ij}$,  $\forall i,j=1, \ldots , n.$
Due to the non-Hermiticity of $X$,  $\{\mathbf{l}_i\}_{i=1}^n$ and $\{\mathbf{r}_i\}_{i=1}^n$ typically differ and they are not orthogonal. To quantify their relation we consider the  scale invariant  quantity, called  \emph{overlap},
\begin{align}
\label{eq:defover}
	\mathcal{O}_{ij}=\<\mathbf{l}_i, \mathbf{l}_j \>\<\mathbf{r}_i, \mathbf{r}_j \>.
\end{align}

While diagonal overlaps, i.e. $\mathcal{O}_{ii} = \| \mathbf{l}_i\|^2  \| \mathbf{r}_i\|^2 $, have been understood fairly well by now \cite{Dubach, CEHS23, Fyo18, HC23, Osmevectors}, much less is know for general $\mathcal{O}_{ij}$. Previous to the current work, the only results for $\mathcal{O}_{ij}$, with $i\ne j$, were computation of their expectation \cite{CM98} and variance \cite{Dubach} (see also \cite{Zeitouni} for the distributional limit of a related quantity) in the special case of the complex Ginibre ensemble, i.e. when $\chi$ is a standard complex Gaussian random variable in Assumption~\ref{ass:mainass}. Our main result is a high probability upper bound for the normalized version of
 $\mathcal{O}_{ij}$ 
for general i.i.d. matrices, which holds uniformly in the spectrum:
\begin{theorem}
\label{thm:maint}
Assume that $X$ is a real or complex i.i.d. matrix satisfying Assumption~\ref{ass:mainass} and that, for simplicity, its spectrum is simple (see Remark~\ref{rem:nonsimpspec}). Let $\mathcal{O}_{ij}$ be the overlaps defined in \eqref{eq:defover}. Then, with very high probability, for any small $\xi>0$, we have
\begin{equation}
\label{eq:ocorn}
\sup_{i,j\in [n]}\big(n|\sigma_i-\sigma_j|^2+1\big)\left[\frac{\big|\langle \mathbf{r}_i, \mathbf{r}_j\rangle\big|^2}{\lVert \mathbf{r}_i\rVert^2\lVert \mathbf{r}_j\rVert^2}+\frac{\big|\langle \mathbf{l}_i, \mathbf{l}_j\rangle\big|^2}{\lVert \mathbf{l}_i\rVert^2\lVert \mathbf{l}_j\rVert^2}\right]\le  n^\xi,
\end{equation}
In particular, by a Cauchy-Schwarz inequality, this implies
\begin{equation}
\label{eq:oij}
\sup_{i,j \in [n]}\big(n|\sigma_i-\sigma_j|^2+1\big)\frac{|\mathcal{O}_{ij}|}{\sqrt{\mathcal{O}_{ii}\mathcal{O}_{jj}}}\le n^\xi.
\end{equation}
\end{theorem}

\begin{remark}
\label{rem:nonsimpspec}
For simplicity of notation we stated Theorem~\ref{thm:maint} for matrices $X$ with simple spectrum. However, we do not need this assumption in the proof (see the end of Section~\ref{sec:imprmultiG}). In fact, from the same proof it follows that for any left/right eigenvectors $\mathbf{l}, \mathbf{r}$ and $\mathbf{l}', \mathbf{r}'$ with corresponding eigenvalues $\sigma, \sigma'$, respectively, for any small $\xi>0$, we have
\begin{equation}
\left[\frac{\big|\langle \mathbf{r}, \mathbf{r}'\rangle\big|^2}{\lVert \mathbf{r}\rVert^2\lVert \mathbf{r}'\rVert^2}+\frac{\big|\langle \mathbf{l}, \mathbf{l}'\rangle\big|^2}{\lVert \mathbf{l}\rVert^2\lVert \mathbf{l}'\rVert^2}\right]\le  \frac{n^\xi}{n|\sigma-\sigma'|^2+1}.
\end{equation}
This bound holds with very high probability simultaneously for all eigenvectors of $X$.
\end{remark}

The bound \eqref{eq:oij} is optimal (up to the $n^\xi$ error), as it can be seen from the variance computation in \cite{Dubach} in the complex Gaussian case.
The optimality of the stronger bound~\eqref{eq:ocorn} follows from
 \cite{Zeitouni}, where the authors proved that $\sqrt{n} (\sigma_i - \sigma_j)
\langle\mathbf{r}_i, \mathbf{r}_j\rangle/\| \mathbf{r}_i\|\| \mathbf{r}_j\|$ 
 is  uniformly (in $n$) sub-Gaussian for a large class of invariant non-normal matrix ensembles
 and even compute its distributional limit in the complex Ginibre case. 
  While our result gives only an upper bound, it holds  for more general i.i.d. 
  matrices satisfying Assumption~\ref{ass:mainass}. In fact, \eqref{eq:ocorn} is the analog of the upper bound in the display below \cite[Eq. (5)]{Zeitouni}. While these results from \cite{Zeitouni, Dubach} are similar
  to ours, the proof methods are completely different.
  
  \begin{remark}
   In Theorem~\ref{thm:maint} we proved an (almost) optimal bound for the overlaps of matrices $X$ with i.i.d. entries as in Assumption~\ref{ass:mainass}. However, inspecting our proof it is clear that this argument can easily be extended to more general models, i.e. to deformed models $X+A$, where $X$ is an i.i.d. matrix and $A$ is a deterministic deformation belonging to a large class of matrices (see e.g. \cite[Section 2.1]{CEHS23}, \cite[Assumption 2.8]{CCEJ24} for bulk and edge eigenvectors, respectively). More precisely, the arguments of Sections~\ref{sec:zig}--\ref{sec:zag} work verbatim in this more general case as well, once the deterministic bounds in Lemma~\ref{lemma_M} are established.
\end{remark}

In Theorem~\ref{thm:maint} we show that the \emph{normalized overlap} $\mathcal{O}_{ij}/\sqrt{\mathcal{O}_{ii}\mathcal{O}_{jj}}$ can be estimated in high probability sense, in particular its tail probability is very small.
 The same is not true for $\mathcal{O}_{ij}$, because overlaps themselves 
are expected to have a heavy tail, similarly to what it has been shown for the diagonal overlaps $\mathcal{O}_{ii}$ \cite{Dubach,  Fyo18, Osmevectors}.   In the following corollary we present a bound on $\mathcal{O}_{ij}$ which holds only in a first moment sense using the analogous bounds on the diagonal overlaps $\mathcal{O}_{ii}$ given in 
\cite[Theorems 2.7, 2.9]{HC23} to account for the normalisation.

\begin{corollary}\label{thm:main}
Consider an i.i.d. matrix $X$ satisfying Assumption \ref{ass:mainass}. Furthermore, we assume that $\chi$ has a 
bounded density $g$ and  in the complex case $\Re\chi, \Im \chi$ are independent. 
Fix any small constants $\tau, \epsilon,\xi>0$  with $\xi\le \epsilon/10$
and any large constant $D>0$. Set $\mathcal{D}^{\xi}_{z}\subset \C$ to be the disk centered at $z\in \C$ with radius $n^{-1/2+\xi}$.  Then for any $z_i\in \C ~(i=1,2)$ with $|z_i|\leq 1+n^{-1/2+\tau}$ and $|z_1-z_2| \geq n^{-1/2+\epsilon}$, there is an event $\Xi$ with very high probability, i.e. $\P(\Xi)\ge 1-n^{-D}$,
  such that
\begin{align}\label{Oij}
 \E \Big(\sum_{\sigma_i\in \mathcal{D}^\xi_{z_1},\sigma_j\in \mathcal{D}^{\xi}_{z_2}} |\mathcal{O}_{ij}|\cdot {\bf 1}_\Xi \Big) &\lesssim \frac{ n^{8\xi}}{|z_1-z_2|^2}, \qquad \mbox{[Complex case]} \\
\label{OijR}
\E \Big(\sum_{\sigma_i\in \mathcal{D}^\xi_{z_1},\sigma_j\in \mathcal{D}^{\xi}_{z_2}} |\mathcal{O}_{ij}|\cdot {\bf 1}_\Xi \Big) &\lesssim \frac{ n^{8\xi}}{|z_1-z_2|^2} \cdot \frac{1}{|\Im z_1|\wedge |\Im z_2|}, \qquad \mbox{[Real case]}
\end{align}

Furthermore, in the complex case 
in the bulk regime, i.e if $|z_1|, |z_2|\le 1-\tau$,
 the event $\Xi$ is trivial, i.e. \eqref{Oij} holds without
${\bf 1}_\Xi$.  In the real case we have a special bound for the real eigenvalues $\sigma_i, \sigma_j$ 
near $z_1, z_2$:
\begin{equation}\label{OijonR}
   \E \Big(\sum_{\substack{\sigma_i\in \mathcal{D}^\xi_{z_1}\cap \R\\\sigma_j\in \mathcal{D}^{\xi}_{z_2}\cap \R}} |\mathcal{O}_{ij}|^{1/2} 
 \Big) \lesssim \frac{ n^{4\xi}}{|z_1-z_2|}.
\end{equation}
\end{corollary}

\begin{remark}\label{rem:locallaw}
	From the local circular laws (see \cite[Theorem 2.2]{BYY14_bulk},  \cite[Theorem 1.3]{BYY14}, \cite[Theorem 1.2]{Yin14}, and \cite[Theorem 2.3]{AEK19}), we know
	 that with very high probability there are eigenvalues close to $z_1$ and $z_2$. More precisely, fix any small $\tau,\xi>0$ with $\tau<\xi/10$, then for any $|z_i|\leq 1+n^{-1/2+\tau}$, there exist eigenvalues $\sigma_i \in \mathcal{D}^{\xi}_{z_1}$ and $\sigma_j \in \mathcal{D}^{\xi}_{z_2}$ with very high probability. Hence, the 
	summations in (\ref{Oij})--(\ref{OijR}) are not empty, in fact 
 they typically contain roughly $n^{2\xi}\times n^{2\xi}$ eigenvalue pairs.
	Our result holds for the corresponding overlap for any of these pairs.
	\end{remark}

We conclude this section explaining  another  interesting result one can obtain using the two--resolvent local law in \eqref{avlaw12} and using very similar techniques to the ones used to prove \eqref{avlaw12} in Section~\ref{sec:zig}. We will not present  the complete proof for brevity, but we indicate the key points and what steps need more work.

\bigskip

\textbf{Mesoscopic Central Limit Theorem uniformly in the spectrum:}
Fix a reference point $|z_0|< 1$, an exponent $a\in [0,1/2)$ and a  smooth and compactly supported test function
$f:\C\to \C$. 
Then the (centered) linear statistics of the eigenvalues $\sigma_i$ around $z_0$ at scale $n^{-a}$,
\[
L_n(f):=\sum_{i=1}^n f_{z_0,a}(\sigma_i)-\E \sum_{i=1}^n f_{z_0,a}(\sigma_i), \qquad f_{z_0,a}(z):=f\big(n^a(z-z_0)\big),
\]
satisfy a Central Limit Theorem (CLT) \cite{complex_CLT,mesoCLT, real_CLT}.  In particular,  
while in the macroscopic case~\cite{complex_CLT, real_CLT}, i.e. $a=0$,
 $f_{z_0,a}$ was allowed to be supported anywhere in the complex plane, in the mesoscopic case~\cite{mesoCLT}, i.e. $a\in (0,1/2)$, $f_{z_0,a}$ had to be supported in the bulk of the spectrum. Thus the mesoscopic CLT close to the edge
 (when $a>0$ and $|z_0|=1$) was not covered in 
 \cite{complex_CLT, mesoCLT,real_CLT}.
   Inspecting the proofs in \cite{complex_CLT, mesoCLT,real_CLT}, it is clear that their main inputs are averaged local laws for products of two resolvents (see \cite[Theorem 3.3]{mesoCLT}) as in \eqref{avlaw12}  and singular vector overlaps bound (see \cite[Theorem 3.1]{mesoCLT}) as in \eqref{overlap}, below. We claim that using these bound as an inputs, after proving a Wick theorem exactly as in \cite[Section 6]{complex_CLT} and replacing the Dyson Brownian motion analysis \cite[Section 7]{complex_CLT} with the one presented in \cite[Section 7]{SpecRadius}, one can routinely prove a mesoscopic CLT for test function supported also close to the edge. More precisely, for reference points  $|z_0|=1$ one would get that $L_n(f)$, for $a\in (0,1/2)$, converges to a centred Gaussian random variable $L(f)$ with the following variance (for simplicity we consider only the complex case)
\begin{equation}
\E\big|L(f)\big|^2=\frac{1}{4\pi}\int_\mathbb{D} \big|\nabla f_{z_0,a}(z)\big|^2\, \dd^2 z+\frac{1}{2}\sum_{k\in\mathbb{Z}}|k|\big|\widehat{f_{z_0,a}}(k)\big|^2,
\end{equation}
where $\widehat{f_{z_0,a}}$ denotes the Fourier transform of the restriction of $f_{z_0,a}$ to the unit circle, and $\mathbb{D}$ the open unit disk. In particular, one can see that compared to \cite[Theorem 2.1]{mesoCLT} one gets an additional contribution from eigenvalues close to the edge, as it was already
 observed in the macroscopic case in \cite[Theorem 2.2]{complex_CLT}. Similarly, one can extend the time--dependent version 
 of the CLT \cite[Theorem 2.3]{BCH24} down to the optimal mesoscopic scale at the edge.

\section{Proof strategy: Multi-resolvent local laws}

 In this section we recall some preliminary inputs and present the main new technical results to obtain the overlap bound in Theorem~\ref{thm:maint}. More precisely, in Section~\ref{sec:singleG} we recall the single resolvent local law and the rigidity estimate for the eigenvalues. Then, in Section~\ref{sec:multiG} we present local laws for products of two resolvents. Finally, in Section~\ref{sec:imprmultiG} we show that the in the regime when the limiting density of states is small we can improve the averaged local law from Section~\ref{sec:multiG}, and use this result to prove Theorem~\ref{thm:maint} and Corollary~\ref{thm:main}. \nc

\subsection{Local laws for the resolvent}
\label{sec:singleG}
Note that $z\in \C$ lies in the spectrum of $X$ if only and if the least singular value of $X-z$ is zero. Hence we link the eigenvalues (resp. eigenvectors) of $X$ with the singular values (resp. singular vectors) of matrix $X-z$. 
As already mentioned in the Introduction, let $(\lambda^z_{i})_{i=1}^n$ denote the singular values of $X-z$ labelled in a non-decreasing order, and  $\{\mathbf{u}_i^z\}_{i=1}^n, \{\mathbf{v}_i^z\}_{i=1}^n$ denote the corresponding left and right singular vectors of $X-z$ normalised such that $\|\uu^z_i\|^2=\|\vv^z_i\|^2=1/2$. In the same spirit of \cite{Girko1984}, we consider the Hermitisation of $X$, \ie
\begin{align}\label{def_H}
	H^{z}:=\begin{pmatrix}
		0  &  X-z  \\
		X^*-\overline{z}   & 0
	\end{pmatrix} \in \C^{2n\times 2n},\qquad z \in \C.
\end{align} 
 Simple algebra (also called  {\it chiral symmetry}) shows that $H^{z}$ has a symmetric spectrum given by $(\lambda^z_{\pm i})_{i=1}^n= (\pm \lambda^z_{i})_{i=1}^n$, with the corresponding (orthonormal) eigenvectors given by $\mathbf{w}^z_{\pm i}=
((\mathbf{u}^z_i)^*,\pm (\mathbf{v}^z_i)^*)^* \in \C^{2n}$. Note that $\mathbf{u}^z_i, \mathbf{v}^z_i$
are the left and right singular vectors of $X-z$.
 Eigenvalues of $H^z$ in general have no direct
  relation with the eigenvalues of $X$, however, from the trivial relation~\eqref{eq:evevrel} it is clear that
eigenvectors of $X$ with eigenvalue $z$ and the singular vectors of $H^z$ with singular value zero coincide.

Let $G^z$ be the resolvent of the Hermitisation of $X-z$, \ie 
\begin{align}\label{def_G}
	G^{z}(w):=(H^{z}-w)^{-1}, 
	\qquad H^{z}=
	\begin{pmatrix}
		0  &  X-z  \\
		X^*-\overline{z}   & 0
	\end{pmatrix},
\qquad w \in \C\setminus \R, \quad z \in \C.
\end{align}
Define the deterministic approximation of $G^z(w)$ as an $(2n)\times (2n)$ block constant matrix given by
\begin{align}\label{Mmatrix}
	M^{z}(w):=\begin{pmatrix}
		m^{z}(w) I_n  &  -zu^z(w)I_n  \\
		-\overline{z}u^z(w)I_n   & m^{z}(w)I_n
	\end{pmatrix},\quad \mbox{with} \qquad  u^z(w):=\frac{{m}^z(w)}{w+ {m}^z(w)},
\end{align}
where ${m}^z(w)$ is the unique solution of the scalar cubic equation with side condition
\begin{align}
	\label{m_function}
	-\frac{1}{{m}^z(w)}=w+{m}^z(w)-\frac{|z|^2}{w+{m}^z(w)}, \qquad\quad \mathrm{Im}[m^z(w)]\mathrm{Im}w>0.
\end{align}
In fact, $M^z(w)$ is the unique solution of  the {\it Matrix Dyson equation} (see e.g. \cite{MDE})
\begin{equation}\label{MDE}
	- [M^z(w)]^{-1} = w + Z+\mathcal{S}[ M^z(w)], \qquad Z:=\left(\begin{matrix}
		0 & z \\
		\overline{z} & 0
	\end{matrix}\right),
\end{equation}
satisfying the side condition $\Im M^z(w) \cdot \Im w >0$, where
\begin{equation}\label{cov}
	\mathcal{S}[\cdot]:=\langle \cdot E_+\rangle E_+-\langle \cdot E_-\rangle E_- = 2\langle \cdot E_1\rangle E_2+2\langle \cdot E_2\rangle E_1, \qquad \mathcal{S}: \C^{2n\times 2n}
	\to \C^{2n\times 2n},
\end{equation}
is the {\it covariance operator} of the Hermitisation of $X$. Additionally, we define the {\it self-consistent eigenvalue density}  
of $H^z$ by (see e.g. \cite{AEK18})
\begin{align}\label{rho_0}
	\rho^z(E):=\lim_{\eta\to 0^+}\frac{1}{\pi} \Im \<M^z(E+\ii \eta)\> , \quad E\in \R.
\end{align}
As summarised in \cite[Section 3.1]{SpecRadius} (see references therein), the density $\rho^z(E)$ has the following properties:
\begin{enumerate}
	\item[1)] if $|z|<1$, then $\rho^z(E)$ has a local
	minimum at $E=0$ of height $\rho^z(0) \sim (1-|z|^2)^{1/2}$; 
	
	\item[2)] if $|z|=1$, then $\rho^z(E)$ has an exact cubic cusp singularity at $E=0$; 
	
	\item[3)]  if $|z|>1$, then there exists a small gap $[-\Delta/2,\Delta/2]$ in the support of the symmetric density $\rho^z(E)$, with $\Delta \sim (|z|^2-1)^{3/2}$.
\end{enumerate}
Moreover, extending to the complex plane, if $|z|\le 1$, then for any $w=E+\ii\eta$, we have
\begin{align}\label{rho_E_bulk}
	 \rho^z(w):=\pi^{-1} |\Im \<M^z(w)\>|\sim (1-|z|^2)^{1/2} +(|E|+\eta)^{1/3}, \qquad |E|\leq c,\quad 0\leq \eta\leq 1,
\end{align}
for some small $c>0$, 
while in the complementary regime
$|z|> 1$, for any $w=\frac{\Delta}{2}+\kappa+\ii \eta$, we have 
\begin{align}\label{rho_E}
	\rho^z(w) \sim \begin{cases}
		(|\kappa|+\eta)^{1/2} (\Delta+|\kappa|+\eta)^{-1/6}, &\quad \kappa \in [0,c]\\ 
		\frac{\eta}{(\Delta+|\kappa|+\eta)^{1/6}(|\kappa|+\eta)^{1/2}} , &\quad  \kappa \in [-\Delta/2,0]
	\end{cases},
	\quad 0\leq \eta\leq 1.  
\end{align}
In this paper we focus on the imaginary axis, \ie $w=\ii \eta$, and we may often drop the superscript $z$ and the argument $w=\ii \eta$ for brevity. For instance, we may write $m=m^{z}(\ii \eta)$ for brevity. In particular,
\begin{align} \label{rho}
	\rho:=\frac{|\Im m|}{\pi} \sim \begin{cases}
		\eta^{1/3}+|1-|z|^2|^{1/2} , &\qquad |z| \leq 1\\
		\frac{\eta}{|1-|z|^2|+\eta^{2/3}}, &\qquad |z| > 1
	\end{cases},
	\qquad  0\leq \eta\leq 1.
\end{align}

Now we state the local law for the resolvent; see \cite[Theorem 3.4]{BYY14_bulk} for $|z|\leq 1-\tau$ and \cite[Theorem 3.1]{SpecRadius} for $||z|-1|\leq \tau$ (see also \cite{BYY14}), where $\tau>0$ is a small constant.
\begin{theorem}		
	\label{local_thm} 
	There is sufficiently small constants $\tau,\tau'>0$ such that for any $z\in\C$ with $|z|\leq 1+\tau$ and 
	for any $w=E+\ii\eta$ with $|E|\le \tau'$ we have
	\begin{align}
		\big| \langle \mathbf{x},  (G^{z}(w)-M^{z}(w)) \mathbf{y} \rangle\big| \prec \; & \|\mathbf{x}\| \|\mathbf{y}\| \left( \sqrt{\frac{\rho^z(w)}{n|\eta|}}+\frac{1}{n|\eta|}\right),\label{entrywise}\\
		\big|\big\<\big( G^{z}(w)-M^{z}(w)\big)A\big\>\big| \prec \; & \frac{\|A\|}{n|\eta|},
		\label{average}
	\end{align}
	uniformly in spectral parameter $|\eta|>0$, deterministic vectors $\mathbf{x}, \mathbf{y}\in\C^{2n}$, and matrices $A\in \C^{2n \times 2n}$.
\end{theorem}

 Using standard arguments in e.g. \cite[Corollary 1.10-1.11]{AEK17},  we obtain the following rigidity estimates on the eigenvalues from Theorem \ref{local_thm}; see also \cite[ Corollary 3.2]{SpecRadius} for $|z|>1$ only.

\begin{corollary}[
		]\label{cor:rigidity}
	Fix a small constant $\tau>0$ and pick $z$ such that $ |z|\le 1+\tau $. Then 
	there exists a small constant $c>0$ such that
	\begin{equation}\label{rigidity}
		|\lambda_i^z-\gamma_i^z|\prec \max\left\{\frac{1}{|i|^{1/4}n^{3/4}+\sqrt{(1-|z|^2)_+}n}, \frac{\Delta^{1/9}}{n^{2/3}|i|^{1/3}}\right\}, \qquad |i|\leq cn
	\end{equation}
	with $\Delta$ denoting the size of the gap around zero
	in the support of $\rho^z$ (in particular $\Delta=0$ if $|z|\leq 1$), where $\gamma_i^z$ is the $i$-th quantile of $\rho^z$ given by
	\begin{equation}\label{quantile}
		\int_0^{\gamma_i^z}\rho^z(x)\,\dd x=\frac{i}{2n}, \qquad \gamma_{-i}^z:=-\gamma_i^z, \qquad i\in [n].
	\end{equation}	
	 In addition, there exists a small constant $c'>0$ such that 
	\begin{equation}\label{rigidity3}
		\Big| \#\{j: E_1 \leq \lambda^z_j \leq E_2\}-2n \int_{E_1}^{E_2} \rho^z(x) \dd x \Big| \prec 1, \qquad |E_i|\leq c', \quad i=1,2.
	\end{equation}

\end{corollary}

\subsection{Local laws for multiple resolvents}
\label{sec:multiG}

 The main technical result of this manuscript are  local laws for the product of two Green functions 
 $$\ga(\ii\eta_1) A \gb(\ii \eta_2), \qquad A\in \C^{2n \times 2n}, \quad G^{z_i}(\ii \eta_i) \in \C^{2n \times 2n},~i=1,2,$$ 
\nc
 with different parameters $z_i \in \C$ and $\eta_i \in \R \setminus \{0\}~(i=1,2)$, which yield the overlap bounds of the eigenvectors stated in Theorem \ref{thm:maint}. Before we state the local laws, we first introduce the deterministic approximation matrices of $\ga(\ii\eta_1) A \gb(\ii \eta_2)$, \ie
\begin{align}
	\label{eq:defM12}
	M_{12}^A:=M^A_{12}(\ii \eta_1,z_1,\ii \eta_2,z_2)= \mathcal{B}^{-1}_{12}\Big[M^{z_1}(\ii \eta_1) A M^{z_2}(\ii \eta_2)\Big],
\end{align}
where the two--body stability operator  $\mathcal{B}_{12}:\C^{2n\times 2n}\to\C^{2n\times 2n}$ is given by
\begin{equation}
	\label{eq:defstabop}
	\mathcal{B}_{12}:=\mathcal{B}_{12}(z_1,\omega_1,z_2,\omega_2)=1-M^{z_1}(\omega_1)\mathcal{S}[\cdot] M^{z_2}(\omega_2),
\end{equation}
 with the covariance operator $\mathcal{S}[\cdot]$ defined in (\ref{cov}); see also (\ref{eq:defm12}) later for the definition of $\mathcal{B}_{ij}$ with general $i,j \in\N$. Using the identity $\Im G=\frac{1}{2\ii}(G-G^*)$ twice and the definition of $M^{A}_{12}$ in (\ref{eq:defM12}), we further define the deterministic approximating matrix of $\Im G^{z_1}(\ii \eta_1) A\Im G^{z_2}(\ii \eta_2)$ as a fourfold linear combination of $M^A_{12}(\pm \ii\eta_1, z_1, \pm \ii \eta_2, z_2)$ as defined in (\ref{eq:defM12}), \ie
\begin{align}\label{Im_M_def}
	\wh M_{12}^{A}:=&\wh M_{12}^{A}(\ii \eta_1,z_1,\ii \eta_2,z_2)\nonumber\\
	=&\frac{1}{4} \big( M^A_{12}(\ii\eta_1, z_1, -\ii \eta_2, z_2)+M^A_{12}(-\ii\eta_1, z_1, \ii \eta_2, z_2)-M^A_{12}(\ii\eta_1, z_1, \ii \eta_2, z_2)-M^A_{12}(-\ii\eta_1, z_1, -\ii \eta_2, z_2) \big).
\end{align}
\nc
To approximate the product $\ga(\ii\eta_1) A_1 \gb(\ii \eta_2) A_2 \ga(\ii\eta_1)$, we also define
\begin{align}\label{eq:defM121}
	M_{121}^{A_1,A_2}:=\mathcal{B}_{11}^{-1}\left[M^{z_1}(\ii \eta_1) A_1M_{21}^{A_2}+M^{z_1}(\ii \eta_1)\mathcal{S}\big[M_{12}^{A_1}\big]M_{21}^{A_2}\right],
\end{align}
with $M_{21}^A=M^A_{21}(\ii \eta_2,z_2,\ii \eta_1,z_1)$ defined as in (\ref{eq:defM12}). 

\bigskip

The proof of the following lemma is presented in Appendix \ref{app:lemma_M}.
\begin{lemma}\label{lemma_M}
For any bounded deterministic matrices $A_1,A_2 \in \C^{2n \times 2n}$, we have
	\begin{align}\label{M_bound}
		\|M^{A_1}_{12}(\ii \eta_1,z_1,\ii \eta_2,z_2)\| \lesssim \frac{1}{\gamma},
	\end{align}
\begin{align}\label{Im_M_bound}
	\big|\<\wh M^{A_1}_{12} (\ii \eta_1,z_1,\ii \eta_2,z_2) A_2 \>\big| \lesssim \frac{\rho_1\rho_2}{\wh \gamma},
\end{align}
\begin{align}\label{M_3_bound}
	\|M^{A_1,A_2}_{121}(\ii \eta_1, A_1,\ii \eta_2, A_2,\ii \eta_1)\|  \lesssim \frac{1}{\eta_*\gamma},
\end{align}
uniformly for any $|z_i|<10$ and  $0<|\eta_i|<10$, where $\eta_*:=\min_{i=1,2}|\eta_i|$ and $\rho^*:=\max_{i=1,2}|\rho_i|$ with $\rho_i:=\pi^{-1}|\Im m^{z_i}(\ii\eta_i)|$\nc, and the following parameters:
\begin{align}\label{parameter}
	\gamma=\gamma(\ii \eta_1,z_1,\ii \eta_2,z_2):=\frac{|z_1-z_2|^2+\rho_1|\eta_1|+\rho_2|\eta_2|+\left(\frac{\eta_1}{\rho_1}\right)^2+\left(\frac{\eta_2}{\rho_2}\right)^2}{|z_1-z_2|+\rho_1^2+\rho_2^2+\frac{\eta_1}{\rho_1}+\frac{\eta_2}{\rho_2}},
\end{align}
\begin{align}\label{parameter2}
	\wh\gamma=\wh\gamma(\ii \eta_1,z_1,\ii \eta_2,z_2):=|z_1-z_2|^2+\rho_1|\eta_1|+\rho_2|\eta_2|+\left(\frac{\eta_1}{\rho_1}\right)^2+\left(\frac{\eta_2}{\rho_2}\right)^2.
\end{align}

\end{lemma}

\begin{theorem}\label{thm:2G}
	 Fix small constants $\epsilon,\tau>0$. Then, for any deterministic matrices $\|A_i\| \lesssim 1$ and vectors $\|\xx\|, \|\yy\| \lesssim 1$, the following isotropic local law: 
	\begin{equation}\label{enlaw12}
		\Big| \big\langle \xx, \big(G^{z_1}(\ii \eta_1)A_1G^{z_2}(\ii \eta_2)-M_{12}^{A_1} \big) \yy \big \rangle\Big| \prec 
		\frac{1}{\sqrt{n \gamma}\eta_*},
	\end{equation}
 and the averaged local law: 
	\begin{equation}\label{avlaw12}
	\Big|\big\langle \big(G^{z_1}(\ii \eta_1) A_1 G^{z_2}(\ii \eta_2)-M_{12}^{A_1}\big)A_2 \big\rangle\Big|\prec
	\frac{1}{\sqrt{n\ell}\gamma }\wedge \frac{1}{n \eta_*^2},
\end{equation}
hold uniformly for any $ |z_i|\leq 1+\tau$ and $\ell \geq n^{-1+\epsilon}$ with $\eta_*:=\min_{i=1,2}|\eta_i|$, $\ell:=\min_{i=1,2}\rho_i|\eta_i|$, and the deterministic matrix $M^{A}_{12}=M^A_{12}(\ii\eta_1,z_1,\ii \eta_2,z_2)$ is defined in (\ref{eq:defM12}).
\end{theorem}

\subsection{Improved local law at the edge}
\label{sec:imprmultiG}

We state the following improved local law for $\Im \ga A \Im \gb$ than (\ref{avlaw12}), with  $\Im G=\frac{1}{2\ii}(G-G^*)$, up to the edge regime  $|z_i| \leq 1+n^{-1/2+\tau}$.

\begin{theorem}\label{thm:2G_edge}
Fix small constants $\epsilon, \tau>0$ with $0<\tau<\epsilon/10$. Then, for any deterministic matrices $\|A_i\| \lesssim 1$, the following local law  
\begin{equation}\label{avlaw12_im}
	\Big|\langle \big(\Im G^{z_1}(\ii \eta_1) A_1 \Im G^{z_2}(\ii \eta_2)-\wh M_{12}^{A_1}\big)A_2\rangle\Big|\prec
	\frac{1}{\sqrt{n\ell}}\frac{\rho_{1}\rho_{2}}{\widehat{\gamma}},
\end{equation}
holds uniformly for any $z_i \in \C$ and $\eta_i \in \R~(i=1,2)$ such that $|z_i|\le 1+n^{-1/2+\tau}$ and $\ell =\min_{i=1,2}\rho_i|\eta_i| \geq n^{-1+\epsilon}$,
with $\wh \gamma$ given by (\ref{parameter2}), and the deterministic matrix $\wh{M}_{12}^A=\wh{M}_{12}^A(\ii\eta_1,z_1,\ii \eta_2,z_2)$ defined in (\ref{Im_M_def}).
\end{theorem}

By a standard argument, we obtain the following optimal bound for the overlaps between singular vectors of $X-z_i$ for different $z_1,z_2 \in \C$. The proof is similar to \cite[Theorem 3.1]{mesoCLT} for $|z|\leq 1-\tau$ and \cite[Corollary 3.4]{gumbel} for $1 \leq |z_i|\leq 1+ n^{-1/2+\tau}$. The only difference is to replace the non-optimal local law for $\Im G^{z_1} \Im G^{z_2} $ (\eg \cite[Theorem 3.3]{gumbel}) with the optimal version stated in Theorem \ref{thm:2G_edge}.

\begin{corollary}\label{coro_edge}
	Fix a small constant $\tau>0$. Then there exists a small $0<\omega_c<\tau/10$ such that, for any $z_i \in \C~(i=1,2)$ with 
	 $|z_i|\le 1+n^{-1/2+\tau}$, the following holds with very high probability,
	\begin{equation}\label{overlap}
		|\<\uu^{z_1}_i,\uu_j^{z_2}\>|^2+|\<\vv^{z_1}_i,\vv_j^{z_2}\>|^2 \lesssim  \frac{n^{-1+2\tau}}{|z_1-z_2|^2+n^{-1}}, \qquad  |i|, |j| \leq n^{\omega_c}.
	\end{equation}
\end{corollary}

We conclude this section by showing that the bound \eqref{avlaw12_im} readily implies those in Theorem~\ref{thm:maint} and Corollary~\ref{thm:main}.  Actually, Corollary~\ref{coro_edge} is obtained as a by-product in the proof of Theorem~\ref{thm:maint} (see (\ref{eq:desboigrrll}) below).

\begin{proof}[Proof of Theorem~\ref{thm:maint}]

Let $\eta_{\mathrm{f}}=\eta_{\mathrm{f}}(z)$ denote the local fluctuation scale of the eigenvalues of $H^z$ close to zero, 
which, for  $|z|\le 1$ 
is implicitly defined by
\[
\int_{-\eta_{\mathrm{f}}}^{\eta_{\mathrm{f}}} \rho^z(x)\,\dd x=\frac{1}{n}.
\]
  From (\ref{rho}), note that $\eta_{\mathrm{f}}\sim n^{-1}$ when $|z|<1$ and $\eta_{\mathrm{f}}\sim n^{-3/4}$ when $|z|=1$.  For $|z|>1$ the density $\rho^z$ develops a gap $\Delta^z$ around zero, and one can see that in this case $\eta_{\mathrm{f}}\sim |\Delta^z|^{1/9}/n^{2/3}$ and $|\Delta^z|\sim (|z|^2-1)^{3/2}$ (see e.g. \cite[Appendix A]{SpecRadius} and references therein).  Fix a small $\xi>0$, and define $\eta_i:=n^\xi \eta_{\mathrm{f}}(z_i)$, for $z_i\in\C$. Then, by \eqref{avlaw12_im} and (\ref{Im_M_bound}), 
we obtain
\begin{equation}
\label{eq:bsup12}
\sup_{|z_1|,|z_2|\le 1+n^{-1/2+\tau}} \frac{\widehat{\gamma}}{\rho_1\rho_2}\big|\langle \Im G^{z_1}(\ii \eta_1) \Im G^{z_2}(\ii \eta_2)\rangle\big| \lesssim 1,
\end{equation}
with very high probability. Here $\rho_i=\pi^{-1}|\Im m^{z_i}(\ii\eta_i)|$ and $\widehat{\gamma}=\widehat{\gamma}(\ii\eta_1,z_1,\ii\eta_2,z_2)$. We point out that to obtain the supremum bound in \eqref{eq:bsup12} we used \eqref{avlaw12_im} and a standard grid argument using the Lipschitz continuity of $z\mapsto \Im G^z$.

Next, by the rigidity estimate of the eigenvalues of $H^z$ in \eqref{rigidity}, we have
\begin{equation}
\label{eq:rigsup}
\max_{|j|\leq n^{\omega_c}}\frac{|\lambda^{z}_j|}{\eta_{\mathrm{f}}(z)}\lesssim 1,
\end{equation}
with very high probability, for some small constant $0<\omega_c<\xi/10$.  Note that for any $j\in [n]$, $\lambda^{z}_j$ is the $j$-th non-negative eigenvalue of $H^{z}$ in the non-deceasing order and $\lambda^{z}_{-j}=-\lambda^{z}_j$. 

We notice that by the spectral decomposition of $\Im G= (G-G^*)/(2\ii)$ and the chiral symmetry of $H^z$ in (\ref{def_H}), we have
\begin{equation}
\<\Im G^{z_1}(\ii \eta_1) \Im G^{z_2}(\ii \eta_2)\>=\frac{4}{n}\sum_{j,k=1}^n \frac{\eta_1\eta_2 \big(|\<\uu^{z_1}_j,\uu_k^{z_2}\>|^2+|\<\vv^{z_1}_j,\vv_k^{z_2}\>|^2\big)}{\big((\lambda^{z_1}_j)^2+\eta_1^2\big)((\lambda^{z_2}_k)^2+\eta_2^2\big)},
\end{equation}
 where $(\lambda^{z_i}_j)_{j=1}^n$ are the non-decreasing singular values of  $X-z_{i}$ with $i=1,2$, and  $\{\uu_j^{z_i}\}_{j=1}^n$, $\{\vv_j^{z_i}\}_{j=1}^n$ are the corresponding left and right singular vectors normalized such that $\|\uu_j^{z_i}\|^2=\|\vv_j^{z_i}\|^2=1/2$. 
Then, on the very high probability event on which \eqref{eq:bsup12}--\eqref{eq:rigsup} hold, we obtain (recall $\eta_i:=n^\xi \eta_{\mathrm{f}}(z_i)$)
\begin{equation}
\label{eq:desboigrrll}
\begin{split}
&\sup_{|z_1|,|z_2|\le 1+n^{-1/2+\tau}}(n|z_1-z_2|^2+1)\big[|\<\uu^{z_1}_1,\uu_1^{z_2}\>|^2+|\<\vv^{z_1}_1,\vv_1^{z_2}\>|^2\big] \\
&\qquad \le \sup_{|z_1|,|z_2|\le 1+n^{-1/2+\tau}} (n|z_1-z_2|^2+1) (n\eta_1\eta_2) \<\Im G^{z_1}(\ii \eta_1) \Im G^{z_2}(\ii \eta_2)\> \\
&\qquad \lesssim \sup_{|z_1|,|z_2|\le 1+n^{-1/2+\tau}} \frac{|z_1-z_2|^2+ n^{-1}\nc}{\widehat{\gamma}} (n\eta_1\rho_1)(n\eta_2\rho_2) \lesssim  n^{4\xi},
\end{split}
\end{equation}
with very high probability, where in the last inequality we used \eqref{parameter2} to show that $(|z_1-z_2|^2+ n^{-1})/\widehat{\gamma}\lesssim 1$ and we used that $n\eta_i\rho_i\lesssim n^{2\xi}$ by the definition of $\eta_{\mathrm{f}}$ above \eqref{eq:bsup12} (note that $\eta_{\mathrm{f}}$ has the property that $n\eta_{\mathrm{f}}\rho^z(\ii\eta_{\mathrm{f}})\sim 1$).

 Recall the basic relation \eqref{eq:evevrel} that $z\in \mathrm{Spec}(X)$ if and only if the least singular value $\lambda^{z}_1$ is zero. Moreover, if $z\in \mathrm{Spec}(X)$, then the left and right eigenvector of $X$ associated to the eigenvalue $z$ are exactly the left and right singluar vector of $X-z$. This implies that
\begin{align}
	&\sup_{i,j\in [n]}\big(n|\sigma_i-\sigma_j|^2+1\big)\left[\frac{\big|\langle \mathbf{r}_i, \mathbf{r}_j\rangle\big|^2}{\lVert \mathbf{r}_i\rVert^2\lVert \mathbf{r}_j\rVert^2}+\frac{\big|\langle \mathbf{l}_i, \mathbf{l}_j\rangle\big|^2}{\lVert \mathbf{l}_i\rVert^2\lVert \mathbf{l}_j\rVert^2}\right]\nonumber\\
	 =& 4 \sup_{z_1,z_2\in \mathrm{Spec}(X)}(n|z_1-z_2|^2+1)\big[|\<\uu^{z_1}_1,\uu_1^{z_2}\>|^2+|\<\vv^{z_1}_1,\vv_1^{z_2}\>|^2\big],
\end{align}
where $\uu^{z_i}_1$ and $\vv^{z_i}_1~(i=1,2)$ are the left and right singular vectors associated to the least singular value of $X-z_i$, normalized such that $\|\uu_1^{z_i}\|^2=\|\vv_1^{z_i}\|^2=1/2$. Note that with very high probability the spectrum of $X$ lies inside the disk $|z|\le 1+n^{-1/2+\tau}$~(see \cite[Theorem 2.1]{AEK19}). Hence, with very high probability, 
\begin{align}\label{inequality}
	&\sup_{i,j\in [n]}\big(n|\sigma_i-\sigma_j|^2+1\big)\left[\frac{\big|\langle \mathbf{r}_i, \mathbf{r}_j\rangle\big|^2}{\lVert \mathbf{r}_i\rVert^2\lVert \mathbf{r}_j\rVert^2}+\frac{\big|\langle \mathbf{l}_i, \mathbf{l}_j\rangle\big|^2}{\lVert \mathbf{l}_i\rVert^2\lVert \mathbf{l}_j\rVert^2}\right]\nonumber\\
	\lesssim & \sup_{|z_1|,|z_2|\le 1+n^{-1/2+\tau}}(n|z_1-z_2|^2+1)\big[|\<\uu^{z_1}_1,\uu_1^{z_2}\>|^2+|\<\vv^{z_1}_1,\vv_1^{z_2}\>|^2\big].
\end{align}
This, together with (\ref{eq:desboigrrll}), immediately implies \eqref{eq:ocorn}--\eqref{eq:oij} and concludes the proof.
 	\end{proof}

\begin{proof}[Proof of Corollary~\ref{thm:main}]
 Fix any small $\xi>0$. From (\ref{eq:oij}), we know that for any $i,j \in [N]$, with very high probability,
\begin{align}\label{Oinequality}
	|\mathcal{O}_{ij}| \leq & \frac{n^{\xi}}{n|\sigma_i-\sigma_j|^2+1} \sqrt{\mathcal{O}_{ii}\mathcal{O}_{jj}} \nonumber\\
	\lesssim & \frac{n^{\xi}}{n|z_1-z_2|^2+1} \sqrt{\mathcal{O}_{ii}\mathcal{O}_{jj}},
\end{align}
where the last step follows if we assume that the corresponding eigenvalues $\sigma_i \in \mathcal{D}^\xi_{z_1}$ and $\sigma_j \in \mathcal{D}^\xi_{z_2}$, with $|z_1-z_2|\geq n^{-1/2+\epsilon}$ and $\xi<\epsilon/10$. Moreover, from 
 the local law mentioned in Remark~\ref{rem:locallaw},  
for any $z\in \C$, with very high probability,
\begin{align}\label{ninequality}
	\#\{\sigma_i \in \mathrm{Spec}(X):\sigma_i \in \mathcal{D}^\xi_{z}\} \leq n^{3\xi}.
\end{align}
We next recall various bounds on the expectation of the diagonal overlaps from \cite[Theorems 2.7, 2.9]{HC23}. For example, in  the complex case, using \cite[Eq (2.14)]{HC23}, there exists an event $\Xi$ with very high probability, i.e. $\P(\Xi)\ge 1-n^{-D}$, such that for any $z\in \C$,
\begin{align}\label{Einequality}
	 \E \Big(\sum_{\sigma_i\in \mathcal{D}^\xi_{z}} |\mathcal{O}_{ii}|\cdot {\bf 1}_\Xi \Big) \lesssim n^{1+3\xi}.
\end{align}
Combining the very high probability bound (\ref{Oinequality})-(\ref{ninequality}) with the expectation bound (\ref{Einequality}) and using a simple Cauchy-Schwarz inequality, we obtain the desired result in (\ref{OijR}) for the complex case. The remaining statements of Corollary~\ref{thm:main} follow similarly by combining (\ref{Oinequality})-(\ref{ninequality}) with the corresponding expectation bounds from Eq (2.15), Eq (2.17), and Eq (2.18) of \cite{HC23}, respectively. This finishes the proof. \nc
\end{proof}

\section{Zig flow}\label{sec:zig}

In this section we show that if certain local laws (see \cite[Theorem 5.2]{complex_CLT}) hold for $\eta\sim 1$ (for simplicity in this introductory paragraph we focus on $|z|<1$), then they can be propagated down to $\eta\gg n^{-1}$ using the Ornstein--Uhlenbeck and characteristic flows.  The main results of this section (Propositions~\ref{prop:aveflow}--\ref{prop:isoflow}) are the fundamental inputs to prove Theorem~\ref{thm:2G}. In fact, they show that if the local laws from Theorem~\ref{thm:2G} hold on a global scale ($\eta\sim 1$), then they also hold down to the optimal $\eta$-scale ($\eta\gg 1/n$), at the price of adding a Gaussian component. This Gaussian component will then be removed in Section~\ref{sec:zag} using a Green's Function Comparison Theorem (GFT).  We point out that some parts of this section, like the definition and properties of the characteristics flow and the analysis of the flow in the averaged case, are very close to \cite[Section 5]{gumbel}; however, we decided to present the details here for the convenience of the reader. Consider the Ornstein--Uhlenbeck flow
\begin{equation}
	\label{eq:OUmat}
	\dd X_t=-\frac{1}{2}X_t\dd t+\frac{\dd B_t}{\sqrt{N}}, \qquad\quad X_0=X,
\end{equation}
where $X$ is an i.i.d. matrix satisfying Assumption~\ref{ass:mainass}. Here $B_t$ is an $n\times n$ matrix with entries being real or complex standard i.i.d. Brownian motions, depending on the symmetry class of $X$.

We now define the characteristics flow associated with \eqref{eq:OUmat}. Define the Hermitisation of $X_t$ by
\begin{equation}
	\label{eq:hermflow}
	W_t:=\left(\begin{matrix}
		0 & X_t \\
		X_t^* & 0
	\end{matrix}\right),
\end{equation}
and consider the spectral parameters $\Lambda\in \C^{2n\times 2n}$ defined by
\begin{equation}
	\label{eq:defchar}
	\Lambda:=\left(\begin{matrix}
		\ii\eta & z \\
		\overline{z} & \ii\eta
	\end{matrix}\right), \qquad\quad \eta\in \R, \qquad z\in\C.
\end{equation}
Throughout this section we consider spectral parameters being $2n\times 2n$ block--constant matrices, i.e. we consider $\Lambda$'s consisting of four $n\times n$ blocks each one of them being a constant multiple of the $n$--dimensional identity matrix. The main object of interest within this section is the resolvent $(W_t-\Lambda_t)^{-1}$, with $W_t$ from \eqref{eq:hermflow} and $\Lambda_t$ being the solution of the \emph{characteristic flow}:
\begin{equation}
	\label{eq:charflowmat}
	\partial_t\Lambda_t= -\mathcal{S}[M(\Lambda_t)]-\frac{\Lambda_t}{2},
\end{equation}
with $\mathcal{S}$ being the covariance operator from \eqref{cov} and $M(\Lambda_t):=M^{z_t}(\ii\eta_t)$. Note that if $\Lambda_0$ is constant in each of its four blocks so is $\Lambda_t$ for any $t\ge 0$, hence the equation \eqref{eq:charflowmat} reduces to two scalar equations:
\begin{equation}
	\label{eq:char}
	\partial_t\eta_t=-\Im m^{z_t}(\ii\eta_t)-\frac{\eta_t}{2}, \qquad\quad \partial_t z_t=-\frac{z_t}{2}.
\end{equation}
We point out that there is a one to one correspondence between $(\eta_t,z_t)$ and $\Lambda_t$. For this reason throughout this section we may refer to one or the other interchangeably. The $\eta_t$--characteristics are decreasing in time, for this reason, throughout this section, we only consider the characteristics up to the first time they touch the real axis
\begin{equation}
	T^*(\eta_0,z_0):=\sup\big\{ t\ge 0: \mathrm{sgn}(\eta_t)=\mathrm{sgn}(\eta_0)\big\}.
\end{equation}
Even if not stated explicitly, in the reminder of this section we always consider initial conditions such that $|\eta_0|+|z_0|\lesssim 1$, which is equivalent to $\lVert \Lambda_0\rVert\lesssim 1$. In the following (trivially checkable) lemma we summarize several properties of the characteristics \eqref{eq:char} (cf. \cite[Lemma 5.1]{gumbel}):
\begin{lemma}
	\label{lem:charprop}
	Fix $\Lambda_0$ with $\lVert \Lambda_0\rVert\lesssim 1$, and let $\Lambda_t$ be the solution of \eqref{eq:charflowmat} with initial condition $\Lambda_0$. Define $M_t:=M^{z_t}(\ii\eta_t)$ and $\rho_t:=\pi^{-1}|\Im m_t|$, then for any $0\le s\le t\le T^*(\eta_0,z_0)$ we have
	\begin{equation}
		\label{eq:easyrelchar}
		z_t= e^{-(t-s)/2}z_s, \qquad M_t=e^{(t-s)/2}M_s, \qquad \frac{|\eta_t|}{\rho_t}=e^{s-t}\frac{|\eta_s|}{\rho_s}-\pi(1-e^{s-t}).
	\end{equation}
	Additionally, the maps $t\mapsto |\eta_t|$, $t\mapsto |\eta_t|\rho_t$, and $t\mapsto|\eta_t|/\rho_t$ are strictly decreasing, and for any $\alpha>1$ there is a constant $C_\alpha$ such that
	\begin{equation}
		\label{eq:explint}
		\int_s^t \frac{1}{|\eta_r|^\alpha}\,\dd t\le \frac{C_\alpha}{|\eta_t|^{\alpha-1}\rho_t}, \qquad\quad \int_s^t \frac{\rho_r}{|\eta_r|}\,\dd r=\frac{1}{\pi}\left[\log\left(\frac{\eta_s}{\eta_t}\right)-\frac{t-s}{2}\right].
	\end{equation}
\end{lemma}
Additionally, by standard ODE theory we have (cf. \cite[Lemma 5.2]{gumbel}):
\begin{lemma}
	\label{lem:ODEtheo}
	Fix any $T\le 1$ and pick $\Lambda$ of the form \eqref{eq:defchar} with $|\eta|>0$ and $\lVert\Lambda\rVert\le 1$. Then there exist $\Lambda_0$ with $\lVert \Lambda_0\rVert\lesssim 1$ and $T^*(\eta_0,z_0)> T$, and the solution $\Lambda_t$ of \eqref{eq:charflowmat} with initial condition $\Lambda_0$ is so that $\Lambda_T=\Lambda$. Furthermore, there exists an $N$--independent constant $c_*>0$ such that $\mathrm{dist}(\ii \eta_0,\mathrm{supp}(\rho^{z_0}))\ge c_*T$.
\end{lemma}

In the rest of this section we will consider times $t\le T^*:=\min_i T_i^*(\eta_{i,0},z_{i,0})$. Note that by the first two relations in \eqref{eq:easyrelchar}, for $s,t\le T^*$, it follows
\begin{equation}
	\label{eq:sizez1z2}
	|z_{1,s}-z_{2,s}|\sim |z_{1,t}-z_{2,t}|, \qquad\quad \rho_s\sim \rho_t.
\end{equation} 

We are now ready to state the two main results of this section. We first state the averaged case, whose statement is very similar to \cite[Proposition 5.4]{gumbel} (note the difference in the definition of $\gamma$ in \eqref{eq:defcontrolparam}), and then we state the isotropic case. For deterministic matrices $A_1,A_2\in\C^{2n\times 2n}$, we define the deterministic approximations for the product of two and three resolvents (using the notation $M_i:=M^{z_i}(\ii\eta_i)$)
\begin{align}
	\label{eq:defm12}
	M_{12}^{A_1}:&=\mathcal{B}_{12}^{-1}\big[M_1A_1M_2\big], \qquad\quad \mathcal{B}_{ij}[\cdot]:=1-M_i\mathcal{S}[\cdot]M_j \\
	M_{121}^{A_1,A_2}:&=\mathcal{B}_{11}^{-1}\left[M_1A_1M_{21}^{A_2}+M_1\mathcal{S}\big[M_{12}^{A_1}\big]M_{21}^{A_2}\right],
	\label{eq:defm121}
\end{align}
where $\mathcal{S}:\C^{2n\times 2n}\to \C^{2N\times 2N}$ is the covariance operator defined in (\ref{cov}).
If we replace $M_i$ by $M_{i,t}=M^{z_{i,t}}(\ii\eta_{i,t})$ in \eqref{eq:defm12}--\eqref{eq:defm121}, we then denote the corresponding quantities by $M_{12,t}^{A_1}, M_{121,t}^{A_1,A_2}$. We are now ready to state the main results of this section: 
\begin{proposition}
	\label{prop:aveflow}
	Fix small $n$--independent constants $\tau,\omega_*,\epsilon>0$, and for $i=1,2$ fix spectral parameters
	\begin{equation}
		\label{eq:defincond}
		\Lambda_{i,0}:=\left(\begin{matrix}
			\ii\eta_{i,0} & z_{i,0} \\
			\overline{z_{i,0}} & \ii\eta_{i,0}
		\end{matrix}\right),
	\end{equation}
	with $|z_{i,0}|\le 10$, $|\eta_{i,0}|\le \omega_*$. For $t\le T^*$ let $\Lambda_{i,t}$ be the solution of \eqref{eq:charflowmat} with initial condition $\Lambda_{i,0}$. Let $G_{i,t}:=(W_t-\Lambda_{i,t})^{-1}$, and let $M_{12,t}^A$ be defined in \eqref{eq:defm12}. Define the following control parameters and their time--dependent versions\footnote{ Recall that $\rho_i=\rho^{z_i}(\ii\eta_i)$, and so in the time dependent case we have $\rho_{i,t}=\rho^{z_{i,t}}(\ii\eta_{i,t})$. For example, using these notations, $\ell_t:=\min_{i=1,2}\rho_{i,t}\eta_{i,t}$.}:
	\begin{equation}
		\begin{split}
			\label{eq:defcontrolparam}
			\ell&=\ell(\eta_1,z_1,\eta_2,z_2):=\min_{i=1,2}\rho_i|\eta_i|, \qquad\qquad\qquad\qquad\qquad\qquad\qquad\qquad\qquad \ell_t:=\ell(\eta_{1,t},z_{1,t},\eta_{2,t},z_{2,t}), \\
			\gamma&=\gamma(\eta_1,z_1,\eta_2,z_2):=\frac{|z_1-z_2|^2+\rho_1|\eta_1|+\rho_2|\eta_2|+\left(\frac{\eta_1}{\rho_1}\right)^2+\left(\frac{\eta_2}{\rho_2}\right)^2}{|z_1-z_2|+\rho_1^2+\rho_2^2+\frac{\eta_1}{\rho_1}+\frac{\eta_2}{\rho_2}}, \qquad\quad \gamma_t:=\gamma(\eta_{1,t},z_{1,t},\eta_{2,t},z_{2,t}).
		\end{split}
	\end{equation}
	Assume that for some small $0<\xi\le \epsilon/10$, with very high probability, with $G_0^{z}(\ii\eta):=(H_0^z-\ii\eta)^{-1}$,  it holds
	\begin{equation}
		\label{eq:inassav}
		\big|\big\langle\big(G_0^{z_{1,0}}(\ii\eta_1)A_1G^{z_{2,0}}(\ii\eta_2)-M_{12}^{A_1}(\eta_1,z_{1,0},\eta_2,z_{2,0})\big)A_2\big\rangle\big|\le \frac{n^\xi}{\sqrt{n\ell_0}\gamma_0}\wedge \frac{n^\xi}{n|\eta_1\eta_2|}
	\end{equation}
	uniformly in $\lVert A_i\rVert\le 1$ and for any $|\eta_i|\in [|\eta_{i,0}|,n^{100}]$, where $\ell_0,\gamma_0$ are evaluated at\footnote{ We point out that we require \eqref{eq:inassav} for $(\eta_1,z_{1,0},\eta_2,z_{2,0})$, with $|\eta_i|\in[|\eta_{i,0}|,n^{100}]$, and not only for $\eta_i=\eta_{i,0}$.} $(\eta_1,z_{1,0},\eta_2,z_{2,0})$. Then,
	\begin{equation}
		\label{eq:goalav}
		\big|\langle (G_{1,t}A_1G_{2,t}-M_{12,t}^{A_1})A_2\rangle\big|\le \frac{n^{2\xi}}{\sqrt{n\ell_t}\gamma_t}\wedge \frac{n^{2\xi}}{n|\eta_{1,t}\eta_{2,t}|},
	\end{equation}
	with very high probability uniformly in $\lVert A_i\rVert\le 1$ and all $t\le T^*$ such that $n\ell_t\ge n^\epsilon$. \end{proposition}

We now briefly comment on the significance of the two control parameters \eqref{eq:defcontrolparam}. The condition $n\ell=1$ denotes the local fluctuation scale of the eigenvalues. We prove local laws always in the regime $n\ell\gg 1$, i.e.
above the fluctuation scale.
The parameter $\gamma$ represents a control on the stability operator. In fact, $\gamma$ is chosen so that (see \eqref{M_bound})
\begin{equation}
	\big\lVert M_{12}^A\big\rVert\lesssim \frac{\lVert A\rVert}{\gamma},
\end{equation}
uniformly in the spectrum. Furthermore, we also point out that the explicit formulas for
$\ell,\gamma$ simplify when $\Lambda_i$ are in the bulk of the spectrum; in fact in this case we have
\begin{equation}
	\label{eq:bulkcontrol}
	\ell\sim\min_{i=1,2}|\eta_i|, \qquad\quad \gamma\sim|z_1-z_2|^2+|\eta_1|+|\eta_2|.
\end{equation}
 Close to the edge of the spectrum, i.e. when $|z_i|\approx 1$, we instead have $\gamma \sim |z_1-z_2|+\eta_1^{2/3}+\eta_2^{2/3}$. \nc

We now state the local law in the isotropic case. We state it in a slightly more
general form as it will be needed later in the zag-part of the proof in Section~\ref{sec:zag}, but we point that the most delicate regime is $b=1$, since is the regime when we gain the most from the parameter $\gamma$.
\begin{proposition}
	\label{prop:isoflow}
	Fix small $n$--independent constants $\tau,\omega_*,\epsilon>0$, and for $i=1,2$ fix spectral parameters $\Lambda_{i,0}$ as in \eqref{eq:defincond}. Fix an exponent $b\in [0,1]$.
	For $t\le T^*$ let $\Lambda_{i,t}$ be the solution of \eqref{eq:charflowmat} with initial condition $\Lambda_{i,0}$. Let $G_{i,t}:=(W_t-\Lambda_{i,t})^{-1}$, and let $M_{12,t}^A$, $M_{121,t}^{A_1,A_2}$ be defined in \eqref{eq:defm12} and \eqref{eq:defm121}, respectively. Let $\ell_t, \gamma_t$ be from \eqref{eq:defcontrolparam}, and assume that for some small $0<\xi\le \epsilon/10$, with very high probability, it holds
	\begin{equation}
		\label{eq:inassiso}
		\begin{split}
			\big|\big\langle {\bm x}, \big(G_0^{z_{1,0}}(\ii\eta_1)A_1G_0^{z_{2,0}}(\ii\eta_2)-M_{12}^{A_1}(\eta_1,z_{1,0},\eta_2,z_{2,0})\big) {\bm y}\big\rangle\big|&\le 
			\frac{n^\xi (\rho^*)^{\frac{1-b}{2}} }{\sqrt{n} \eta_*^{\frac{3-b}{2}}\gamma_0^{\frac{b}{2}} }, \\
			\big|\big\langle {\bm x}, \big(G_0^{z_{1,0}}(\ii\eta_1)A_1G_0^{z_{2,0}}(\ii\eta_2)A_2G_0^{z_{1,0}}(\ii\eta_1)-M_{121}^{A_1,A_2}(\eta_1,z_{1,0},\eta_2,z_{2,0})\big) {\bm y}\big\rangle\big|&\le \frac{n^\xi}{\eta_*^{2-b}\gamma_0^b},
		\end{split}
	\end{equation}
	where $\eta_*:=|\eta_1|\wedge|\eta_2|$,  uniformly in $\lVert A_i\rVert\le 1$, $\lVert {\bm x}\rVert+\lVert {\bm y}\rVert\lesssim 1$, and for any $|\eta_i|\in [|\eta_{i,0}|,n^{100}]$, where $\ell_0,\gamma_0$ are evaluated at $(\eta_1,z_{1,0},\eta_2,z_{2,0})$. Then,
	\begin{equation}
		\label{eq:goaliso}
		\begin{split}
			\big|\big\langle {\bm x}, \big(G_{1,t}A_1G_{2,t}-M_{12,t}^{A_1}\big) {\bm y}\big\rangle\big|&\le
			\frac{n^{2\xi} (\rho^*_t)^{\frac{1-b}{2}} }{\sqrt{n} \eta_{*,t}^{\frac{3-b}{2}}\gamma_t^{\frac{b}{2}} },\\
			\big|\big\langle {\bm x}, \big(G_{1,t}A_1G_{2,t}A_2G_{1,t}-M_{121,t}^{A_1,A_2}\big) {\bm y}\big\rangle\big|&\le \frac{n^{2\xi}}{(\eta_{*,t})^{2-b}\gamma_t^b},
		\end{split}
	\end{equation}
	where $\eta_{*,t}:=|\eta_{1,t}|\wedge|\eta_{2,t}|$, with very high probability uniformly in $\lVert A_i\rVert\le 1$, $\lVert {\bm x}\rVert+\lVert {\bm y}\rVert\lesssim 1$, and all $t\le T^*$ such that $n\ell_t\ge n^\epsilon$.
\end{proposition}

 Note that the estimate on the  three-resolvent isotropic chain is far from optimal, it shows only 
that the random chain can be bounded by the size of its deterministic approximation $M_{121}$, see \eqref{M_3_bound}.
This will, however, be sufficient for our purpose since for the quadratic variation we anyway need only a bound on the
random chain itself.
We are now ready to present the proofs of Propositions~\ref{prop:aveflow}--\ref{prop:isoflow}. We first present the proof in the averaged case and then consider the isotropic case. For the sake of brevity throughout this section (and in Section~\ref{sec:edge}) we only focus on the complex case, all the additional terms appearing in the real case can be handled analogously, we thus omit the details.

\begin{proof}[Proof of Proposition~\ref{prop:aveflow}]
	
	The proof of this proposition is very similar to \cite[Proposition 5.4]{gumbel},  in fact, given the bounds on the smallest eigenvalues of the stability operator from Lemma~\ref{lemma_beta}, the proof is the basically the same as the proof of \cite[Proposition 5.4]{gumbel}. While in the edge regime it is necessary to follow very precisely two bad directions, as a consequence of the instability of the cusp regime (see e.g. \cite[Eq.s (5.46)--(5.47)]{gumbel}),  in the bulk there is only one bad direction to follow rendering the argument easier and simpler to follow. For this reason we refer the reader to \cite[Proposition 5.4]{gumbel} for a general proof which works everywhere in the spectrum, once the input Lemma~\ref{lemma_beta} is given, and here we present the simpler proof in the bulk for reader's convenience. Since there are still similarities with estimates performed in the proof of \cite[Proposition 5.4]{gumbel}, here  we highlight the main differences and omit several technical details that can be found in \cite[Section 5]{gumbel}.

	Recall the definition $G_{i,t}=(W_t-\Lambda_{i,t})^{-1}$, then for any $A_1,A_2\in\C^{2n\times 2n}$, by It\^{o}'s formula, we have
	\begin{equation}
		\label{eq:fulleqaasimp}
		\begin{split}
			\dd \langle (G_{1,t}A_1G_{2,t}-M_{12,t}^{A_1})A_2\rangle&=\sum_{a,b=1}^{2n}\partial_{ab}\langle G_{1,t}A_1G_{2,t}A_2\rangle\frac{\dd B_{ab,t}}{\sqrt{n}}+\langle (G_{1,t}A_1G_{2,t}-M_{12,t}^{A_1})A_2\rangle\dd t \\
			&\quad+2\sum_{i\ne j}\langle (G_{1,t}A_1G_{2,t}-M_{12,t}^{A_1})E_i\rangle\langle M_{21,t}^{A_2}E_j\rangle\dd t \\
			&\quad+2\sum_{i\ne j}\langle M_{12,t}^{A_1} E_i\rangle\langle (G_{2,t}A_2G_{1,t}-M_{21,t}^{A_2})E_j\rangle\dd t \\
			&\quad+2\sum_{i\ne j}\langle (G_{1,t}A_1G_{2,t}-M_{12,t}^{A_1})E_i\rangle\langle (G_{2,t}A_2G_{1,t}-M_{21,t}^{A_2})E_j\rangle\dd t \\
			&\quad+ \langle (G_{1,t}-M_{1,t})\rangle\langle G_{1,t}^2A_1G_{2,t}A_2\rangle\dd t + \langle (G_{2,t}-M_{2,t})\rangle\langle G_{2,t}^2A_2G_{1,t}A_1\rangle\dd t.
		\end{split}
	\end{equation}
	Here $\partial_{ab}$ denotes the directional derivative in the direction $(W_t)_{ab}$, $B_{ab,t}$ denotes the $(a,b)$--entry of the matrix Brownian motion $B_t$ from \eqref{eq:OUmat}, and $\sum_{i\ne j}$ denotes a summation over the pairs $(i,j)\in\{(1,2),(2,1)\}$. We point out that here we also used the following lemma describing the evolution of the deterministic approximation for the product of two resolvents:
	\begin{lemma}[Lemma 5.5 \cite{mesoCLT}]
		For any matrices $A_1,A_2\in\C^{2N\times 2N}$, we have
		\begin{equation}
			\label{eq:evolM12}
			\partial_t \langle M_{12,t}^{A_1}A_2\rangle=\langle M_{12,t}^{A_1}A_2\rangle+\langle \mathcal{S}[M_{12,t}^{A_1}]M_{21,t}^{A_2}\rangle.
		\end{equation}
	\end{lemma}
	
	For concreteness, in the remainder of the proof we assume $\eta_{1,t}\eta_{2,t}<0$, the other case being completely analogous. Define
	\begin{equation}
		\label{eq:defshorthand}
		Y_t:=\max_{B_1,B_2\in\{A_1,A_2,E_+,E_-\}}\big|\langle (G_{1,t}B_1G_{2,t}-M_{12,t}^{B_1})B_2\rangle\big|,
	\end{equation}
	and the stopping time
	\begin{equation}
		\label{eq:defstopt}
		\tau_1:=\inf\left\{t\ge 0:  Y_t=n^{2\xi} \alpha_t\right\}\wedge T^*, \qquad\quad  \alpha_t:= \frac{1}{n|\eta_{1,t}\eta_{2,t}|}\wedge \frac{1}{\sqrt{n\ell_t}\gamma_t}.
	\end{equation}
	Recall that by \eqref{eq:bulkcontrol} in the bulk regime $\gamma_t\sim |z_{1,t}-z_{2,t}|^2+|\eta_{1,t}|+|\eta_{2,t}|$. From the monotonicity of $t\mapsto|\eta_t|$ (see Lemma~\ref{lem:charprop}), and \eqref{eq:sizez1z2}  it also follows that $\gamma_s\gtrsim \gamma_t$ for $s\le t$. If $|z_{1,t}-z_{2,t}|\ge n^{-\delta}$, for some very small fixed $0<\delta\le \xi/10$, then the desired result follows by a simple (almost) global law (see e.g. \cite[Lemma 5.4]{mesoCLT}). For the remainder of the proof we thus assume $|z_{1,t}-z_{2,t}|\le n^{-\delta}$. Note that $|z_{1,t}-z_{2,t}|\sim |z_{1,s}-z_{2,s}|$ for any $s,t$, by 
	\eqref{eq:sizez1z2}, it thus makes no difference at which time we impose conditions on $|z_{1,t}-z_{2,t}|$.

	Writing 
	\begin{equation}
		\label{eq:12topm}
		E_1=\frac{1}{2}\big(E_++E_-\big), \qquad\quad E_2=\frac{1}{2}\big(E_+-E_-\big),
	\end{equation}
	and proceeding similarly to \cite[Eqs. (5.28)--(5.41)]{gumbel}, from \eqref{eq:fulleqaasimp}  we readily obtain
	\begin{equation}
		\label{eq:critical}
		\begin{split}
			&\dd \langle (G_{1,t}A_1G_{2,t}-M_{12,t}^{A_1})A_2\rangle \\
			&\qquad\qquad\quad= \langle G_{1,t}A_1G_{2,t}-M_{12,t}^{A_1}\rangle\langle M_{21,t}^{A_2}\rangle\dd t -\langle (G_{1,t}A_1G_{2,t}-M_{12,t}^{A_1})E_-\rangle\langle M_{21,t}^{A_2}E_-\rangle\dd t\\
			&\qquad\qquad\quad\quad+\langle M_{12,t}^{A_1} \rangle\langle G_{2,t}A_2G_{1,t}-M_{21,t}^{A_2}\rangle\dd t-\langle M_{12,t}^{A_1} E_- \rangle\langle (G_{2,t}A_2G_{1,t}-M_{21,t}^{A_2})E_-\rangle\dd t+ h_t\dif t\nc +\dd e_t.
		\end{split}
	\end{equation}
	Here $e_t,h_t$ are a martingale and a forcing term such that
	\begin{equation}
		\label{eq:estforc}
		\int_0^{t\wedge \tau_1} \big|h_s\big|\,\dd s+\left[\int_0^{\cdot}\dif e_s\right]_{t\wedge \tau_1}^{1/2}\lesssim n^\xi\alpha_{t\wedge\tau_1},
	\end{equation}
	with very high probability.  Here $[\cdot]_t$ denotes the quadratic variation process.
	
	By explicit but fairly tedious computations we find  (see  (\ref{app_M2}) in Appendix \ref{app:lemma_M}) (recall $\lVert A_i\rVert\le 1$)
		\begin{equation}
		\label{eq:explcompdet}
		\big|\langle M_{12,t}^{A_1} E_- \rangle\big|\lesssim \frac{|z_1-z_2|}{\gamma_t}\le \frac{n^{-\delta}}{\gamma_t}, \qquad\quad \big|\langle M_{21,t}^{A_2}\rangle\big|\le \big|\langle M_{21,t}^{E_+}\rangle\big|+\mathcal{O}\left( \frac{n^{-\delta}}{\gamma_t}\right).
	\end{equation}
	 Integrating \eqref{eq:critical} in time (from $0$ to $t\wedge\tau_1$) and using \eqref{eq:explcompdet} to estimate the $M_{12,t}, M_{21,t}$--terms in \eqref{eq:critical}, we thus obtain
	\begin{equation}
		Y_{t\wedge\tau_1}\le Y_0+\int_0^{t\wedge\tau_1} \left(2\big|\langle M_{21,s}^{E_+}\rangle\big|+\frac{n^{-\delta}}{\gamma_s}\right)Y_s\,\dd s+n^\xi\alpha_{t\wedge\tau_1}.
	\end{equation}
	Furthermore, by \cite[Eq. (5.31)]{mesoCLT}, we have
	\begin{equation}
		\label{eq:propagator}
		\exp\left(2\int_s^t \big|\langle M_{21,r}^{E_+}\rangle\big|\,\dd r\right)\le \frac{\gamma_s^2}{\gamma_t^2}.
	\end{equation}
	Finally, by \eqref{eq:propagator}, using a Gronwall inequality together with \eqref{eq:inassiso} to estimate $Y_0\le \alpha_{t\wedge\tau_1}$, we obtain $Y_{t\wedge\tau_1}\le n^\xi\alpha_{t\wedge\tau_1}$. We point out that here we also
	used (recall the definition of $\alpha_t$ from \eqref{eq:defstopt})
	\begin{equation}
		\label{eq:impfinalineq}
		\int_0^t\alpha_s \frac{1}{\gamma_s} \frac{\gamma_s^2}{\gamma_t^2}\,\dd s\lesssim \alpha_t,
	\end{equation}
	which follows by simple explicit computations. Note that the exponent $2$ in \eqref{eq:impfinalineq} is very important in the regime $\gamma_s\sim |\eta_{1,s}|+|\eta_{2,s}|$, in fact, in this regime, \eqref{eq:impfinalineq} would not be correct if $(\gamma_s/\gamma_t)$ had any other exponent larger than $2$. In the regime when $\gamma_s\sim |z_{1,s}-z_{2,s}|^2$, the exact power of $(\gamma_s/\gamma_t)$ does not matter. This proves $\tau_1=T^*$ with very high probability, concluding the proof.
\end{proof}

\begin{proof}[Proof of Proposition~\ref{prop:isoflow}]
	
	Within this proof  we use that the averaged local law from Proposition~\ref{prop:aveflow} is already proven. In fact the evolution of averaged quantities does not contain isotropic quantities, but the evolution of isotropic quantities contains both averaged and isotropic ones.
	
	We now write the evolution of isotropic products of two and three resolvents $G_{i,t}=(W_t-\Lambda_{i,t})^{-1}$ along the flow \eqref{eq:OUmat}. Let $A_1,A_2\in\C^{2n\times 2n}$, then, by It\^{o}'s formula, we have the following equations. For products of two resolvents we have: 
	\begin{equation}
		\begin{split}
			\label{eq:2gisoa}
			\dd (G_{1,t}A_1G_{2,t}-M_{12,t}^{A_1})_{\bm x\bm y}&
			=\dd \mathcal{E}_t+(G_{1,t}A_1G_{2,t}-M_{12,t}^{A_1})_{\bm x\bm y}\dd t 
			+2\sum_{i\ne j}\langle M_{12,t}^{A_1}E_i\rangle (G_{1,t}E_jG_{2,t}-M_{12,t}^{E_j})_{\bm x\bm y}\dd t \\
			&\quad +2\sum_{i\ne j}(M_{12,t}^{E_j})_{\bm x\bm y} \langle (G_{1,t}A_1G_{2,t} -M_{12,t}^{A_1})E_i\rangle \dd t\\
			&\quad+2\sum_{i\ne j}\langle (G_{1,t}A_1G_{2,t} -M_{12,t}^{A_1})E_i\rangle (G_{1,t}E_jG_{2,t}-M_{12,t}^{E_j})_{\bm x\bm y}\dd t \\
			&\quad+\langle G_{1,t}-m_{1,t}\rangle  (G_{1,t}^2A_1G_{2,t})_{\bm x\bm y} \dd t +\langle G_{2,t}-m_{1,t}\rangle \dd (G_{1,t}A_1G_{2,t}^2)_{\bm x\bm y} \dd t,
		\end{split}
	\end{equation}
	with
	\begin{equation}
		\dd \mathcal{E}_t:= \frac{1}{\sqrt{n}}\sum_{a,b=1}^{2n}\partial_{ab} (G_{1,t}A_1G_{2,t})_{\bm x\bm y} \dd B_{ab,t}.
	\end{equation}
	 Here we recall that $\partial_{ab}$ denotes the directional derivative in the direction $(W_t)_{ab}$, $B_{ab,t}$ denotes the $(a,b)$--entry of the matrix Brownian motion $B_t$ from \eqref{eq:OUmat}, and $\sum_{i\ne j}$ denotes a summation over the pairs $(i,j)\in\{(1,2),(2,1)\}$.  For products of three resolvents we have:
	\begin{equation}
		\begin{split}
			\label{eq:3gisosubma}
			&\dd (G_{1,t}A_1G_{2,t}A_2G_{1,t}-M_{121,t}^{A_1,A_2})_{\bm x\bm y} \\
			&=\dd \widehat{\mathcal{E}}_t+\frac{3}{2}(G_{1,t}A_1G_{2,t}A_2G_{1,t}-M_{121,t}^{A_1,A_2})_{\bm x\bm y}  \dd t\\
			&\quad+(LT)+(EI)+(EII)+(EIII).
		\end{split}
	\end{equation}
	with
	\begin{equation}
		\dd \widehat{\mathcal{E}}_t:= \frac{1}{\sqrt{n}}\sum_{a,b=1}^{2n}\partial_{ab} (G_{1,t}A_1G_{2,t}A_2G_{1,t})_{\bm x\bm y} \dd B_{ab,t},
	\end{equation}
	and
	\begin{equation}
		\begin{split}
			(LT):& =\sum_{i\ne j}\langle M_{21,t}^{A_2}E_i\rangle (G_{1,t}A_1G_{2,t}E_jG_{1,t}-M_{121,t}^{A_1,E_j})_{\bm x\bm y}\dd t \\
			&\quad+\sum_{i\ne j}\langle M_{12,t}^{A_1}E_i\rangle (G_{1,t}E_jG_{2,t}A_2G_{1,t}-M_{121,t}^{E_j,A_2})_{\bm x\bm y}\dd t \\
			(EI):&=\sum_{i\ne j}\langle (G_{1,t}A_1G_{2,t}A_2G_{1,t}-M_{121,t}^{A_1,A_2})E_i\rangle \big[(M_{11,t}^{E_j})_{\bm x\bm y}+(G_{1,t}E_jG_{1,t}-M_{11,t}^{E_j})_{\bm x\bm y}\big]\dd t \\
			&\quad+\sum_{i\ne j}\langle M_{121,t}^{A_1,A_2}E_i \rangle (G_{1,t}E_jG_{1,t}-M_{11,t}^{E_j})_{\bm x\bm y}\dd t\\
			(EII):&=\sum_{i\ne j}\langle (G_{1,t}A_1 G_{2,t}-M_{12,t}^{A_1})E_i\rangle\big[(M_{121,t}^{E_j,A_2})_{\bm x\bm y}+(G_{1,t}E_jG_{2,t}A_2G_{1,t}-M_{121,t}^{E_j,A_2})_{\bm x\bm y}\big]\dd t\\
			&\quad+\sum_{i\ne j}\langle (G_{2,t}A_2 G_{1,t}-M_{21,t}^{A_2})E_i\rangle\big[(M_{121,t}^{A_1,E_j})_{\bm x\bm y}+(G_{1,t}A_1G_{2,t}E_jG_{1,t}-M_{121,t}^{A_1,E_j})_{\bm x\bm y}\big]\dd t\\
			(EIII):&=\langle G_{1,t}-m_{1,t}\rangle  (G_{1,t}^2A_1G_{2,t}A_2G_{1,t})_{\bm x\bm y} \dd t +\langle G_{2,t}-m_{2,t}\rangle (G_{1,t}A_1G_{2,t}^2A_2G_{1,t})_{\bm x\bm y} \dd t \\
			&\quad+\langle G_{1,t}-m_{1,t}\rangle  (G_{1,t}A_1G_{2,t}A_2G_{1,t}^2)_{\bm x\bm y} \dd t,
		\end{split}
	\end{equation}

	We point out that in the derivation of \eqref{eq:2gisoa}--\eqref{eq:3gisosubma} we also used the following lemma whose proof is postponed to Appendix~\ref{app:lemma_M}. 
	\begin{lemma}\label{lemma_meta}
		For any matrices $A_1,A_2\in\C^{2N\times 2N}$ and vectors ${\bm x}, {\bm y}\in \C^{2N}$, we have
		\begin{equation}
			\begin{split}
				\label{eq:evolM121iso}
				\partial_t (M_{12,t}^{A_1})_{{\bm x}{\bm y}}&=(M_{12,t}^{A_1})_{{\bm x}{\bm y}}+(\mathcal{S}[M_{12,t}^{A_1}]M_{21,t}^I)_{{\bm x}{\bm y}}, \\
				\partial_t (M_{121,t}^{A_1,A_2})_{\bm x\bm y}&=(M_{121,t}^{A_1,A_2})_{\bm x\bm y}+(\mathcal{S}[M_{121,t}^{A_1,A_2}]M_{11,t}^I)_{\bm x\bm y}+(\mathcal{S}[M_{12,t}^{A_1}]M_{121,t}^{I,A_2}))_{\bm x\bm y} +(M_{121,t}^{A_1,I})\mathcal{S}[M_{21,t}^{A_2}])_{\bm x\bm y}.
			\end{split}
		\end{equation}
	\end{lemma}

	We now introduce the short--hand notations
	\begin{equation}
		\label{eq:defYiso}
		\begin{split}
			Y_{12,t}(B_1):&= (G_{1,t}B_1G_{2,t}-M_{12,t}^{B_1})_{{\bm x}{\bm y}}, \\
			Y_{121,t}(B_1,B_2):&=(G_{1,t}B_1G_{2,t}B_2G_{1,t}-M_{121,t}^{B_1,B_2})_{{\bm x}{\bm y}},
		\end{split}
	\end{equation}
	where the vectors ${\bm x}, {\bm y}$ are fixed and omitted from the notation, and define the times
	\begin{equation}
		T_i^\epsilon=T^\epsilon(\eta_{i,0},z_{i,0}):=\min\{t>0: N\eta_{i,t}\rho_{i,t}=n^\epsilon\}
	\end{equation}
	such that $\eta_{i,t}\rho_{i,t}$ reaches the threshold $n\eta_{i,T_i^\epsilon}\rho_{i,T_i^\epsilon}=n^\epsilon$. For the deterministic approximations we have the bounds (see \eqref{M_bound} and \eqref{M_3_bound} of Lemma \ref{lemma_M})
	\begin{equation}
		\label{eq:detbounds}
		\lVert M_{12,t}^{A_1}\rVert\lesssim \frac{\lVert A_1\rVert}{\gamma_t}, \qquad\quad \lVert M_{121,t}^{A_1,A_2}\rVert\lesssim \frac{\lVert A_1\rVert\lVert A_2\rVert}{\eta_{*,t}\gamma_t}.
	\end{equation}
	
 From now on, for simplicity of the presentation, we assume that $b=1$, which is anyway the most delicate regime, as we pointed out before the statement of Proposition~\ref{prop:isoflow}.	Define the stopping time
	\begin{equation}
		\label{eq:deftau2}
		\tau_2:=\inf\left\{t\ge 0:\sup^*_{\widetilde{\eta}_{i,0}}Y_{12,t}=\frac{n^{2\xi} }{\sqrt{n\widetilde{\gamma}_t} \widetilde{\eta}_{*,t}}, \qquad \sup^*_{\widetilde{\eta}_{i,0}}Y_{121,t}=\frac{n^{2\xi}}{\widetilde{\eta}_{*,t}\widetilde{\gamma}_t}\right\}\wedge T_1^\epsilon\wedge T_2^\epsilon,
	\end{equation}
	 where
	\[
	Y_{12,t}:=\max_{B_1\in\{A_1,E_+,E_-\}}\big|Y_{12,t}(B_1)\big|, \qquad\quad Y_{121,t}:=\max_{B_1,B_2\in\{A_1,A_2,E_+,E_-\}} \big|Y_{121,t}(B_1,B_2)\big|.
	\]
	Here $\widetilde{\gamma}_t:=\gamma(\widetilde{\eta}_{1,t},z_{1,t},\widetilde{\eta}_{2,t},z_{2,t})$, with $\widetilde{\eta}_{i,t}$ being the solution of \eqref{eq:char} with initial condition $\widetilde{\eta}_{i,0}$, and $\widetilde{\ell}_t$, $\widetilde{\eta}_{*,t}$ are defined similarly. The supremum $\sup^*$ is taken over all the initial conditions $\widetilde{\eta}_{i,0}$ with $\widetilde{\eta}_{i,0}\in [\eta_{i,0},\omega_*]$. We point out that in \eqref{eq:deftau2} we need to take the supremum over a larger set of initial conditions $\widetilde{\eta}_{i,0}$ only because at some point (see \eqref{eq:longred} below) we need to consider quantities involving $|G|$, and for $|G|$ we use an integral representation which requires integration over an entire segment of spectral parameters.
	
	In the following, for $t\le \tau_2$, we will often use the bound
	\begin{equation}
		\label{eq:usb}
		\sup_{\eta_{i,t}\le\check{\eta}_i\le \omega_*/2} \left[\sqrt{n\check{\gamma}_t} \check{\eta}_{*,t} Y_{12,t}+\widehat{\eta}_{*,t}\check{\gamma}_tY_{121,t}\right]\le n^{2\xi},
	\end{equation}
	with very high probability. Here $\check{\gamma}_t:=\gamma(\check{\eta}_1,z_{1,t},\check{\eta}_2,z_{2,t})$, and $\check{\ell}_t$, $\check{\eta}_{*,t}$ are defined similarly. We point that in \eqref{eq:usb} we also used that, as a consequence of \eqref{lem:ODEtheo},  for any fixed $t\le T_1^\epsilon\wedge T_2^\epsilon$ and any $\check{\eta}_i\in[\eta_{i,0},\omega_*/2]$ there is a $\widetilde{\eta}_{i,0}\in [\eta_{i,0},\omega_*]$ such that $\widetilde{\eta}_{i,t}=\check{\eta}_i$.
	
	We now estimate the terms in the rhs. of \eqref{eq:2gisoa}--\eqref{eq:3gisosubma} one by one. We first present the estimates for \eqref{eq:2gisoa} and then for \eqref{eq:3gisosubma}.
	
	\medskip
	
	\textbf{Estimates for \eqref{eq:2gisoa}:} We start computing the quadratic variation of the stochastic term
	\begin{equation}
		\label{eq:explquadvar}
		\dd[\mathcal{E}_t,\mathcal{E}_t]=\frac{1}{n\eta_{1,t}^2} (\Im G_{1,t})_{\bm x\bm x} (G_{2,t}^*A_1\Im G_{1,t}A_1G_{2,t})_{\bm y \bm y}\dd t+\frac{1}{n\eta_{2,t}^2} (G_{1,t}A_1\Im G_{2,t}A_1G_{1,t}^*)_{\bm x \bm x} (\Im G_{2,t})_{\bm y\bm y} \dd t.
	\end{equation}
	Using \eqref{eq:usb} together with the second bound in \eqref{eq:detbounds}, we estimate
	\begin{equation}
		\dd[\mathcal{E}_t,\mathcal{E}_t]\lesssim  \frac{n^{2\xi}\rho_{1,t}}{n\eta_{1,t}^2\eta_{*,t}\gamma_t}+\frac{n^{2\xi}\rho_{2,t}}{n\eta_{2,t}^2\eta_{*,t}\gamma_t}\dd t.
	\end{equation}
	 We point out that in the following we will often use several properties of the characteristics from Lemma~\ref{lem:charprop} even if not stated explicitly. \nc By the Burkholder--Davis--Gundy inequality we thus obtain
	\begin{equation}
		\label{eq:BDG1}
		\sup_{0\le r\le t}\left|\int_0^{r\wedge\tau_2}\dd \mathcal{E}_s\right|\lesssim n^{3\xi/2} \sqrt{\int_0^{t\wedge\tau_2} \frac{1}{n\eta_{*,s}\gamma_s}\left(\frac{\rho_{1,s}}{\eta_{1,s}^2}+\frac{\rho_{2,s}}{\eta_{2,s}^2}\right)\dd s}\lesssim \frac{n^{3\xi/2}}{\sqrt{n\gamma_{t\wedge\tau_2}}\eta_{*,{t\wedge\tau_2}}},
	\end{equation}
	where we used $\gamma_s\gtrsim \gamma_t$, for $s\le t$, and the first relation in \eqref{eq:explint}.
	
	The second term in the rhs. of \eqref{eq:2gisoa} can be neglected since, by considering the evolution of $e^{-t}(G_{1,t}A_1G_{2,t}-M_{12,t}^{A_1})_{{\bm x}{\bm y}}$, it amounts to a simple, negligible, rescaling of size $e^t\sim 1$. Next, by Proposition~\ref{prop:aveflow}, we estimate the terms in the second line of \eqref{eq:2gisoa} by
	\begin{equation}
		\left|\int_0^t (M_{12,s}^{E_j})_{\bm x\bm y} \langle (G_{1,s}A_1G_{2,s} -M_{12,s}^{A_1})E_i\rangle \dd s\right|\lesssim \int_0^t \frac{1}{\gamma_s}\cdot \left(\frac{n^\xi}{\sqrt{n\ell_s}\gamma_s}\wedge \frac{n^\xi}{n\eta_{*,s}^2}\right)\,\dd s \lesssim  \frac{n^\xi}{\sqrt{n\ell_t}\gamma_t}\le \frac{1}{\sqrt{n\gamma_t}\eta_{*,t}},
	\end{equation}
	where in the last inequality we used that  $\sqrt{\ell_t\gamma_t}\ge \eta_{*,t}$.
	In the second inequality we also used $\gamma_s\gtrsim \gamma_t$, for $s\le t$
	and the first relation in \eqref{eq:explint}. Similarly, using the definition of stopping time $\tau_2$, we estimate
	\begin{equation}
		\left|\int_0^{t\wedge\tau_2} \langle (G_{1,s}A_1G_{2,s} -M_{12,s}^{A_1})E_i\rangle (G_{1,s}E_jG_{2,s}-M_{12,s}^{E_j})_{\bm x\bm y}\dd s\right|\lesssim \frac{n^{2\xi}}{ \sqrt{n\ell_s}}\cdot \frac{1}{\sqrt{n\gamma_{t\wedge\tau_2}}\eta_{*,{t\wedge\tau_2}}}  \le  \frac{1}{\sqrt{n\gamma_{t\wedge\tau_2}}\eta_{*,{t\wedge\tau_2}}}.
	\end{equation}
	 where in the last inequality we used that $\sqrt{n\ell_t}\ge n^{\epsilon/2}>n^{2\xi}$. 
	
	For the terms in the last line of \eqref{eq:3gisosubma}, by spectral decomposition, for $t\le \tau$, we use
\begin{equation}
 \big| (G_{1,t}^2A_1G_{2,t})_{\bm x\bm y}\big|\le \frac{1}{\eta_{1,t}} (\Im G_{1,t})_{\bm x\bm x}^{1/2} (G_{2,t}^*A_1\Im G_{1,t}A_1G_{2,t})_{\bm y\bm y}^{1/2}\lesssim \frac{n^\xi\sqrt{\rho_{1,t}}}{\eta_{1,t}\sqrt{\eta_{*,t}\gamma_t}}.
\end{equation}
	We thus estimate those terms by
	\begin{equation}
		\label{eq:estwred}
		\left|\int_0^t \langle G_{1,s}-m_{1,s}\rangle  (G_{1,s}^2A_1G_{2,s})_{\bm x\bm y} \dd s\right| \lesssim \int_0^t \frac{n^\xi}{n\eta_{1,s}}\cdot  \frac{n^\xi\sqrt{\rho_{1,s}}}{\eta_{1,s}\sqrt{\eta_{*,s}\gamma_s}}\,\dd s\lesssim \frac{n^\xi}{\sqrt{n\gamma_t}\eta_{*,t}}.
	\end{equation}
	
	From \eqref{eq:2gisoa}, writing $E_1,E_2$ as in \eqref{eq:12topm}, combining \eqref{eq:BDG1}--\eqref{eq:estwred}, and using \eqref{eq:explcompdet}, we conclude
	\begin{align}
		\label{eq:almthtwo}
		(G_{1,{t\wedge\tau_2}}A_1G_{2,{t\wedge\tau_2}}-M_{12,{t\wedge\tau_2}}^{A_1})_{\bm x\bm y}=&(G_{1,0}A_1G_{2,0}-M_{12,0}^{A_1})_{\bm x\bm y}+\int_0^{t\wedge\tau_2}\langle M_{12,s}^{A_1}\rangle(G_{1,s}G_{2,s}-M_{12,s}^I)_{\bm x\bm y}\, \dd s\nonumber\\
		&+\mathcal{O}\left(\frac{n^{3\xi/2}}{\sqrt{n\gamma_{t\wedge\tau_2}}\eta_{*,{t\wedge\tau_2}}}\right).
	\end{align}

 From now on we assume that $A_1\in \{E_+, E_-\}$ without loss of generality, for presentational simplicity. In fact, if this is not the case then we can decompose
\begin{equation}
\label{eq:Adec}
A_1=\langle A_1 E_+\rangle E_++\langle A_1 E_-\rangle E_-+A_c,
\end{equation}
with $A_c$ so that $\lVert M_{12,t}^{A_c}\rVert\lesssim 1$, uniformly in $t\ge 0$. This last inequality holds since the subspace corresponding to the smallest eigenvalues of the stability operator is spanned by $E_-,E_+$ (see e.g. Appendix~\ref{app:lemma_M}).
 \nc
	
	Finally, using \eqref{eq:inassiso} to estimate the initial condition in \eqref{eq:almthtwo}, we obtain 
	\begin{equation}
	\label{eq:almtf}
	\begin{split}
		Y_{12, t\wedge\tau_2}(E_+)&= \int_0^{t\wedge\tau_2}\big[\langle M_{12,s}^{E_+}\rangle Y_{12,s}(E_+)+\langle M_{12,s}^{E_+}E_-\rangle Y_{12,s}(E_-)\big]\, \dd s+ \mathcal{O}\left(\frac{n^{3\xi/2}}{\sqrt{n\gamma_{t\wedge\tau_2}}\eta_{*,{t\wedge\tau_2}}}\right), \\
		Y_{12, t\wedge\tau_2}(E_-)&= \int_0^{t\wedge\tau_2}\big[\langle M_{12,s}^{E_-}\rangle Y_{12,s}(E_+)-\langle M_{12,s}^{E_-}E_-\rangle Y_{12,s}(E_-)\big]\, \dd s+\mathcal{O}\left(\frac{n^{3\xi/2}}{\sqrt{n\gamma_{t\wedge\tau_2}}\eta_{*,{t\wedge\tau_2}}}\right).
		\end{split}
	\end{equation}
	We are now ready to conclude the estimate for $Y_{12,t}$. In the following we omit some details and present only the main steps since this is very similar to the proof in \cite[Eqs. (5.43)--(5.46) and Eqs. (5.58)--(5.59)]{gumbel}.
	First, we introduce the short--hand notations
	\[
	\mathcal{Y}_t:=\left(\begin{matrix}
	Y_{12, t\wedge\tau_2}(E_+) \\
	Y_{12, t\wedge\tau_2}(E_-)
	\end{matrix}\right), \qquad\qquad\quad \mathcal{M}_t:=\left(\begin{matrix}
	\langle M_{12,s}^{E_+}\rangle & \langle M_{12,s}^{E_+}E_-\rangle \\
	\langle M_{12,s}^{E_-}\rangle & -\langle M_{12,s}^{E_-}E_-\rangle
	\end{matrix}\right),
	\]
	and
	\[
	\mathcal{F}_t:=\mathrm{last\,\, three \,\, lines \,\, of \,\, \eqref{eq:2gisoa}}.
	\]
	Then, from \eqref{eq:2gisoa}, combining \eqref{eq:explquadvar}--\eqref{eq:Adec}, we can rewrite \eqref{eq:almtf} in the more precise form
	\begin{equation}
	\label{eq:matrf}
	\mathcal{Y}_{t\wedge\tau_2}=\int_0^{t\wedge\tau_2} \mathcal{M}_s\mathcal{Y}_s\,\dif s+\int_0^{t\wedge\tau_2} \big[\mathcal{F}_s\dif s+\dif\mathcal{E}_s\big],
	\end{equation}
	with
	\begin{equation}
	\left(\int_0^{t\wedge\tau_2} \lVert \mathcal{F}_s\rVert\,\dif s\right)^2+\int_0^{t\wedge\tau_2}\lVert \mathcal{C}_s\rVert \,\dif s\le \frac{n^{3\xi}}{n\gamma_{t\wedge\tau_2}\eta_{*,t\wedge\tau_2}^2}.
	\end{equation}
	Here $\mathcal{C}_s$ denotes the $2\times 2$ covariance matrix of the martingale $\mathcal{E}_s$. Next, we recall that from \cite[Eqs. (A.14)--(A.16)]{gumbel} it follows that $\langle M_{12,s}^{E_+}\rangle>-\langle M_{12,s}^{E_-}E_-\rangle$, for $\eta_{1,s}\eta_{2,s}<0$ and that $\langle M_{12,s}^{E_\pm}E_\mp\rangle$ is purely imaginary.
	
	We are now ready to conclude the proof by an integral Gronwall inequality as in \cite[Lemma 5.6]{gumbel}. Define
	\begin{equation}
	\label{eq:defalphaf}
	\alpha_t:=\frac{1}{\sqrt{n\gamma_t}\eta_{*,t}},\qquad\quad f_t:=\langle M_{12,t}^{E_+}\rangle.
	\end{equation}
	Then, by \cite[Eqs. (A.14)--(A.16)]{gumbel}, it is easy to see that \eqref{eq:matrf} satisfies the hypothesis of \cite[Lemma 5.6]{gumbel}. We can thus apply a matrix Gronwall inequality from \cite[Lemma 5.6]{gumbel} and obtain \nc  $Y_{12, t\wedge\tau_2}= \lVert \mathcal{Y}_{t\wedge\tau_2}\lVert\nc\lesssim n^{3\xi/2}/(\sqrt{n\gamma_{t\wedge\tau_2}}\eta_{*,{t\wedge\tau_2}})$. We point out that in this last inequality we used
	\begin{equation}
	\label{eq:newestiso}
	\exp\left(\int_s^t f_t\, \dd r\right)\lesssim \frac{\eta_{*,s}}{\eta_{*,t}},
	\end{equation}
	which follows from $f_t\lesssim \sqrt{\rho_{1,r}\rho_{2,r}/(\eta_{1,r}\eta_{2,r})}$ and $\sqrt{\eta_{1,s}\eta_{2,s}/(\eta_{1,t}\eta_{2,t})}\sim \eta_{*,s}/\eta_{*,t}$ (for this second relation, see the similar proof \cite[Eq. (5.32)]{mesoCLT}), and
	\begin{equation}
		\label{eq:lineqbast}
		\int_0^t  \frac{1}{\sqrt{n\gamma_s}\eta_{*,s}} \cdot \frac{1}{\gamma_s}\cdot  \frac{\eta_{*,s}}{\eta_{*,t}} \, \dd s\lesssim  \frac{1}{\sqrt{n\gamma_t}\eta_{*,t}},
	\end{equation}
	where we used that $\int (1/\gamma_s)\lesssim 1$. Note that in the averaged case we used the more involved estimate \eqref{eq:propagator}, instead here we used the weaker and simply obtainable bound \eqref{eq:newestiso} which is enough for the proof of the isotropic law. We conclude the estimate of the isotropic estimate of the product of two resolvents by pointing out that a similar proof goes through for the case $b\in [0,1)$, with the difference that \eqref{eq:lineqbast} can be replaced with
	\begin{equation}
	\int_0^t  \frac{(\rho_s)^*}{\sqrt{n}\eta_{*,s}^{(3-b)/2}\gamma_s^{b/2}} \cdot \frac{1}{\gamma_s}\cdot \frac{\eta_{*,s}}{\eta_{*,t}} \, \dd s\lesssim  \frac{(\rho_t)^*}{\sqrt{n}\eta_{*,t}^{3/2-b}\gamma_t^{b/2}},
	\end{equation}
	A similar comment applies to isotropic chains of three resolvents below, even if we will not point it out explicitly. \nc

	\medskip
	
	\textbf{Estimates for \eqref{eq:3gisosubma}:} Similarly to the two resolvent case above, we notice that the second term in the rhs. of \eqref{eq:3gisosubma} can be neglected from the analysis, since it only amounts to a simple rescaling of size $e^{3t/2}\sim 1$. Several of the following estimates are similar to the estimates that we preformed in the rhs. of \eqref{eq:2gisoa}, for this reason we omit some easily checkable details.  In particular, we point out that throughout we will use properties of the characteristics from Lemma~\ref{lem:charprop} even if not stated explicitly. 
	
	We now compute the quadratic variation of the stochastic term in \eqref{eq:3gisosubma}:
	\begin{equation}
		\label{eq:quadvar3gsa}
		\begin{split}
			\dif[\widehat{\mathcal{E}}_t,\widehat{\mathcal{E}}_t]&=\frac{1}{n\eta_{1,t}^2}\bigg[ (\Im G_{1,t})_{\bm x\bm x} (G_{1,t}^*A_2G_{2,t}^*A_1\Im G_{1,t}A_1G_{2,t}A_2G_{1,t})_{\bm y\bm y} \\
			&\qquad\qquad\qquad\qquad\quad+(G_{1,t}A_1G_{2,t}A_2\Im G_{1,t}A_2G_{2,t}^*A_1G_{1,t}^*)_{\bm x\bm x} (\Im G_{1,t})_{\bm y\bm y}\bigg] \dd t\\
			&\quad+\frac{1}{n\eta_{2,t}^2} (G_{1,t}A_1\Im G_{2,t}A_1G_{1,t}^*)_{\bm x\bm x}(G_{1,t}^*A_2\Im G_{2,t}A_2G_{1,t})_{\bm y\bm y} \dd t.
		\end{split}
	\end{equation}
	Note that (the first two lines of) the quadratic variation process in \eqref{eq:quadvar3gsa} contains chains with more than three resolvents, we thus rely on the following \emph{reduction inequalities} to shorten these chains.
	\begin{lemma}[Reduction inequalities]
		For any $Q_1,Q_2,Q_3,R_1,R_2\in\C^{2n\times 2n}$ it holds
		\begin{equation}
			\begin{split}
				\label{eq:redina}
				\big|(Q_1G(w_1)Q_2G(w_2)Q_3)_{\bm x\bm y}\big|&\lesssim \sqrt{N} (Q_1|G(w_1)|Q_1^*)_{\bm x\bm x}^{1/2} (Q_3|G(w_2)|Q_3^*)_{\bm x\bm x}^{1/2}\langle |G(w_1)|Q_2|G(w_2)|Q_2^*\rangle^{1/2} \\
				\langle |G(w_1)|R_1G(w_2)R_2|G(w_1)|R_2^* G(w_2)^*R_1^*\rangle&\lesssim N\langle |G(w_1)|R_1|G(w_2)|R_1^*\rangle\langle |G(w_1)|R_2^*|G(w_2)|R_2\rangle,
			\end{split}
		\end{equation}
	\end{lemma}
	Note that the reduction inequalities above involve $|G|$ as well. For this reason in the following we need to show that if a bound holds for certain products of resolvent then the same bound also holds if one (or more) $G$ is replaced by $|G|$ (see \eqref{eq:vabsval} below).

	To estimate \eqref{eq:quadvar3gsa}, for $t\le\tau_2$, we use \eqref{eq:redina} in the form
	\begin{equation}
		\label{eq:longred}
		(G_{1,t}^*A_2G_{2,t}^*A_1\Im G_{1,t}A_1G_{2,t}A_2G_{1,t})_{\bm y\bm y} \lesssim n (G_{1,t}^*|G_{2,t}|G_{1,t})_{\bm y\bm y} \langle \Im G_{1,t}|G_{2,t}|\rangle\lesssim \frac{n^{1+2\xi}}{\eta_{*,t}\gamma_t^2}.
	\end{equation}
	We point out that in the last inequality of \eqref{eq:longred} we used that for $t\le\tau_2$ we have
	\begin{equation}
		\label{eq:vabsval}
		(G_{1,t}^*|G_{2,t}|G_{1,t})_{\bm y\bm y}\lesssim \frac{n^{2\xi}\log n}{\eta_{*,t}\gamma_t}, \qquad\quad \langle \Im G_{1,t}|G_{2,t}|\rangle\lesssim \frac{\log n}{\gamma_t}.
	\end{equation}
	This follows by \cite[Appendix A]{gumbel} (see also \cite[Section 5]{Optimal_local}) and the fact that the same bound holds for the quantities with $|G|$ replaced with $G$ (see e.g. \cite[Proof of Eq. (5.32) in Appendix A]{gumbel} and \cite[Lemma 4.10]{CEHS23}). By \eqref{eq:quadvar3gsa} and the BDG inequality we thus obtain
	\begin{equation}
		\label{eq:BDG2}
		\sup_{0\le r\le t}\left|\int_0^{r\wedge\tau_2} \dif \widehat{\mathcal{E}}_s \right|\lesssim n^{3\xi/2}\sqrt{\int_0^{t\wedge\tau_2} \frac{1}{\eta_{*,s}\gamma_s^2}\left(\frac{\rho_{1,s}}{\eta_{1,s}^2}+\frac{\rho_{2,s}}{\eta_{2,s}^2}\right)\,\dd s} \lesssim \frac{n^{3\xi/2}}{\eta_{*, t\wedge\tau_2}\gamma_{t\wedge\tau_2}},
	\end{equation}
	with very high probability, where we used that the third line of \eqref{eq:quadvar3gsa}, by the definition of $\tau_2$, is bounded by $n^{4\xi}/(n\eta_{*,t}^4\gamma_t^2)\le 1/(\eta_{*,t}^3\gamma_t^2)$ (recall that $\xi\le\epsilon/10$).
	
	We now explain how to estimate the remaining terms in \eqref{eq:3gisosubma}. By spectral decomposition, Proposition~\ref{prop:aveflow}, and \eqref{eq:detbounds}, we have
	\[
	\big|\langle G_{1,t}A_1G_{2,t}A_2G_{1,t}\rangle\big|+\big|\langle M_{121,t}^{A_1,A_2} \rangle\big|\lesssim \frac{1}{\eta_{1,t}}\langle\Im G_{1,t} A_1|G_{2,t}|A_1^*\rangle^{1/2}\langle\Im G_{1,t} A_2|G_{2,t}|A_2^*\rangle^{1/2} +\big|\langle M_{121,t}^{A_1,A_2} \rangle\big|\lesssim \frac{1}{\eta_{*,t}\gamma_t},
	\]
	with very high probability.
	We can thus estimate the terms in (EI) of \eqref{eq:3gisosubma} by
	\begin{equation}
		\int_0^{t\wedge\tau_2} \frac{1}{\eta_{*,s}\gamma_s}\cdot  \frac{\rho_{1,s}}{\eta_{1,s}}\, \dd s\lesssim \frac{1}{\eta_{*, {t\wedge\tau_2}}\gamma_{t\wedge\tau_2}},
	\end{equation}
	 where we also used that $\lVert M_{11,t}^A\rVert\lesssim (\rho_{1,t}/\eta_{1,t}) \lVert A\rVert$.
	
	Then, by Proposition~\ref{prop:aveflow}, we estimate the terms in (EII) of \eqref{eq:3gisosubma} by
	\begin{equation}
		\int_0^{t\wedge\tau_2} \frac{1}{\sqrt{n\eta_{*,s}}\gamma_s}\cdot  \frac{n^{2\xi}}{\eta_{*,s}\gamma_s}\, \dd s\lesssim \frac{n^{2\xi}}{\sqrt{n\eta_{*, {t\wedge\tau_2}}}}\cdot \frac{1}{\eta_{*, {t\wedge\tau_2}}\gamma_{t\wedge\tau_2}}.
	\end{equation}
	We now consider the terms in (EIII) of \eqref{eq:3gisosubma}. For these terms, by \eqref{eq:redina}, we estimate (we just consider one of those term for concreteness)
\[
\begin{split}
\big|(G_{1,t}^2A_1G_{2,t}A_2G_{1,t})_{\bm x\bm y}\big|&\lesssim \frac{1}{\eta_{1,t}} (\Im G_{1,t})_{\bm x\bm x}^{1/2}(G_{1,t}^*A_2G_{2,t}^*A_1\Im G_{1,t}A_1G_{2,t}A_2G_{1,t})_{\bm y\bm y}^{1/2} \\
&\lesssim \frac{n^{1/2+\xi}\sqrt{\rho_{1,t}}}{\eta_{1,t}\eta_{*,t}\gamma_t},
\end{split}
\]
where in the last inequality we used \eqref{eq:longred}.
	We can thus estimate the terms in (EIII) of  \eqref{eq:3gisosubma} by
	\begin{equation}
		\label{eq:estlastlines}
		\int_0^t \frac{1}{n\eta_{*,s}}\cdot  \frac{n^{1/2+\xi}\sqrt{\rho_{1,t}}}{\eta_{1,t}\eta_{*,t}\gamma_t} \, \dd s\lesssim \frac{1}{\eta_{*,t}\gamma_t}.
	\end{equation}
	
	 From now on, without loss of generality, for simplicity we assume that $A_1,A_2\in\{E_+,E_-\}$ (see \eqref{eq:Adec}).
	
	Combining \eqref{eq:BDG2}--\eqref{eq:estlastlines} and using \eqref{eq:inassiso} to estimate the initial condition,
	from \eqref{eq:3gisosubma}, using the short--hand notation $\Upsilon_t^{\pm,\pm}:=Y_{121,t}(E_\pm,E_\pm)$, we obtain
	\begin{equation}
	\begin{split}
	\label{eq:almt3}
		\Upsilon_{t\wedge\tau}^{+,+}&=\int_0^{t\wedge\tau_2} \big[2\langle M_{12,s}^{E_+}\rangle \Upsilon_s^{+,+}-\langle M_{21,s}^{E_+}E_-\rangle \Upsilon_s^{+,-}-\langle M_{12,s}^{E_+}E_-\rangle \Upsilon_s^{-,+}\big]\, \dd s+\mathcal{O}\left(\frac{n^{3\xi/2}}{\eta_{*,t}\gamma_t}\right), \\
		\Upsilon_{t\wedge\tau}^{+,-}&=\int_0^{t\wedge\tau_2} \big[\big(\langle M_{12,s}^{E_+}\rangle-\langle M_{21,s}^{E_-}E_-\rangle\big) \Upsilon_s^{+,-}+\langle M_{21,s}^{E_-}\rangle \Upsilon_s^{+,+}-\langle M_{12,s}^{E_+}E_-\rangle \Upsilon_s^{-,-}\big]\, \dd s +\mathcal{O}\left(\frac{n^{3\xi/2}}{\eta_{*,t}\gamma_t}\right), \\
		\Upsilon_{t\wedge\tau}^{-,+}&=\int_0^{t\wedge\tau_2} \big[\big(\langle M_{21,s}^{E_+}\rangle-\langle M_{12,s}^{E_-}E_-\rangle\big) \Upsilon_s^{-,+}+\langle M_{12,s}^{E_-}\rangle \Upsilon_s^{+,+}-\langle M_{21,s}^{E_+}E_-\rangle \Upsilon_s^{-,-}\big]\, \dd s+\mathcal{O}\left(\frac{n^{3\xi/2}}{\eta_{*,t}\gamma_t}\right), \\
		\Upsilon_{t\wedge\tau}^{-,-}&=\int_0^{t\wedge\tau_2} \big[-2\langle M_{12,s}^{E_-}E_-\rangle \Upsilon_s^{-,-}+\langle M_{12,s}^{E_-}E_+\rangle \Upsilon_s^{+,-}+\langle M_{21,s}^{E_-}E_+\rangle \Upsilon_s^{-,+}\big]\, \dd s+\mathcal{O}\left(\frac{n^{3\xi/2}}{\eta_{*,t}\gamma_t}\right).
	\end{split}
	\end{equation}
	We are now ready to conclude the estimate of $Y_{121,t}$. We present only the main steps and omit some details, since this proof is analogous to the one presented in \eqref{eq:almtf}--\eqref{eq:lineqbast} and to the one in \cite[Eqs. (5.43)--(5.46) and Eqs. (5.58)--(5.59)]{gumbel}. From \eqref{eq:3gisosubma}, combining \eqref{eq:quadvar3gsa}--\eqref{eq:estlastlines}, in order to conclude the proof we can now proceed as in \eqref{eq:almtf}--\eqref{eq:lineqbast} and apply \cite[Lemma 5.6]{gumbel} to a linear system of the form
	\begin{equation}
	\label{eq:matrf2}
	\widehat{\mathcal{Y}}_{t\wedge\tau_2}=\int_0^{t\wedge\tau_2} \widehat{\mathcal{M}}_s\mathcal{Y}_s\,\dif s+\int_0^{t\wedge\tau_2} \big[\widehat{\mathcal{F}}_s\dif s+\dif\widehat{\mathcal{E}}_s\big],
	\end{equation}
	with
	\begin{equation}
	\label{eq:alhadef}
	\left(\int_0^{t\wedge\tau_2} \lVert \widehat{\mathcal{F}}_s\rVert\,\dif s\right)^2+\int_0^{t\wedge\tau_2}\lVert \widehat{\mathcal{C}}_s\rVert \,\dif s\le \frac{n^{3\xi}}{n\gamma_{t\wedge\tau_2}\eta_{*,t\wedge\tau_2}^2}=:\alpha_{t\wedge\tau_2},
	\end{equation}
	where $\widehat{\mathcal{C}}_t$ denotes the $4\times 4$ covariance matrix of the martingale $\widehat{\mathcal{E}}_t$. Here we also used the notations
\begin{equation}
\widehat{\mathcal{Y}}_t:=\left(\begin{matrix}
\Upsilon_t^{+,+} \\
\Upsilon_t^{+,-} \\
\Upsilon_t^{-,+} \\
\Upsilon_t^{-,-}
\end{matrix}\right), \qquad\quad \widehat{\mathcal{M}}_t:=\left(\begin{matrix}
2a_{+,t} & \ii b_t & -\ii b_t & 0\\
\ii b_t & a_{+,t}-a_{-,t} & 0 & -\ii b_t \\
-\ii b_t & 0 & a_{+,t}-a_{-,t} & \ii b_t \\
0 & -\ii b_t & \ii b_t & 2 a_{-,t},
\end{matrix}\right)
\end{equation}
where
\[
a_{+,t}:=\langle M_{12,t}^{E_+}\rangle, \qquad\quad a_{-,t}:= \langle M_{12,t}^{E_-}E_-\rangle,
\]
and $b_t$ is real and defined in \cite[Eq. (A.14)]{gumbel}. Note that the matrix $\widehat{\mathcal{M}}_t$ is exactly the same as the one \cite[Eq. (5.44)]{gumbel}, and that $\alpha_{t\wedge \tau_2}$ defined in \eqref{eq:alhadef} is different compared to the one defined in \eqref{eq:defalphaf}. We can thus use again the matrix Gronwall inequality from \cite[Lemma 5.6]{gumbel} for $\alpha_t$ defined in \eqref{eq:alhadef} and $f_t:=2\langle M_{12,t}^{E_+}\rangle$ (note that here we redefined $f_t$, i.e. there is an additional factor of $2$ compared to \eqref{eq:defalphaf}). By  \eqref{eq:newestiso}, we thus conclude \[
Y_{121,t\wedge\tau_2}= \lVert \mathcal{Y}_{t\wedge\tau_2}\rVert\nc\lesssim \frac{n^{3\xi/2}\log n}{\eta_{*,t\wedge\tau_2}\gamma_{t\wedge\tau_2}}.
\]
We point out that to obtain this inequality we also used
	\[
	\int_0^t \frac{1}{\eta_{*,s}\gamma_s}\cdot \frac{1}{\gamma_s}\cdot \frac{\eta_{*,s}^2}{\eta_{*,t}^2}\, \dd s\lesssim \frac{\log n}{\eta_{*,t}\gamma_t}.
	\]

	\medskip
	
	\textbf{Conclusion:} In the previous two parts we proved that
	\begin{equation}
		\label{eq:finalbounds}
		Y_{12, t\wedge\tau_2}\lesssim \frac{n^\xi}{\sqrt{n\gamma_{t\wedge\tau_2}}\eta_{*,{t\wedge\tau_2}}}, \qquad\quad Y_{121,t\wedge\tau_2}\lesssim \frac{n^{3\xi/2}\log n}{\eta_{*,t\wedge\tau_2}\gamma_{t\wedge\tau_2}}.
	\end{equation}
	These two bounds together prove the desired estimate for the initial conditions $\eta_{i,0}$. A completely analogous proof holds if we consider any other initial condition $\widetilde{\eta}_{i,0}\in [\eta_{i,0},\omega_*]$, obtaining exactly the same bounds \eqref{eq:finalbounds} (with tilde on all quantities, i.e. $\gamma\to \widetilde{\gamma}$, etc...). This, together with the definition of $\tau_2$ in \eqref{eq:deftau2} shows that $\tau_2=T_1^\epsilon\wedge T_2^\epsilon$ with very high probability.
	
\end{proof}

\subsection{Improved local law at the edge of the spectrum}\label{sec:edge}

In this section we present the proof of the averaged local law at the edge of the spectrum, including $\Im G$'s. While in the bulk of the spectrum there is no difference between $G$ and $\Im G$, we will see that close to the edge there will be an improvement (proportional to the local density) if we replace one (or more) $G$ by $\Im G$. For brevity, we present this proof only in the averaged case, since this is what gives the optimal bound on the eigenvector overlaps in Theorem~\ref{coro_edge}.  The result from Proposition~\ref{prop:aveflowim} below is the main technical input to prove Theorem~\ref{thm:2G_edge} (see the beginning of Section~\ref{sec:zig} for a related explanation).

\begin{proposition}
	\label{prop:aveflowim}
	Fix small $n$--independent constants $\tau,\omega_*,\epsilon>0$, and for $i=1,2$ fix spectral parameters $\Lambda_{i,0}$ as in \eqref{eq:defincond}, with $|z_{i,0}|\le 10$, $|\eta_{i,0}|\le \omega_*$. For $t\le T^*$ let $\Lambda_{i,t}$ be the solution of \eqref{eq:charflowmat} with initial condition $\Lambda_{i,0}$. Let $G_{i,t}:=(W_t-\Lambda_{i,t})^{-1}$, and let $\widehat{M}_{12,t}^A$ be the deterministic approximation of $\Im G_{1,t}A\Im G_{2,t}$ from \eqref{avlaw12_im}. Let $\ell$ from \eqref{eq:defcontrolparam}, and define the control parameter
	\begin{equation}
		\label{eq:defcontrolparama}
		\widehat{\gamma}=\widehat{\gamma}(\eta_1,z_1,\eta_2,z_2):=|z_1-z_2|^2+\rho_1|\eta_1|+\rho_2|\eta_2|+\left(\frac{\eta_1}{\rho_1}\right)^2+\left(\frac{\eta_2}{\rho_2}\right)^2, \qquad\quad \widehat{\gamma}_t:=\widehat{\gamma}(\eta_{1,t},z_{1,t},\eta_{2,t},z_{2,t}).
	\end{equation}
	Assume that for some small $0<\xi\le \epsilon/10$, with very high probability, it holds
	\begin{equation}
		\big|\big\langle\big(\Im G_{1,0}A_1\Im G_{2,0}-\widehat{M}_{12,0}^{A_1}\big)A_2\big\rangle\big|\le \frac{n^\xi\rho_{1,0}\rho_{2,0}}{\sqrt{n\ell_0}\widehat{\gamma}_0}\wedge \frac{n^\xi}{n|\eta_{1,0}\eta_{2,0}|}
	\end{equation}
	uniformly in $\lVert A_i\rVert\le 1$ and for any $|\eta_i|\in [|\eta_{i,0}|,n^{100}]$. Then,
	\begin{equation}
		\big|\langle (\Im G_{1,t}A_1\Im G_{2,t}-\widehat{M}_{12,t}^{A_1})A_2\rangle\big|\le \frac{n^{2\xi}\rho_{1,t}\rho_{2,t}}{\sqrt{n\ell_t}\widehat{\gamma}_t}\wedge \frac{n^{2\xi}}{n|\eta_{1,t}\eta_{2,t}|},
	\end{equation}
	with very high probability uniformly in $\lVert A_i\rVert\le 1$ and all $t\le T^*$ such that $n\ell_t\ge n^\epsilon$.
\end{proposition}

\begin{proof}

	To keep the presentation simple, we only consider the case $A_1,A_2\in \{E_+,E_-\}$. By \eqref{eq:fulleqaasimp} we have
	\begin{equation}
		\label{eq:imflow}
		\begin{split}
			&\dd \langle (\Im G_{1,t}A_1\Im G_{2,t}-M_{12,t}^{A_1})A_2\rangle \\
			&=\sum_{a,b=1}^{2n}\partial_{ab}\langle \Im G_{1,t}A_1 \Im G_{2,t}A_2\rangle\frac{\dd B_{ab,t}}{\sqrt{n}}+\langle (\Im G_{1,t}A_1 \Im G_{2,t}-\widehat{M}_{12,t}^{A_1})A_2\rangle\dd t \\
			&\quad+(LT)+(EI)+(EII)+(EIII),
		\end{split}
	\end{equation}
	where we defined
	\begin{equation}
		\begin{split}
			(LT):&=2\sum_{i\ne j}\langle (G_{1,t}^*A_1G_{2,t}-M_{\overline{1}2,t}^{A_1})E_i\rangle\langle \widehat{M}_{21,t}^{A_2}E_j\rangle\dd t \\
			&\quad+2\sum_{i\ne j}\langle (\Im G_{1,t} A_1 \Im G_{2,t}-\widehat{M}_{12,t}^{A_1})E_i\rangle\langle M_{\overline{2}1,t}^{A_2}E_j\rangle\dd t \\
			&\quad-4\sum_{i\ne j}\langle \widehat{M}_{12,t}^{A_1}E_i\rangle\langle (\Im G_{2,t}A_2\Im G_{1,t}-\widehat{M}_{21,t}^{A_2}\rangle)E_j\rangle \dd t \\
			&\quad-4\sum_{i\ne j}\langle (\Im G_{1,t}A_1\Im G_{2,t}-\widehat{M}_{12,t}^{A_1}\rangle)E_i\rangle\langle \widehat{M}_{21,t}^{A_2}E_j\rangle \dd t \\
			(EI):&= 2\sum_{i\ne j}\langle M_{\overline{1}2,t}^{A_1} E_i\rangle\langle (\Im G_{2,t}A_2\Im G_{1,t}-\widehat{M}_{21,t}^{A_2})E_j\rangle\dd t \\
			&\quad+2\sum_{i\ne j}\langle (\Im G_{1,t}A_1\Im G_{2,t}-\widehat{M}_{12,t}^{A_1})E_i\rangle\langle (G_{2,t}^*A_2 G_{1,t}-M_{\overline{2}1,t}^{A_2})E_j\rangle\dd t \\
			(EII):&=2\sum_{i\ne j}\langle (G_{1,t}^*A_1G_{2,t}-M_{\overline{1}2,t}^{A_1})E_i\rangle \langle (\Im G_{2,t}A_2\Im G_{1,t}-\widehat{M}_{21,t}^{A_2})E_j\rangle \dd t \\
			&\quad+2\sum_{i\ne j}\langle \widehat{M}_{12,t}^{A_1} E_i\rangle\langle (G_{2,t}^*A_2 G_{1,t}-\widehat{M}_{\overline{2}1,t}^{A_2})E_j\rangle\dd t \\
			&\quad-4\sum_{i\ne j}\langle (\Im G_{1,t}A_1\Im G_{2,t}-\widehat{M}_{12,t}^{A_1}\rangle)E_i\rangle\langle (\Im G_{2,t}A_2\Im G_{1,t}-\widehat{M}_{21,t}^{A_2}\rangle)E_j\rangle \dd t \\
			(EIII):&= \langle (G_{1,t}-M_{1,t})\rangle\langle \Im G_{1,t} A_1\Im G_{2,t}A_2G_{1,t}\rangle\dd t + \langle (G_{2,t}-M_{2,t})\rangle\langle \Im G_{2,t} A_2\Im G_{1,t}A_1G_{2,t}\rangle\dd t \\
			&\quad+ \langle (G_{1,t}-M_{1,t})^*\rangle\langle \Im G_{1,t} A_1\Im G_{2,t}A_2 G_{1,t}^*\rangle\dd t + \langle (G_{2,t}-M_{2,t})^*\rangle\langle \Im G_{2,t}A_2 \Im G_{1,t}A_1G_{2,t}^*\rangle\dd t \\
			&\quad+\left(\frac{\langle \Im G_{1,t}- \Im M_{1,t}\rangle}{\eta_{1,t}}+\frac{\langle \Im G_{2,t}-\Im M_{2,t}\rangle}{\eta_{2,t}}\right)\langle \Im G_{1,t}A_1\Im G_{2,t}A_2\rangle .
		\end{split}
	\end{equation}
	We point out that here we used that by the chiral symmetry of $H^z$ for any diagonal $B_1,B_2$ we have $\langle G_1B_1\Im G_2B_2\rangle=\ii\langle\Im G_1B_1\Im G_2B_2\rangle$.
	
	We now present the estimate of all the various terms in the rhs. of \eqref{eq:imflow}. These estimates are very similar to the ones presented in \cite[Section 5.1]{gumbel}, hence we omit some details. Throughout the proof we use that $\widehat{\gamma}_s\gtrsim \widehat{\gamma}_t$, $\gamma_s\gtrsim \gamma_t$, for $s\le t$, and \eqref{eq:explint} even if not mentioned explicitly. For simplicity of notation we also assume that $\eta_{i,t}>0$.
	
	Define
	the stopping time
	\begin{equation}
		\label{eq:deftau3}
		\tau_3:=\inf\left\{t\ge 0: \max_{A_1,A_2\in\{E_+,E_-\}} \big|\langle (\Im G_{1,t}A_1\Im G_{2,t}-\widehat{M}_{12,t}^{A_1})A_2\rangle\big|=n^{2\xi} \widehat{\alpha}_t\right\}\wedge T^*, \qquad  \widehat{\alpha}_t:= \frac{1}{n|\eta_{1,t}\eta_{2,t}|}\wedge \frac{\rho_{1,t}\rho_{2,t}}{\sqrt{n\ell_t}\widehat{\gamma}_t}.
	\end{equation}
	Recall (see \eqref{Im_M_bound} of Lemma \ref{lemma_M}) that for the deterministic term we have the bound
	\begin{equation}
		\label{eq:hatMbound}
		\lVert \widehat{M}_{12,t}^A\rVert\lesssim \frac{1}{\widehat{\gamma}_t}.
	\end{equation}
	
	We start with the estimate of the stochastic term in \eqref{eq:imflow}. Its quadratic variation process is bounded by
	\begin{equation}
		\frac{1}{n^2\eta_{*,t}^2}\langle \Im G_{1,t}A_1\Im G_{2,t}A_2\Im G_{1,t}A_2\Im G_{2,t}A_1\rangle\dd t\lesssim \left(\frac{\rho_{1,t}}{n^2\eta_{1,t}^3\eta_{2,t}^2}\wedge \frac{\rho_{1,t}^2\rho_{2,t}^2}{n\eta_{*,t}^2\widehat{\gamma}_t^2}\right)\,\dd t,
	\end{equation}
	where to obtain the second term in the rhs. we used the averaged reduction inequality \eqref{eq:redina} to estimate the trace of the product of four $G$'s in terms of the product of two traces involving two $G$'s each. To estimate the traces involving two resolvents we then used \eqref{eq:hatMbound} and the definition of the stopping time $\tau_3$ in \eqref{eq:deftau3}. By the BDG inequality we thus obtain
	\begin{equation}
		\label{eq:BDG3}
		\sup_{0\le r\le t}\left|\int_0^{r\wedge\tau_3} \sum_{a,b=1}^{2n}\partial_{ab}\langle \Im G_{1,s}A_1 \Im G_{2,s}A_2\rangle\frac{\dd B_{ab,s}}{\sqrt{n}}\right|\lesssim n^\xi \widehat{\alpha}_{t\wedge\tau_3}.
	\end{equation}
	
	Notice that the second term in the second line of \eqref{eq:imflow} can be neglected, since it just amounts to a simple rescaling of a factor $e^t\sim 1$. Next, using that by spectral decomposition we have
	\begin{equation}
		\big|\langle \Im G_{1,t} A_1\Im G_{2,t}A_2G_{1,t}\rangle\big|\lesssim \frac{1}{\eta_{1,t}}\langle \Im G_{1,t} A_1\Im G_{2,t}A_1^*\rangle^{1/2}\langle \Im G_{1,t} A_2\Im G_{2,t}A_2^*\rangle^{1/2},
	\end{equation}
	and using the single resolvent local law $|\langle G_{i,t}-M_{i,t}\rangle|\le n^\xi/(n\eta_{i,t})$, we estimate the term in $(EIII)$ by
	\begin{equation}
		\int_0^{t\wedge\tau_3}\frac{n^\xi}{n\eta_{*,s}^2} \frac{\rho_{1,s}\rho_{2,s}}{\widehat{\gamma_s}}\, \dd s\lesssim \frac{n^\xi}{n\ell_{t\wedge\tau_3}} \cdot \frac{\rho_{1,t\wedge\tau_3}\rho_{2,t\wedge\tau_3}}{\widehat{\gamma_t}}\le \widehat{\alpha}_{t\wedge\tau_3},
	\end{equation}
	where in the last inequality we also used that $\xi\le\epsilon/10$.
	
	Next, using the definition of the stopping time $\tau_3$ in \eqref{eq:deftau3} and the local law
	\begin{equation}
		\label{eq:noimllaw}
		\big|\langle \big(G_{1,t}A_1G_{2,t}-M_{12,t}^{A_1}\big)A_2\rangle\big|\lesssim \frac{n^\xi}{n\eta_{1,t}\eta_{2,t}}\wedge \frac{n^\xi}{\sqrt{n\ell_t}\gamma_t},
	\end{equation}
	we estimate the terms in $(EII)$ by
	\begin{equation}
		\int_0^{t\wedge\tau_3} \left(\frac{n^\xi}{n\eta_{1,s}\eta_{2,s}}\wedge \frac{n^\xi}{\sqrt{n\ell_s}\gamma_s}\right) \left(\frac{n^\xi}{n\eta_{1,s}\eta_{2,s}}\wedge \frac{n^\xi \rho_{1,s}\rho_{2,s}}{\sqrt{n\ell_s}\widehat{\gamma}_s}\right)\, \dd s\lesssim \frac{n^{2\xi}}{n\ell_{t\wedge\tau_3}}\widehat{\alpha}_{t\wedge\tau_3}.
	\end{equation}
	Furthermore, using \eqref{eq:hatMbound} and \eqref{eq:noimllaw}, the terms in $(EI)$ are estimated by
	\begin{equation}
		\label{eq:g1g2terms}
		\int_0^t \frac{\rho_{1,s}\rho_{2,s}}{\widehat{\gamma}_s} \cdot\left(\frac{n^\xi}{n\eta_{1,s}\eta_{2,s}}\wedge \frac{n^\xi}{\sqrt{n\ell_s}\gamma_s}\right)\, \dd s\lesssim  n^\xi \widehat{\alpha}_t.
	\end{equation}
	
	We are thus left only with the linear terms $(LT)$ in the rhs. of \eqref{eq:imflow}. These terms will be Gronwalled, giving the desired bound. This last step is the only place where there is a relevant difference compared to \cite[Section 5]{gumbel}. In fact, in \cite{gumbel} when we considered $G_1G_2$, to deal with the two bad directions for the stability operator $E_-,E_+$, we studied a $4\times 4$ system of equations, corresponding to the four choices $A_1,A_2\in \{E_+,E_-\}$. Instead, for $\Im G_1\Im G_2$, due to additional symmetries of $\Im G$ (see \eqref{eq:symmetry}) it is enough to study a scalar equation. Indeed, by the chiral symmetry of $H^z$ we have
	\begin{equation}
		\label{eq:symmetry}
		\langle \Im G_{1,t}E_-\Im G_{2,t}E_-\rangle=\langle \Im G_{1,t} \Im G_{2,t}\rangle, \quad \langle \Im G_{1,t}E_-\Im G_{2,t}\rangle=\langle \Im G_{1,t} \Im G_{2,t}E_-\rangle, \quad \langle\widehat{M}_{12,t}^{E_\pm}E_\mp\rangle=0.
	\end{equation}
	This shows that it is enough to only consider two cases: $(A_1,A_2)=(E_+,E_+)$ and $(A_1,A_2)=(E_+,E_-)$. We first consider the case $(A_1,A_2)=(E_+,E_+)$. By inspecting \eqref{eq:imflow}, one can easily see that the equation for $(A_1,A_2)=(E_+,E_+)$ is self--contained; in fact, the terms containing $(E_+,E_-)$, which naturally emerge in the rhs. of \eqref{eq:imflow}, cancel algebraically. More precisely, defining
	\[
	\widehat{Y}_t:=\big|\langle (\Im G_{1,t}E_+\Im G_{2,t}-\widehat{M}_{12,t}^{E_+})E_+\rangle\big|,
	\]
	and combining \eqref{eq:BDG3}--\eqref{eq:g1g2terms}, we obtain
	\begin{equation}
		\widehat{Y}_{t\wedge\tau_3}=\widehat{Y}_0+\frac{1}{2}\int_0^t \big[\langle M_{12,s}^I\rangle+\langle M_{\overline{1}\overline{2},s}^I\rangle+\langle M_{\overline{1}2,s}^I\rangle+\langle M_{1\overline{2},s}^I\rangle\big] \widehat{Y}_s\,\dd s+ n^\xi\widehat{\alpha}_{t\wedge\tau_3}.
	\end{equation}
	By a Gronwall inequality, using that by explicit computations (see e.g. \cite[Lemma 5.7]{gumbel})
	\begin{equation}
		\exp\left(\frac{1}{2}\int_s^t \big[\langle M_{12,r}^I\rangle+\langle M_{\overline{1}\overline{2},r}^I\rangle+\langle M_{\overline{1}2,r}^I\rangle+\langle M_{1\overline{2},r}^I\rangle\big]\, \dd r\right)\le \exp\left(\int_s^t \big|\langle M_{12,r}^I\rangle\big|\,\dd r\right) \lesssim \frac{\gamma_s^2}{\gamma_t^2}\sim  \frac{\widehat{\gamma}_s^2}{\widehat{\gamma}_t^2},
	\end{equation}
	and proceeding similar arguments as in the end of \cite[Secion 5.1]{gumbel} starting from Eq. (5.58) therein, we obtain $\widehat{Y}_{t\wedge\tau_3}\le n^\xi\widehat{\alpha}_{t\wedge\tau_3}$. This concludes the proof for $(A_1,A_2)=(E_+,E_+)$. Given the result for $(A_1,A_2)=(E_+,E_+)$, the estimate in the case  $(A_1,A_2)=(E_+,E_-)$ is completely analogous obtaining the bound $n^\xi\widehat{\alpha}_{t\wedge\tau_3}$ in this case as well. This concludes the proof showing that $\tau_3=T^*$ with very high probability.
	
\end{proof}

\section{The zag step: Proofs of Theorem \ref{thm:2G} and Theorem \ref{thm:2G_edge}}
\label{sec:zag}

In Section \ref{sec:zig} (zig step), we proved in Proposition \ref{prop:aveflow} that if the local laws from Theorem~\ref{thm:2G} hold on a global scale $\eta\sim 1$, then they also hold down to the optimal local scale $\eta\gg 1/n$, at the cost of adding a Gaussian component. This section (zag step) is devoted to remove this added Gaussian component using a Gronwall argument (see Lemma \ref{lemma_RS_gft} and (\ref{gft_RR}) below) and hence prove Theorem \ref{thm:2G}. The proof of Theorem \ref{thm:2G_edge} near the edge is similar to Theorem \ref{thm:2G} using (\ref{gft_im}) below to remove the Gaussian component added in Proposition \ref{prop:aveflowim}.

The outline of this section is as follows. We first present the proof of the isotropic local law (\ref{enlaw12}) and the averaged law in Section \ref{sec:proof_2G_iso} and \ref{sec:2G_ave} respectively. In Sections \ref{sec:nogain}, we prove standard local laws without $|z_1-z_2|$ decay stated in Proposition \ref{prop_initial} below, which serve as an input to prove Theorem \ref{thm:2G}. The proof of Theorem \ref{thm:2G_edge} is presented in Section \ref{sec:im} using Theorem \ref{thm:2G_edge} as an input.

\medskip

Before we prove Theorem \ref{thm:2G}, we introduce the local laws for $G^{z_1}(\ii \eta_1) A_1 G_{2}^{z_2}(\ii \eta_2)$ without gaining from the $|z_1-z_2|$-decorrelations. The proof of these local laws is standard and thus postponed to Section~\ref{sec:nogain}.

\begin{proposition}\label{prop_initial}
	Fix small constants $\epsilon,\tau>0$. The following hold
	\begin{align}
		&\Big| \big(G^{z_1}(\ii \eta_1) A_1 G^{z_2}(\ii \eta_2)-M_{12}^{A_1}\big)_{\xx \yy}\Big| \prec \frac{\sqrt{\rho^*}}{\sqrt{n} \eta_*^{3/2}}, \label{initial_en}\\
		&\Big|\big\langle \big(G^{z_1}(\ii \eta_1) A_1 G^{z_2}(\ii \eta_2)-M_{12}^{A_1}\big)A_2\big\rangle\Big|\prec \frac{1}{n\eta_*^2},\label{initial_ave}
	\end{align}
	uniformly for any $|z_i| \leq 1+\tau$, $\min_{i=1}^2 \{|\eta_i| \rho_i \} \geq n^{-1+\epsilon}$, any deterministic matrices $\|A_1\|+\|A_2\|\lesssim 1$ and unit vectors $\xx, \yy$. Here $\rho_i : =\rho^{z_i}(\ii \eta_i)$,
	 $\eta_*:=\min_{i=1}^2|\eta_i|$, $\rho^*:=\max_{i=1}^2\rho_i$, and $M_{12}^{A_1}$ is given in (\ref{eq:defM12}). 
	 Moreover, recalling $M_{121}^{A_1,A_2}$ defined in (\ref{eq:defM121}), we also have the following estimate for three resolvents
	 \begin{align}
	 \Big|\big(G^{z_1}(\ii \eta_1)A_1G^{z_2}(\ii \eta_2)A_2G^{z_1}(\ii \eta_1)-M_{121}^{A_1,A_2}\big)_{\xx \yy}\Big| \prec \frac{1}{\eta_*^2}.\label{initial_3G}
	 \end{align}\nc
\end{proposition}
Next we aim to improve Proposition \ref{prop_initial} to Theorem \ref{thm:2G} gaining from the $|z_1-z_2|$-decorrelation.

\subsection{Proof of the isotropic local law in  (\ref{enlaw12})}\label{sec:proof_2G_iso}

We start with proving the isotropic local law in (\ref{enlaw12}), using Proposition \ref{prop_iteration} below iteratively.
In this iteration procedure we gradually
  replace a small power of 
$\eta_*/\rho^*$ with $\gamma$. 
 We will focus on the regime that $\gamma \gtrsim  \eta_*/\rho^*$, otherwise the desired local law follows directly from Proposition \ref{prop_initial}.

\begin{proposition}\label{prop_iteration}
	Fix small constants $\epsilon,\tau>0$.  Assume that for any $z_i \in \C$ and $\eta_i \in \R$ satisfying $|z_i| \leq 1+\tau$, $\min_{i=1}^2 \{|\eta_i| \rho_i \} \geq n^{-1+\epsilon}$, $\gamma \gtrsim  \eta_*/\rho^*$, the following hold for some $0\leq b < 1$, \ie
	\begin{align}\label{assump_R}
		&\Big| \big(G^{z_1}(\ii \eta_1) A_1 G^{z_2}(\ii \eta_2)-M_{12}^{A_1}\big)_{\xx \yy}\Big|\prec\frac{(\rho^*)^{\frac{1-b}{2}}}{\sqrt{n}(\eta_*)^{\frac{3-b}{2}} \gamma^{\frac{b}{2}}},
	\end{align}
	\begin{align}
		&\Big|\big(G^{z_1}(\ii \eta_1)A_1G^{z_2}(\ii \eta_2)A_2G^{z_1}(\ii \eta_1)-M_{121}^{A_1,A_2}\big)_{\xx \yy}\Big| \prec \frac{1}{\eta_*^{2-b} \gamma^{b}},\label{assump_S}
	\end{align}
	  uniformly for any deterministic bounded  matrices $A_i~(i=1,2)$ and unit vectors $\xx, \yy$, where $\gamma$ is defined in (\ref{parameter}) and  $M_{12}^{A_1}$,
	$M_{121}^{A_1,A_2}$ are defined in (\ref{eq:defM12}), (\ref{eq:defM121}).
		Then the following also hold for a larger  $b'=\min\big\{b+\theta,~1\big\}$ with some $0<\theta\leq \epsilon/10$,
	\begin{align}\label{goal_R}
		&\Big| \big(G^{z_1}(\ii \eta_1) A_1 G^{z_2}(\ii \eta_2)-M_{12}^{A_1}\big)_{\xx \yy}\Big|\prec\frac{(\rho^*)^{\frac{1-b'}{2}}}{\sqrt{n}(\eta_*)^{\frac{3-b'}{2}} \gamma^{\frac{b'}{2}}},
	\end{align}
	\begin{align}\label{goal_S}
		&\Big|\big(G^{z_1}(\ii \eta_1)A_1G^{z_2}(\ii \eta_2)A_2G^{z_1}(\ii \eta_1)-M_{121}^{A_1,A_2}\big)_{\xx \yy}\Big| \prec \frac{1}{\eta_*^{2-b'} \gamma^{b'}}.
	\end{align}
  uniformly for any deterministic bounded  matrices $A_i~(i=1,2)$ and unit vectors $\xx, \yy$. 
\end{proposition}

\begin{proof}[Proof of (\ref{enlaw12})]
 Note that Proposition \ref{prop_initial} implies that the initial assumptions in (\ref{assump_R})-(\ref{assump_S}) with $b=0$ are satisfied. Starting from (\ref{assump_R})-$(\ref{assump_S})$ for $b=0$ as the initial step and iterating 
Proposition \ref{prop_iteration} for at most $O(\theta^{-1})$
 times by gradually increasing $b$ from 0 to $1$, we have proved the isotropic local law for $G^{z_1}A_1G^{z_2}$ in (\ref{enlaw12}) modulo Proposition \ref{prop_iteration} that we will prove shortly. For a graphical illustration of this iterative procedure we refer to Figure~\ref{fig:iterate}.
 \end{proof}
 

 Note that as a by-product, we also obtained
 the following isotropic version for three resolvents. This
  Lemma \ref{lemma:3G} will be needed as an auxiliary estimate 
  in the next section to prove the averaged local law in (\ref{avlaw12}). 
\nc
\begin{lemma}\label{lemma:3G}
	Fix small constants $\epsilon,\tau>0$. For any deterministic matrices $\|A_1\|+\|A_2\| \lesssim 1$ and vectors $\|\xx\|, \|\yy\| \lesssim 1$, the following holds
	\begin{equation}
		\label{eq:3giso}
		\Big| \big\langle \xx, \big( G^{z_1}(\ii \eta_1)A_1G^{z_2}(\ii \eta_2)A_2G^{z_1}(\ii \eta_1)-M_{121}^{A_1,A_2}\big) \yy \big \rangle \Big|\prec \frac{1}{\eta_*\gamma}.
	\end{equation}
	uniformly for any $ |z_i|\leq 1+\tau$ and $\min_{i=1}^2 \{|\eta_i| \rho_i \} \geq n^{-1+\epsilon}$. 
\end{lemma}

 Before we give the actual proof of Proposition \ref{prop_iteration}, we introduce the following short-hand notations and a preliminary lemma for zag step (see Lemma \ref{lemma_RS_gft} below). \nc Define
\begin{align}\label{RS_def}
	R^{xy}:=\big(G_1 A G_2-M^{A}_{12}\big)_{\xx \yy}, \qquad S^{xy}:=\big(G_{1}A_1G_{2}A_2G_{1}-M^{A_1,A_2}_{121}\big)_{\xx \yy},
\end{align}
with $G_i:=G^{z_i}(\ii \eta_i)$, $i=1,2$,
and set the following control parameters 
\begin{align}\label{EE}
	&\mathcal{E}_{\mathfrak{b}}:=\frac{(\rho^*)^{\frac{1-\mathfrak{b}}{2}}}{\sqrt{n}(\eta_*)^{\frac{3-\mathfrak{b}}{2}} \gamma^{\frac{\mathfrak{b}}{2}}}, \qquad \qquad  \wt{\mathcal{E}}_{\mathfrak{b}}:=\frac{1}{\eta_*^{2-\mathfrak{b}} \gamma^\mathfrak{b}}, \qquad \qquad  \mathfrak{b}\in [0,1].
\end{align}
It is easy to check that, for any $0\leq b\leq b'\leq 1$,
\begin{align}\label{para_bound}
	\frac{1}{\sqrt{n\gamma }\eta_*} \lesssim \mathcal{E}_{b'} \leq \mathcal{E}_b \leq \frac{\sqrt{\rho^*}}{\sqrt{n} \eta_*^{3/2}}, \qquad \qquad \frac{1}{\eta_* \gamma} \lesssim \wt{\mathcal{E}}_{b'} \leq \wt{\mathcal{E}}_b \leq \frac{1}{\eta_*^2}.
\end{align}
With the above short-hand notations, we aim to prove, for any $z_i \in \C$ and $\eta \in \R$ satisfying $|z_i|\leq 1+\tau$, $n|\eta_i| \rho_i \gtrsim n^{\epsilon}$, and $\gamma\gtrsim \eta_*/\rho^*$},
\begin{align}\label{simple}
	|R^{xy}| \prec \mathcal{E}_{b'}, \qquad |S^{xy}| \prec \wt{\mathcal{E}}_{b'}, \qquad b'=\min\big\{b+\theta,~1\big\},
\end{align}
assuming that $|R^{xy}| \prec \mathcal{E}_b$ and $ |S^{xy}| \prec \wt{\mathcal{E}}_b$ for some $0\leq b<1$.

 The proof of Proposition \ref{prop_iteration} will be a Gronwall argument [zag step] together with Proposition \ref{prop:isoflow} [zig step]. The bounds \eqref{goal_R} and \eqref{goal_S}
are already  known from Proposition \ref{prop:isoflow} (together with the global laws needed as an initial input
that will be stated later in \eqref{eq:englobal}--\eqref{eq:englobal3G}) 
 whenever  $X$ has a small  Gaussian component. We show that this Gaussian component
does not change the estimate up to the relevant order by running an Ornstein-Uhlenbeck flow
to remove the small Gaussian component.
We will work with the time evolution of  high moments of $R^{xy}_t, S^{xy}_t$.  The key of the Gronwall argument
is to give an estimate on the time derivative of these high moments in terms of themselves, this will be
the content of Lemma \ref{lemma_RS_gft} below. Similarly to the previous section, to keep the presentation simple, we now focus only on the complex case.

We now  setup this flow argument.
	Given the initial ensemble $H^{z}$ in (\ref{def_G}) for fixed $z\in \C$ and $X$ satisfying Assumption \ref{ass:mainass}, we consider the following Ornstein-Uhlenbeck matrix flow
\begin{align}\label{flow}
	\dd H^z_t=-\frac{1}{2} (H^z_t+Z)\dd t+\frac{1}{\sqrt{n}} \dd \mathcal{B}_t, \quad Z:=\begin{pmatrix}
		0  &  zI  \\
		\overline{z}I   & 0
	\end{pmatrix}, \quad 
	\mathcal{B}_t:=\begin{pmatrix}
		0  &  B_t  \\
		B^*_t   & 0
	\end{pmatrix},
\end{align}
with initial condition $H^z_{t=0}= H^z$, where $B_t$ is an $n \times n$ matrix with i.i.d. standard complex valued Brownian motion entries.  Indeed, the matrix flow $H_t^z$ interpolates between the initial matrix $H^{z}$
and the same  matrix 
with $X$ replaced with an independent complex Ginibre ensemble, \ie
\begin{align}\label{Htt}
	X_t \stackrel{\mathrm{d}}{=} e^{-\frac{t}{2}} X + \sqrt{1-e^{-t}} \mathrm{Gin}(\C).
\end{align}
The corresponding time-dependent resolvent of $H_t^{z_i}~(i=1,2)$ is denoted by 
$$ G_i:=G^{z_i}_t(\ii \eta_i)=(H^{z_i}_t-\ii \eta_i)^{-1}, \qquad i=1,2,$$
which we use to define time dependent $R^{xy}_t$, $S^{xy}_t$ as in (\ref{RS_def}).  For notational brevity, we will not keep track of the dependence of $t$ in the notations $G_i$ from now on. \nc Note that the parameters $z_i\in \C$, $\eta_i \in \R$ are fixed and hence the deterministic matrices $M^{A}_{12}$ and $M^{A_1,A_2}_{121}$ in (\ref{RS_def}) are independent of the time $t$. 

\begin{lemma}\label{lemma_RS_gft}
	Fix small constant $\epsilon,\tau>0$ and  any integer $m\geq 1$.
	 For any  $|z_i| \leq 1+\tau$ and $\eta_i \in \R$ satisfying $ \min_{i=1}^2 \{|\eta_i| \rho_i\} \geq n^{-1+\epsilon}$ and $\gamma \gtrsim \eta_*/ \rho^* $, 	
	assuming\footnote{We remark that stochastic domination $\prec$ 
	and upper bounds of arbitrarily high moments of a random variable 
	are equivalent for random variables bounded by some $n^C$
	 which is always true in our case, \ie  for $|X|\le n^C$ we have
	 that $X \prec \Psi$ if and only if  $\E|X|^{2m} \le n^\xi \Psi^{2m}$ for any $m\geq 1$ and $\xi>0$.
	  We frequently use this fact  in the proof without mentioning it. \nc}
	that for some $0\leq b<1$, 
	\begin{align}  
		&\E|R_t^{xy}|^{2m} \prec (\mathcal{E}_b)^{2m}, \quad \qquad  R_t^{xy}=\big(G_1 A G_2-M^{A}_{12}\big)_{\xx \yy}, \nonumber\\
		&\E|S_t^{xy}|^{2m} \prec (\wt{\mathcal{E}}_b)^{2m},  \quad \qquad  S_t^{xy}=\big(G_{1}A_1G_{2}A_2G_{1}-M^{A_1,A_2}_{121}\big)_{\xx \yy},   \nonumber
	\end{align}
	  uniformly for any $t\geq 0$, any deterministic bounded  matrices $A_i~(i=1,2)$ and unit vectors $\xx, \yy$,  then  the following hold for $b'=\min\big\{b+\theta,~1\big\}$, $0<\theta\leq \epsilon/10$, and the prefactor $C_n:=1+\frac{1}{\sqrt{n} \eta_*^{1+\frac{\epsilon}{2}}}$,
	\begin{align}\label{gft_1}
		\Big|\frac{\dd \E |R^{xy}_t|^{2m}}{\dd t}\Big| \lesssim & C_n
		\Big(  \E \big|R^{xy}_t\big|^{2m} +O_\prec\big(\mathcal{E}_{b'}\big)^{2m}\Big),
	\end{align}
	\begin{align}\label{gft_2}
		\Big|\frac{\dd \E \big|\mathcal{S}^{xy}_t\big|^{2m}}{\dd t}\Big| \lesssim & C_n
		\Big(  \E \big|\mathcal{S}^{xy}_t\big|^{2m} +O_\prec\big(\wt{\mathcal{E}}_{b'}\big)^{2m}\Big),
	\end{align}
	 uniformly for any $t\geq 0$, any deterministic bounded  matrices $A_i~(i=1,2)$ and unit vectors $\xx, \yy$. 
\end{lemma}

The proof of Lemma \ref{lemma_RS_gft} will be postponed till the end of this section. Now we are ready to use Proposition \ref{prop:isoflow} (\emph{zig-step}) and Lemma \ref{lemma_RS_gft}  (main technical input for \nc \emph{zag-step})  
to prove Proposition \ref{prop_iteration}.  

\begin{proof}[Proof of Proposition \ref{prop_iteration}]

	Recall that the matrix flows $X_t$, $W_t$ defined in (\ref{eq:OUmat})-(\ref{eq:hermflow}) and the charateristic flow $\Lambda(t):=\Lambda_t$ defined in (\ref{eq:charflowmat}).  The corresponding resolvent is given by $G_t^{z_t}(\ii \eta_t)=(H_t^{z_t}-\ii \eta_t)^{-1}$, where $H_t^{z_t}$ is defined as in (\ref{def_H}) with time dependent $z_t$ and $X_t$. We remark that one should not confuse this notation with $G^{z}_t(\ii \eta)=(H_t^z-\ii \eta)^{-1}$ defined below (\ref{Htt}), where both $\eta\in \R$ and $z \in \C$ are fixed.

	Fix any $|z_i| \leq 1+\tau$ and $\eta_i$ with $n |\eta_i| \rho_i \gtrsim n^{\epsilon}$
	where $\rho_i:= \rho^{z_i}(\ii\eta_i)$.  Choose a  small $T\sim 1$. Using Lemma \ref{lem:ODEtheo}, there exist initial conditions $\Lambda_i(0)$  satisfying $\mathrm{dist}(\ii \eta_i(0), \mathrm{supp}(\rho_i(0))) \gtrsim 1$ such that at time $t=T$, $z_i(T)=z_i$ and $\eta_i(T)=\eta_i$. We also remark that from Lemma \ref{lem:charprop}, the initial conditions $\eta_i(0), z_i(0), \rho_i(0)$ have the following properties: 
	\begin{align}\label{property}
		\frac{\eta_i(0)}{\rho_i(0)} \gtrsim T, \qquad \rho_i(0) \sim \rho_i(T), \qquad z_i(0)=z_i(T)(1+O(T)).
	\end{align}
	We next state the following global laws, \ie  for any 
	spectral parameters with $ \mathrm{dist}(\ii \eta_i, \mathrm{supp}(\rho^{z_i})) \sim 1$,
	\begin{equation}
		\label{eq:englobal}
		\Big| \big(G^{z_1}(\ii \eta_1)A_1G^{z_2}(\ii \eta_2)-M_{12}^{A_1}\big)_{\xx \yy}\Big|\prec \frac{1}{\sqrt{n}},
	\end{equation}
	\begin{equation}
		\label{eq:englobal3G}
		\Big| \big(G^{z_1}(\ii \eta_1)A_1G^{z_2}(\ii \eta_2) A_2 G^{z_1}(\ii \eta_1)-M_{121}^{A_1,A_2}\big)_{\xx \yy}\Big|\prec 1.
	\end{equation}
	The proof of these global laws is standard and similar to \cite[Proposition 4.1]{CEHK24} (see also \cite[Appendix B]{Optimal_local}) with the only exception that in the current case we need to consider the $2\times 2$ block structure due to the chiral symmetry, but this can be done with very minor changes in the proof, and so omitted.  These global laws directly imply that the initial condition (\ref{eq:inassiso}) of Proposition \ref{prop:isoflow} at $t=0$ is satisfied, hence from this proposition we obtain the desired local laws at $t=T \sim 1$, \ie
	\begin{align}\label{zero}
		|R_T^{xy}| \prec \frac{1}{\sqrt{n \gamma}\eta_*}, \qquad   |S_T^{xy}| \prec \frac{1}{\eta_*\gamma}, \qquad \qquad  \eta_i(T)=\eta_i, \quad z_i(T)=z_i.
	\end{align}
	This completes the zig step and it proves Theorem \ref{thm:2G} for $X$ being a Gaussian divisible matrix, 
	albeit with a large Gaussian component, 
	\ie $X_T\stackrel{\mathrm{d}}{=} e^{-\frac{T}{2}} X + \sqrt{1-e^{-T}} \mathrm{Gin}(\C)$, with $T\sim 1$.

	Next we aim to use Lemma \ref{lemma_RS_gft} to remove the Gaussian component $T \sim 1$.
         However the differential inequalities in (\ref{gft_1})-(\ref{gft_2}) are only effective in the Gronwall argument
	 when the
	 pre--factor $C_n$ is bounded, 
	 which means that $\eta_*$ needs to be
	  somewhat large. So we will need an iterative procedure to handle the smaller $\eta_*$ regime using
	  a smaller Gaussian component, hence smaller $T$, so that the larger prefactor $C_n$ is compensated
	  by the shorter integration time. 
	  
	  In the first step, we assume that $|\eta_*| \gtrsim n^{-1/2+\epsilon/2}$, then the prefactor $C_n \leq C$ for some constant $C>0$. Using Gronwall's inequality for $T\sim 1$, we have
	\begin{align}\label{Gronwall}
		\E |R_0^{xy}|^{2m} \leq e^{\int_{0}^{T} C \dd t} \E |R_T^{xy}|^{2m} + \int_{0}^{T} e^{\int_{s}^{T} C \dd t} O_\prec\big({\mathcal{E}}_{b'}\big)^{2m} \dd s =O_\prec \big({\mathcal{E}}_{b'}\big)^{2m},
	\end{align}	 
	where we also used (\ref{zero}) and  (\ref{para_bound}). One can handle $S_0^{xy}$ in a similar way.  This hence proves (\ref{simple}) for any $|z_i|\leq 1+\tau$ and $|\eta_*| \gtrsim n^{-1/2+\epsilon/2}$, \ie	\begin{align}\label{goal_k}
		\big|R^{xy}_{0}\big| \prec \mathcal{E}_{b'},  \qquad  \big|S^{xy}_0\big| \prec \wt{\mathcal{E}}_{b'}, \qquad \qquad \forall |\eta_i| \gtrsim n^{-1/2+\epsilon/2}.
	\end{align}
This completes the zag step. 
	
\begin{figure}
	\includegraphics[width=8cm]{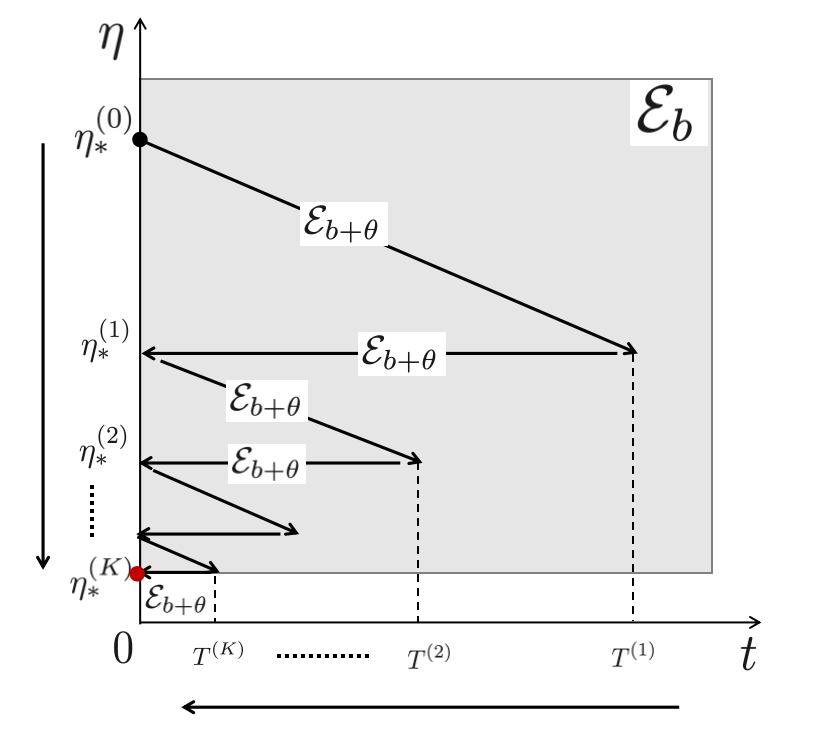}
	\centering
	\caption{The shaded square indicates the region in the $(\eta,t)$ plane where the local law holds with an error bound $\mathcal{E}_b$, as given in assumption (\ref{assump_R}). Starting from the global law at the initial black dot ($\eta_*^{(0)} \sim 1)$, the slanted line downward indicates the zig step using Proposition \ref{prop:isoflow}, while the horizontal line leftward indicates the zag step using Lemma \ref{lemma_RS_gft} and a Gronwall argument. Iterating finitely many zig-zags  until reaching the final red point ($\eta_*^{(K)} \sim \eta_*$, $K=O(\epsilon^{-1})$)  yields the improved local law with an improved error $\mathcal{E}_{b+\theta}$ for any $\min_{i=1}^{2}|\eta_i| \rho_i \gtrsim n^{-1+\epsilon}$.\nc} 
	\label{fig:zigzag}
\end{figure}

	To extend the above results to smaller $|\eta_*| \lesssim n^{-1/2+\epsilon/2}$, we use zig-zag iteratively as follows (see the illustration in Figure \ref{fig:zigzag} below). In each iteration step $k\geq 1$, we use the superscript $(k)$ to indicate the step number, \eg we denote the spectral parameters by $\eta^{(k)}_i \in \R$, $z^{(k)}_i \in \C$ and the time by $T^{(k)} \in \R^+$. Recall that our final goal is to prove (\ref{simple}) for any fixed $|z_i|\leq 1+\tau$ and $n |\eta_i| \rho_i \gtrsim n^{\epsilon}$.  We also recall from Lemma \ref{lem:charprop} that 
	the size of $\rho_i$ does not change along the characteristic flow in (\ref{eq:charflowmat}) by more than a constant factor.
	Note that (\ref{rho}) and $n |\eta_i| \rho \gtrsim n^{\epsilon}$ imply that $|\eta_i| \gtrsim n^{-1+\epsilon}$ and $n^{-1/2+\epsilon/2} \lesssim \rho_i \lesssim 1$. Then we set $\rho_*=\min_{i=1}^2 \rho_i\sim n^{-c_*}$ with $0\leq c_* \leq 1/2-\epsilon/2$. In general for any $k\geq 1$, we define the scales recursively as follows:
	\begin{align}\label{eta_k}
		\eta_*^{(k)}:=n^{-a_{k}}, \qquad a_{k}=\frac{\frac{1}{2}+a_{k-1}-c_*}{1+\frac{\epsilon}{2}}, \qquad a_1=\frac{1}{2}-\frac{\epsilon}{2}.
	\end{align}
	We will stop the iterations if $\eta^{(k)}_*$ reaches $\eta_*=\min_{i=1}^{2}|\eta_i| \gtrsim n^{-1+\epsilon}$, so 
	$a_k \ge  1-\epsilon$.  The number of iteration steps is denoted by $K$. \nc It is straightforward to check that $a_k$ is increasing in $k$ 
	 at least up to $1$ using that $0\leq c_* \leq 1/2-\epsilon/2$.

	In the $k$-th step, we fix $|\eta_i^{(k)}| \gtrsim \eta_*^{(k)}=n^{-a_{k}}$ and $|z^{(k)}_i|\leq 1+\tau$ with $\rho_i^{(k)} \sim \rho_i$. Assume that in the previous step, we have proved (\ref{goal_k}) for any $|\eta^{(k-1)}_i| \gtrsim \eta_*^{(k-1)}=n^{-a_{k-1}}$ and $|z^{(k-1)}_i| \leq 1+\tau$ with $\rho_i^{(k-1)} \sim \rho_i$. Recall the characteristic flow $\Lambda(t)=\Lambda_t$ in (\ref{eq:charflowmat}) and choose $T^{(k)} \sim \eta_*^{(k-1)}/\rho_*=n^{-a_{k-1}+c_*}$.  From Lemma \ref{lem:charprop}, there exist initial conditions $|z_i(0)|=|z_i^{(k)}|+O(T^{(k)})$ and $|\eta_i(0)| \gtrsim  \rho_*T^{(k)}=\eta_*^{(k-1)} = n^{-a_{k-1}}$ such that $\eta_i(T^{(k)})=\eta^{(k)}_i$ and $ z_i(T^{(k)})=z^{(k)}_i$. The result from the previous step 
	 in \eqref{goal_k} ensures that the assumption (\ref{eq:inassiso}) of Proposition \ref{prop:isoflow}  (with $b’$ 
	 playing the role of $b$)
  is satisfied for any $|\eta^{(k-1)}_i| \gtrsim \eta_*^{(k-1)}=n^{-a_{k-1}}$. 
	Thus  Proposition \ref{prop:isoflow} implies that for $t=T^{(k)} \sim n^{-a_{k-1}+c_*}$,  
	\begin{align}
		\big|R^{xy}_{T^{(k)}}\big| \prec \mathcal{E}_{b'},  \qquad  \big|S^{xy}_{T^{(k)}}\big| \prec \wt{\mathcal{E}}_{b'}, \qquad z_i(T^{(k)})=z^{(k)}_i,\quad \eta_i(T^{(k)})=\eta^{(k)}_i,
	\end{align}
	with $|z^{(k)}_i| \leq 1+\tau$ and $|\eta^{(k)}_i|\gtrsim \eta^{(k)}_*= n^{-a_{k}}$ such that $\rho^{(k)}_i \sim \rho_i$. Then we apply Gronwall's inequality to (\ref{gft_1})-(\ref{gft_2}) similarly to (\ref{Gronwall}) to remove the Gaussian component $T^{(k)} \sim n^{-a_{k-1}+c_*}$. Note that the prefactor $C_n=1+\frac{1}{\sqrt{n} (\eta^{(k)}_*)^{1+\frac{\epsilon}{2}}} $ satisfies
	\begin{align}\label{check}
		e^{\int_{0}^{T^{(k)}} C_n \dd t} \lesssim 1+T^{(k)} C_n \lesssim 1+ n^{-a_{k-1}+c_*} \Big(1+\frac{1}{ n^{1/2-a_{k}(1+\frac{\epsilon}{2})}}) \Big)  \lesssim 1,
	\end{align}
	using (\ref{eta_k}) and that $c_*\leq 1/2-\epsilon/2 \leq a_k$ for any $k\geq 1$. Hence we extend (\ref{goal_k}) down to the regime  $|z^{(k)}_i| \leq 1+\tau$, $|\eta^{(k)}_i|\gtrsim \eta^{(k)}_*= n^{-a_{k}}$, $\rho^{(k)}_i \sim \rho_i$ in the $k$-th iteration step.

	We iterate the above process and  stop at $K$-th step \nc when $\eta^{(K)}_*$ reaches $\eta_*=\min_{i=1}^{2}|\eta_i|$. The number of iterations $K$ indeed depends on $c_*$ and $\eta_*$ but it always remains $n$-independent. For example, if both $z_1$ and $z_2$ are in the bulk $|z_i|\leq 1-\tau$ and $\eta_*=n^{-1+\epsilon}$, hence $c_*=0$, then we only need to iterate twice using (\ref{eta_k}). The same also applies if both are at the edge $|z_i|=1+O(n^{-1/2})$ and $\eta_*=n^{-3/4+\epsilon}$, in which
	case $c_*=1/4-\epsilon/4$. However, if $z_1$ is in the bulk and $z_2$ is at the edge with $c_*=1/4-\epsilon/4$ and $\eta_*=n^{-1+\epsilon}$, then we need to iterate for three times. The worst situation is when $|z_1|=1-\tau$ and $|z_2|=1+\tau$ with $c_*=1/2-\epsilon/2$ and $\eta_*=n^{-1+\epsilon}$. Then from (\ref{eta_k}), we know that
	$$a_k-a_{k-1}=\frac{a_{k-1}+\frac{\epsilon}{2}}{1+\frac{\epsilon}{2}}-a_{k-1}=\frac{\epsilon}{2}\frac{1-a_{k-1}}{1+\frac{\epsilon}{2}} \lesssim \epsilon^2,$$
	using that $a_k\leq 1-\epsilon$. This implies that we need to iterate for at most $O(\epsilon^{-2})$ times to extend (\ref{goal_k}) down to the final regime $n|\eta_i| \rho_i \gtrsim n^{\epsilon}$. 	
\end{proof}

To summarize the proof of the isotropic local law, we use Figure~\ref{fig:iterate} to illustrate both the zig-zag iteration used in the proof of Proposition~\ref{prop_iteration} and the $b$-iterations used to prove (\ref{enlaw12}).
\medskip
 
\begin{figure}
	\includegraphics[width=16cm]{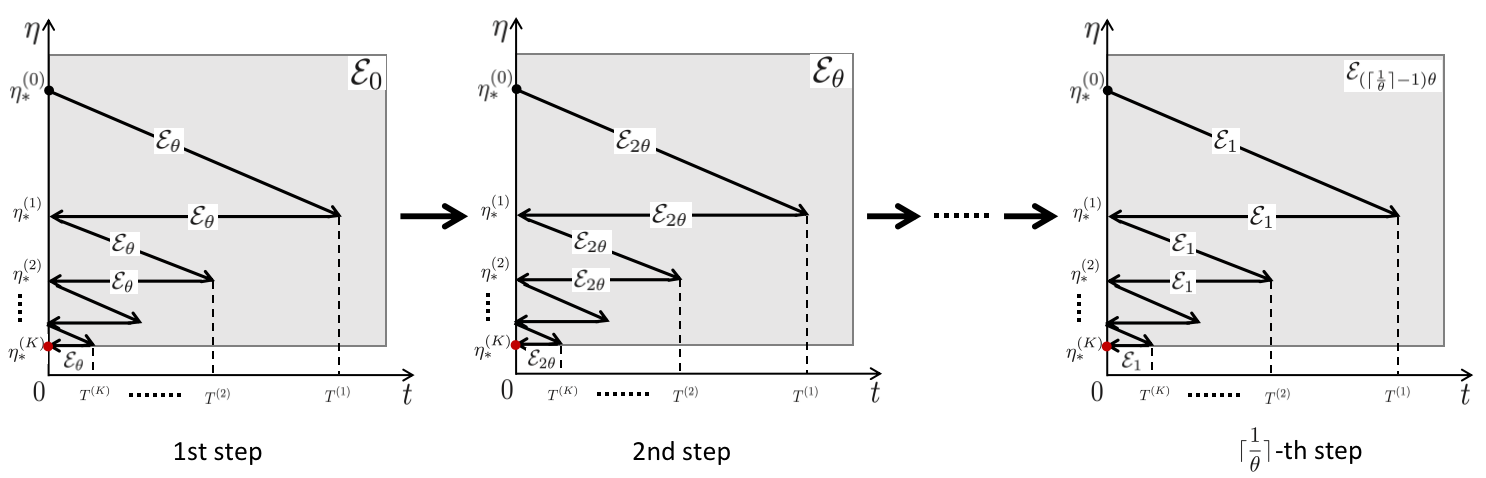}
	\centering
	\caption{In each step we use the same zig-zag as explained in Figure \ref{fig:zigzag} to improve the local law obtained in the previous step (with the error bound $\mathcal{E}_{(k-1)\theta}$ indicated on the upper right corner of the shaded square) to the next order bound $\mathcal{E}_{k\theta}$ (indicated along the zigzag lines). Starting from Theorem \ref{prop_initial} as the initial input, we iterate this process for $\lceil \frac{1}{\theta} \rceil$ times to improve the a priori error bound $\mathcal{E}_0$ to the desired $\mathcal{E}_1$.\nc}
	\label{fig:iterate}
\end{figure}

Finally, we end this section with the proof of Lemma \ref{lemma_RS_gft}.

\begin{proof}[Proof of Lemma \ref{lemma_RS_gft}]
	To simplify the notations, we  assume that  $|\eta_1| \sim |\eta_2| \sim \eta$ and $\rho_1 \sim \rho_2 \sim \rho$
	for some $\eta, \rho$. 
	The proof for general $\eta_i$ and $\rho_i$ is essentially the same with $\eta$, $\rho$ below replaced by $\eta_*$, $\rho^*$. \nc
	
	We first apply It\^{o}'s formula to $|R^{xy}_t|^{2m}=(R^{xy}_t)^{m} (\overline{R^{xy}_t)})^m$ and perform the cumulant expansions (see \eg \cite[Lemma 7.1]{He+Knowles})  as below. Without loss of generality, we may assume in the following that all the cumulants of the normalized entries of $\sqrt{n}X$ are equal to one. We thus obtain 
	\begin{align}\label{cumulant_app0}
		\frac{\dd \E |R^{xy}_t|^{2m}}{\dd t}=&	\sum_{k=3}^{K_0} \frac{1}{k!} \mathcal{T}_{k}+O_\prec\big( n^{-100m}\big)\nonumber\\
		=&\sum_{k=3}^{K_0} \frac{1}{k!} \left(\sum_{l=1}^{k}  \sum_{s_1+\cdots+s_l=k;s_i \geq 1} \mathcal{T}^{(s_1,\cdots, s_l)}_{k} \right)+O_\prec\big( n^{-100m}\big),
	\end{align}
	with the $k$-th order cumulant expansion term given by
	\begin{align}\label{higher_xy}
		\mathcal{T}^{(s_1,\cdots, s_l)}_{k}:=\frac{1}{n^{\frac{k}{2}}} \sum_{a=1}^{n} \sum_{B=n+1}^{2n} \E\Big[   
		\big(\partial^{s_1} R^{xy}_t\big)  \cdots \big(\partial^{s_l} R^{xy}_t\big)  \big(R^{xy}_t\big)^{2m-l}\Big], \qquad s_i \geq 1,
	\end{align}
 where  $\partial^{s_i}$ denotes the $s_i-$th directional derivative either in the direction $h_{aB}$ or in $h_{aB}$, with $(h_{\mathfrak{u}\mathfrak{v}})_{\mathfrak{u},\mathfrak{v}\in [2n]}:=H^{z}_t$ given by (\ref{flow}). Note that for $a,b\in [n]$ and $A,B\in [n+1,2n]$, the corresponding entries $\{h_{ab}\}$ and $\{h_{AB}\}$ all equal to zero.  For notational brevity, we will not distinguish between  $\partial h_{aB}$ and $\partial h_{Ba}$ and use the notation  $\partial$ to indicate either $\partial/\partial h_{aB}$ or $\partial/\partial h_{Ba}$.  From now on, with a slight abuse of notation $R^{xy}_t$ indicates either $R^{xy}_t\in \C$ itself or its complex conjugate $\overline{R^{xy}_t)}$.  Here we stopped the cumulant expansions in (\ref{cumulant_app0}) at a sufficiently high power $K_0=O(m)$, and the last error term $O_\prec( n^{-100m})$ was estimated using the moment condition in (\ref{eq:hmb}) and that $\|G\|\leq |\eta|^{-1}$. This argument is quite standard and already used in many previous works, \eg below \cite[Eq. (7.6)]{gumbel}, so we omit the details.

	By a direct computation, we obtain that
	\begin{align}\label{rule_xy}
		\frac{\partial R^{xy}_t}{\partial h_{aB}}=&-\big((G_1)_{\xx \ea} (G_1AG_2)_{\eB \yy}+(G_1AG_2)_{\xx \ea} (G_2)_{\eB \yy}\big).
	\end{align}
	 In general for any $s\geq 1$, each partial derivative $\partial^{s} R^{xy}_t$ contains a product of $s$ Green function entries $\prod_{i=1}^{s} (G)_{ \bm u_i, \bm v_i}$ and one factor $(G_1AG_2)_{\bm u_{s+1}  \bm v_{s+1}}$ with $\bm u_i, \bm v_i~(1\leq i\leq s+1)$ chosen from $\{\xx, \yy, \ea, \eB\}$. Moreover,  the choice of $\xx$ and $\yy$ only appears once, while both $a$ and $B$ appear $s$ times. For instance, we list below several examples of  $\partial^2 R^{xy}_t$  and $\partial^3 R^{xy}_t$, \ie
	\begin{align}
		&\partial^2 R^{xy}_t: \qquad 	 (G_1)_{\xx \eB} (G_1)_{aa} (G_1AG_2)_{\eB \yy}, \quad (G_1)_{\xx \ea} (G_1AG_2)_{BB} (G_2)_{\ea \yy};\label{second}\\
		&	\partial^3 R^{xy}_t: \qquad (G_1)_{\xx \ea} (G_1)_{aa}(G_1)_{BB} (G_1AG_2)_{\eB \yy}, \qquad (G_1)_{\xx \ea} (G_1)_{aa}(G_1AG_2)_{BB} (G_2)_{\eB \yy}.\label{third}
	\end{align}
	Moreover, we introduce the following estimates on the products of Green function entries.	Note that using (\ref{Mmatrix}) and (\ref{M_bound}), we have\nc
	\begin{align}\label{M_term}
		\sum_{\mathfrak{j}=1}^{2n} |(M)_{\uu \mathbf{e}_{\mathfrak{j}}}|^{k} \lesssim  1+  \sqrt{n} \one_{k=1}, \qquad\sum_{\mathfrak{j}=1}^{2n} |(M^A_{12})_{\uu \mathbf{e}_{\mathfrak{j}}}|^{k} \lesssim \frac{1}{\gamma^k}+\frac{\sqrt{n}}{\gamma} \one_{k=1},
	\end{align} 
	for any $k\geq 1$ and any bounded deterministic vector $\uu \in \C^n$.
	 Note that for $k=1$ we also used the Cauchy-Schwarz inequality. As a consequence of (\ref{M_term}), combining with the isotropic local law in (\ref{entrywise}) and the assumption in (\ref{assump_R}), we have 
	\begin{align}\label{estimate}
		&\sum_{\mathfrak{j}=1}^{2n} \big|(G)_{\uu \mathbf{e}_{\mathfrak{j}} }\big|^k \prec n\Big( \sqrt{\frac{\rho}{n\eta}} \Big)^k+ 1+\sqrt{n} \one_{k=1}, \qquad \sum_{\mathfrak{j}=1}^{2n} \big|(G_1 A G_2)_{\uu \mathbf{e}_{\mathfrak{j}}}\big|^k\prec  n \mathcal{E}_b^k+ \frac{1}{\gamma^k}+\frac{\sqrt{n}}{\gamma} \one_{k=1}, 
	\end{align}
	 for any $k\geq 1$ and any bounded deterministic vector $\uu \in \C^n$. We remark that the above estimates hold simultaneously for all $t\geq 0$. This follows by a standard grid argument using the Lipschitz continuity of $G_t=(H_t-\ii \eta)^{-1}$ for $t\in [0, T_0]$ with $T_0=100\log n$, while for $t\geq T_0$, $G_t$ is almost the resolvent of a Ginibre ensemble up to an error $O_{\prec}(n^{-10})$. Without specific mentioning, the estimates in the following hold uniformly for all $t$.

	Now we are ready to look at the most critical third order terms given by (\ref{higher_xy}), \ie
	\begin{align}\label{third_xy}
		\mathcal{T}^{(3)}_3=&\mathcal{T}^{(1,1,1)}_3+\mathcal{T}^{(1,2)}_3+\mathcal{T}^{(3)}_3\nonumber\\
		=&\frac{1}{n^{3/2}} \sum_{a,B} \E\Big[  \big(\partial R^{xy}_t\big)^3 \big(R^{xy}_t\big)^{2m-3} + 
		\partial R^{xy}_t\big(\partial^2 R^{xy}_t\big)\big(R^{xy}_t\big)^{2m-2} +(\partial^3 R^{xy}_t) \big(R^{xy}_t\big)^{2m-1}\Big].
	\end{align}
	Using (\ref{rule_xy}) and (\ref{estimate}) for $k=3$, the first part $\mathcal{T}^{(1,1,1)}_3$ in (\ref{third_xy}) is bounded by
	\begin{align}\label{T11}
		\big|\mathcal{T}^{(1,1,1)}_3\big|	 \lesssim& \frac{1}{n^{3/2}}\sum_{a,B} \E \Big[\Big(\Big| (G_1)_{\xx \ea} (G_1AG_2)_{\eB \yy}\Big|^{3}+\Big|(G_1AG_2)_{\xx \ea} G_{\eB \yy} \Big|^{3} \Big) \big|R^{xy}_t\big|^{2m-3}\Big]\nonumber\\
		\prec& \Big(\frac{\rho^{3/2}}{n\eta^{3/2}}+\frac{1}{\sqrt{n}} \Big) \Big( \mathcal{E}_{b}^3 +\frac{1}{n} \gamma^{-3} \Big)\E \big|R^{xy}_t\big|^{2m-3}\nonumber\\
		\lesssim& \Big(1+\frac{1}{n \eta^{3/2+\frac{\epsilon}{2}}} \Big) \Big(  \E \big|R^{xy}_t\big|^{2m} +O_\prec\big(\mathcal{E}_{b'}\big)^{2m}\Big),
	\end{align}
 where in the last step we used Young's inequality and the following bounds from (\ref{EE})- (\ref{para_bound}):
\begin{align}\label{eq_EE}
	\frac{\mathcal{E}_b}{\mathcal{E}_{b'}} =\Big( \frac{\rho \gamma}{\eta}\Big)^{\frac{b'-b}{2}} \lesssim \Big(\frac{1}{\rho \eta} \Big)^{\frac{\theta}{2}}, \qquad \frac{\gamma^{-1}}{\mathcal{E}_{b'}} \lesssim \frac{\sqrt{n}\eta}{\sqrt{\gamma}} \lesssim \sqrt{n\eta \rho},
\end{align}
with $b'=\min\big\{b+\theta,~1\big\}$, $0<\theta\leq \epsilon/10$, using that $\gamma \lesssim \rho^{-2}$ from (\ref{parameter}) and $\gamma \gtrsim \eta/\rho$.

We next consider the second part $\mathcal{T}^{(1,2)}_3$ in (\ref{third_xy}). Recall the second derivative terms for instance in (\ref{second}).  One of the most critical term (containing the least off-diagonal entries of $G$) is given by
\begin{align}\label{worst_third}
	(*):=\frac{1}{n^{3/2}} \sum_{a,B} \E\Big[   
(G_1)_{\xx \ea} (G_1AG_2)_{\eB \yy} (G_1)_{\xx \eB} (G_1)_{aa} (G_1AG_2)_{\eB \yy} \Big) \big|R^{xy}_t\big|^{2m-2} \Big].
\end{align}
Using (\ref{entrywise}) and by a Cauchy-Schwarz inequality, such term is bounded by 
	\begin{align}\label{schwarz}
		|(*)|\prec  & \frac{\rho^{3/2}}{n \sqrt{\eta}} \E \Big[ \sqrt{\sum_{B} \big|(G_1)_{\xx \eB}\big|^2 \sum_{B}  \big|(G_1 A G_2)_{\eB \yy}\big|^4} \big|R^{xy}_t\big|^{2m-2} \Big]\nonumber\\
		\prec &\rho^2 \Big( \frac{\mathcal{E}^2_b}{\sqrt{n}\eta} +\frac{1}{n \eta \gamma^2}\Big) \E \big|R^{xy}_t\big|^{2m-2} \nonumber\\
		\lesssim & \Big(1+\frac{1}{\sqrt{n} \eta^{1+\frac{\epsilon}{2}}} \Big) \Big(  \E \big|R^{xy}_t\big|^{2m} +O_\prec\big(\mathcal{E}_{b'}\big)^{2m}\Big),
	\end{align}
where we used (\ref{estimate}) for $k=2,4$ in the second step and (\ref{eq_EE}) in the last step. The other terms in $\mathcal{T}^{(1,2)}_3$ can be bounded similarly as in (\ref{schwarz}) even with less efforts. Then we conclude that 
	\begin{align}\label{T22}
\Big|\mathcal{T}^{(1,2)}_3\Big| \lesssim \Big(1+\frac{1}{\sqrt{n} \eta^{1+\frac{\epsilon}{2}}} \Big) \Big(  \E \big|R^{xy}_t\big|^{2m} +O_\prec\big(\mathcal{E}_{b'}\big)^{2m}\Big).
	\end{align}
	We finally consider the last part $\mathcal{T}^{(3)}_3$ in (\ref{third_xy}). Recall the third derivative terms in (\ref{third}). We only look at one of the worst terms (with the least off-diagonal entries of $G$), \ie
	\begin{align}\label{worst}
	(**):=&\Big|\frac{1}{n^{3/2}} \sum_{a,B} \E\Big[  (G_1)_{\xx \ea} (G_1)_{aa} (G_1)_{BB}  (G_1AG_2)_{\eB \yy} \big(R^{xy}_t\big)^{2m-1}\Big]\Big|\nonumber\\
	\prec &\Big|\frac{1}{n^{3/2}} \sum_{a,B} \E\Big[  (G_1)_{\xx \ea} (G_1-M_1)_{aa} (G_1-M_1)_{BB}  (G_1AG_2)_{\eB \yy} \big(R^{xy}_t\big)^{2m-1}\Big]\Big|\nonumber\\
	&\qquad \qquad   + \Big(\frac{\rho^{3/2}}{\sqrt{n\eta}} \mathcal{E}_b+\frac{\rho^2}{\sqrt{n} \gamma} +\frac{\rho^2}{n\eta \gamma}\Big) \E  \big|R^{xy}_t\big|^{2m-1},
	\end{align}
where we used the following isotropic bounds 
\begin{align}\label{entry}
	\Big| \frac{1}{\sqrt{n}} \sum_{a} \big( G\big)_{\xx \ea} \Big| \prec 1+\sqrt{\frac{\rho}{n\eta}}, \qquad \Big| \frac{1}{\sqrt{n}} \sum_{B} \big(G_1 A G_2\big)_{\eB \yy} \Big| \prec \frac{1}{\gamma} + \mathcal{E}_b,
\end{align}
from (\ref{entrywise}), the assumption (\ref{assump_R}) and (\ref{M_term}). Using (\ref{estimate}) for $k=1$, we obtain that 
\begin{align}\label{T33}
	|(**)|\prec& \Big(\frac{\rho^{3/2}}{n\eta^{3/2}} \mathcal{E}_b+\frac{\rho^{3/2}}{\sqrt{n\eta}} \mathcal{E}_b+\frac{\rho^2}{\sqrt{n} \gamma} +\frac{\rho^2}{n\eta \gamma}\Big) \E  \big|R^{xy}_t\big|^{2m-1}\nonumber\\ 
	\lesssim	&   \Big(1+\frac{1}{\sqrt{n} \eta^{1+\frac{\epsilon}{2}}} \Big) \Big(  \E \big|R^{xy}_t\big|^{2m} +O_\prec\big(\mathcal{E}_{b'}\big)^{2m}\Big),
	\end{align}
	where we also used (\ref{eq_EE}) and Young's inequality in the last step. The other terms in $\mathcal{T}^{(3)}_3$ can be handled similarly as above even with less efforts. Hence we obtain the same upper bound for $\mathcal{T}^{(3)}_3$ as in (\ref{T22}).
\nc
	
To sum up (\ref{T11})-(\ref{T33}), all the third order terms in (\ref{third_xy}) are bounded by
	\begin{align}
		\big|\mathcal{T}_3\big|	
		\lesssim & \Big(1+\frac{1}{\sqrt{n} \eta^{1+\frac{\epsilon}{2}}} \Big) \Big(  \E \big|R^{xy}_t\big|^{2m} +O_\prec\big(\mathcal{E}_{b'}\big)^{2m}\Big).
	\end{align}
	Note that the higher order terms in the cumulant expansion in (\ref{higher_xy}) with $k\geq 4$ can be bounded similarly using the differentiation rule (\ref{rule_xy}) and the bound (\ref{estimate}) with much less efforts, since we gain an additional $n^{-1/2}$ from the moment condition in (\ref{eq:hmb}), so we omit the details. Therefore, we conclude
	\begin{align}\label{gft_xy_imp}
		\Big|\frac{\dd \E |R^{xy}_t|^{2m}}{\dd t}\Big| \lesssim & \Big(1+\frac{1}{\sqrt{n} \eta^{1+\frac{\epsilon}{2}}} \Big)
		\Big(  \E \big|R^{xy}_t\big|^{2m} +O_\prec\big(\mathcal{E}_{b'}\big)^{2m}\Big).
	\end{align}
	This proves the first inequality in (\ref{gft_1}).

	The proof of the second inequality (\ref{gft_2}) is similar to (\ref{gft_1}), actually with less efforts, so we only sketch the modifications. We perform similar arguments on $S^{xy}_t$ as in (\ref{cumulant_app0})-(\ref{gft_xy_imp}). The  corresponding differentiation rule in (\ref{rule_xy}) is then replaced with
	\begin{align}\label{rule_S}
		\frac{\partial S^{xy}_t}{\partial h_{aB}}=&-\big((G_1)_{\xx \ea} (G_1A_1G_1A_2G_2)_{\eB \yy}+(G_1A_1G_2)_{\xx \ea}(G_2A_2G_1)_{\eB \yy}+(G_1A_1G_2AG_1)_{\xx \ea} (G_1)_{\eB \yy}\big).
	\end{align}
	Similarly to (\ref{estimate}),  we obtain from the assumption (\ref{assump_S}) and the deterministic bound (\ref{M_3_bound}) that
	\begin{align}\label{estimate_S}
		\sum_{\mathfrak{j}=1}^{2n} \big|(G_1 A_1 G_2 A_2 G_1)_{\uu \mathbf{e}_{\mathfrak{j}}}\big|^k\prec  n \wt{\mathcal{E}}_b^k+ \frac{1}{(\eta\gamma)^k}+\frac{\sqrt{n} }{\eta\gamma}\one_{k=1}.
	\end{align}
	Moreover, we have the analgue of (\ref{eq_EE}), \ie for  $b'=\min\big\{b+\theta,~1\big\}$ with $0<\theta\leq \epsilon/10$, 
	\begin{align}\label{eq_EE2}
		\frac{\wt{\mathcal{E}}_b}{\wt{\mathcal{E}}_{b'}} =\Big( \frac{ \gamma}{\eta}\Big)^{b'-b} \lesssim \Big(\frac{1}{\rho^2 \eta} \Big)^{\theta}, \qquad \frac{1}{\eta \gamma \mathcal{E}_{b'}} \lesssim 1.
	\end{align}
	Using (\ref{estimate}), (\ref{rule_S})-(\ref{eq_EE2}), and repeating similar arguments above, we obtain the analog of (\ref{gft_xy_imp})
	\begin{align}\label{gft_s_imp}
		\frac{\dd \E \big|\mathcal{S}^{xy}_t\big|^{2m}}{\dd t} \lesssim & \Big(1+\frac{1}{\sqrt{n} \eta^{1+\frac{\epsilon}{2}}} \Big)
		\Big(  \E \big|\mathcal{S}^{xy}_t\big|^{2m} +O_\prec\big(\wt{\mathcal{E}}_{b'}\big)^{2m}\Big).
	\end{align}
	This proves the second inequality (\ref{gft_2}) and completes the proof of Lemma \ref{lemma_RS_gft}.
\end{proof}

\subsection{Proof of the averaged local law in (\ref{avlaw12})}\label{sec:2G_ave}

The proof of the averaged local law in (\ref{avlaw12}) is similar to the proof of (\ref{enlaw12}) in the previous subsection, with the difference that now we additionally use (\ref{enlaw12}) and Lemma~\ref{lemma:3G} as technical inputs. Note that we focus on the regime that $\gamma \geq \eta_*/\rho^*$, while for the complementary regime the desired local law follows directly from (\ref{initial_ave}).

We first establish a similar differentiation inequality as Lemma \ref{lemma_RS_gft}. Define, for brevity, the notations
\begin{align}
	R_t:=\<\big(G_1 A_1 G_2 -M^{A_1}_{12}\big) A_2 \>, \qquad G_i=G^{z_i}_t(\ii \eta_i)=(H^{z_i}_t-\ii \eta_i)^{-1}, \qquad i=1,2.
\end{align}
For any  $|z_i| \leq 1+\tau$ and $\eta_i \in \R$ satisfying $\ell =\min_{i=1}^2 \{|\eta_i| \rho_i\} \geq n^{-1+\epsilon}$ and $\gamma \gtrsim \eta_*/ \rho^* $, we claim
\begin{align}\label{gft_RR}
	\frac{\dd \E |{R_t}|^{2m}}{\dd t}\lesssim &\Big(1+\frac{1}{\sqrt{n}\eta_*} \Big)
	\left(  \E \big|{R_t}\big|^{2m} +O_\prec\Big(\frac{1}{\sqrt{n\ell} \gamma}\Big)^{2m} \right).
\end{align}
Assume (\ref{gft_RR}) is proved. Then the desired local law (\ref{avlaw12}) follows from using similar zig-zag arguments as in the proof of  Proposition \ref{prop_iteration} above. The only difference is to replace Proposition \ref{prop:isoflow} with Proposition \ref{prop:aveflow},  and replace the global law in (\ref{eq:englobal})  with 
its averaged version from \cite[Theorem 5.2]{complex_CLT}, \ie
\begin{equation}
	\label{eq:aveglobal}
	\Big| \big\langle \big(G^{z_1}(\ii \eta_1)A_1G^{z_2}(\ii \eta_2)-M_{12}^{A_1}\big) A_2 \big\rangle \Big|\prec \frac{1}{n}, \qquad \mathrm{dist}(\ii \eta, \mathrm{supp}(\rho^z)) \sim 1.
\end{equation}
 This finishes the proof of (\ref{avlaw12}), hence the proof of Theorem \ref{thm:2G}, modulo the differentiation inequality in (\ref{gft_RR}) that we will prove below.  \qed
\medskip

\begin{proof}[Proof of (\ref{gft_RR})]
For notational simplicity, we assume that $|\eta_1| \sim |\eta_2| \sim \eta$ and $\rho_1 \sim \rho_2 \sim \rho$. The proof for general $\eta_i$ and $\rho_i$ is essentially the same with $\eta$, $\rho$ replaced by $\eta_*$, $\rho^*$ in the following.
We apply It\^{o}'s formula to $|{R_t}|^{2m}$ and perform the cumulant expansions as in (\ref{cumulant_app0})-(\ref{higher_xy}).  As explained below (\ref{higher_xy}), for notational brevity, we will use $\partial$ to indicate either $\partial/\partial h_{aB}$ or $\partial/\partial h_{Ba}$, and with a slight abuse of notation $R_t$ indicates either the complex number $R_t$ or $\overline{R_t}$. By a direct computation, we obtain that
\begin{align}\label{deri_RR}
	\frac{\partial R_t}{\partial h_{aB}}  =
	&-\frac{1}{n} \Big((G_1A_1G_2A_2G_1)_{Ba}+(G_2A_2G_1A_1G_2)_{Ba}\Big)=O_\prec \Big(\frac{1}{n\eta \gamma}\Big),
\end{align}
where we also used \eqref{M_3_bound} and  (\ref{eq:3giso}) in the last step.  In general for $s\geq 2$, each partial derivative $\partial^{s} R_t$ contains $s+2$ Green functions. The resulting terms have two forms:  1) a product of $s-1$ Green function entries $\prod_{i=1}^{s-1} (G)_{ \bm u_i, \bm v_i}$ and one factor $(G_1A_1G_2A_2 G_1)_{\bm u_{s}  \bm v_{s}}$; 2) a product of $s-2$ Green function entries $\prod_{i=1}^{s-2} (G)_{ \bm u_i, \bm v_i}$, one factor $(G_1A_1G_2)_{\bm u_{s-1}  \bm v_{s-1}}$, and another factor $(G_2A_2G_1)_{\bm u_{s}  \bm v_{s}}$, with $\bm u_i, \bm v_i~(1\leq i\leq s)$ chosen from $\{\ea, \eB\}$. Note that both choices of $\ea$ and $\eB$ appear exactly $s$ times.  For instance, we list below several examples of $\partial^2 R_t$  and $\partial^3 R_t$, \ie
\begin{align}
	&\partial^2 R_t: \qquad 	 \frac{1}{n} (G_1A_1  G_2A_2G_1)_{aa}  (G_1)_{BB}, \qquad \frac{1}{n}(G_1A_1G_2)_{aa} (G_2A_2G_1)_{BB} ;\label{second_R}\\
	&	\partial^3 R_t: \qquad \frac{1}{n} (G_1A_1  G_2A_2G_1)_{aB} (G_1)_{aa} (G_1)_{BB}, \qquad \frac{1}{n} (G_1A_1  G_2A_2G_1)_{aa} (G_1)_{BB} (G_1)_{aB},\label{third_R}\\
	&\quad \qquad \qquad \frac{1}{n}(G_1A_1G_2)_{aa} (G_2A_2G_1)_{BB} (G_1)_{aB}, \qquad \frac{1}{n}(G_1A_1G_2)_{aB} (G_2A_2G_1)_{aa} (G_1)_{BB}.
\end{align}
 Note that from the 1G local law (\ref{entrywise}) and (\ref{Mmatrix}), we have 
\begin{align}\label{G_bound}
	\max_{a=1}^{n} \max_{B=n}^{2n} \Big\{ |G_{aa}|, |G_{BB}|, |G_{aB}|,|G_{Ba}|\Big\}  \lesssim  \rho+ \one_{B=a+n},
\end{align} 
for $n\eta \rho \gtrsim  n^{\epsilon}$. Further using 2G local law (\ref{enlaw12}) and (\ref{M_bound}), we have
\begin{align}\label{GG_bound}
	\max_{1\leq \mathrm{i,j}\leq 2n} \Big\{ \Big|(G_1A_1G_2)_{\mathfrak{ij}}\Big|, \Big|(G_2A_2G_1)_{\mathfrak{ij}} \Big| \Big\}  \prec  \frac{1}{\gamma}+\frac{1}{\sqrt{n\gamma}\eta},
\end{align} 
while  the 3G bound (\ref{eq:3giso}) and (\ref{M_3_bound}) imply that
\begin{align}\label{GGG_bound}
	\max_{1\leq \mathrm{i,j}\leq 2n} \Big\{ \Big|(G_1A_1G_2 A_2 G_1)_{\mathfrak{ij}}\Big|, \Big|(G_2A_2G_1 A_1 G_2)_{\mathfrak{ij}}\Big|\Big\}  \prec  \frac{1}{\eta \gamma}.
\end{align} 
Therefore, we obtain that
\begin{align}\label{deff_rule}
	\Big| \frac{\partial^s {R_t}}{\partial h^{s_1}_{aB}\partial h^{s_2}_{Ba}} \Big| \prec\frac{\rho^{s-1}}{n\eta \gamma}+\frac{\rho^{s-2}}{n\gamma^2}+\frac{\rho^{s-2}}{n^2\gamma \eta^2} \lesssim \frac{\rho^{s}}{n\eta \rho\gamma}, \qquad s=s_1+s_2\geq 1,
\end{align}
where we also used that $\gamma \gtrsim \frac{\eta}{\rho}$ and $n\eta \rho \gtrsim  n^{\epsilon}$. 
Thus using (\ref{deff_rule}), the $k$-th order term ${\mathcal{T}}_k^{(s_1,\cdots,s_l)}$ with $s_1+\cdots+s_l=k$ given as in (\ref{higher_xy})  is bounded by
\begin{align}\label{K_high_term}
	\big|{\mathcal{T}}_k^{(s_1,\cdots,s_l)}\big| \prec & \frac{\rho^{k}}{n^{\frac{k-4}{2}}} \Big(\frac{1}{n\eta \rho \gamma}\Big)^l \E\big| {R_t}\big|^{2m-l} \nonumber\\
	\lesssim &
	\frac{\rho^{k}}{n^{\frac{k-4}{2}}}     \Big( \frac{1}{\sqrt{n\eta \rho}}\Big)^{l} \left( \E\big| {R_t}\big|^{2m} +O_\prec\Big(\frac{1}{\sqrt{n\eta \rho} \gamma}\Big)^{2m} \right).
\end{align}
Hence for any $k\geq 4$, we obtain that
\begin{align}\label{temp11}
	\big|{\mathcal{T}}_k^{(s_1,\cdots,s_l)}\big| \lesssim   \E \big|{R_t}\big|^{2m} +O_\prec\Big(\frac{1}{\sqrt{n\eta \rho} \gamma}\Big)^{2m}.
\end{align}
where we also used $n\eta \rho \gg 1$ and $\rho \lesssim 1$. For the third order terms with $k=3$ and $l\geq 2$, we also have
\begin{align}\label{temp111}
	\big|{\mathcal{T}}_k^{(s_1,\cdots,s_l)}\big| \lesssim  \frac{\rho^2}{\sqrt{n}\eta} \left(  \E \big|{R_t}\big|^{2m} +O_\prec\Big(\frac{1}{\sqrt{n\eta \rho} \gamma}\Big)^{2m} \right).
\end{align}
We next focus on the only remaining term ${\mathcal{T}}^{(3)}_3$ with $k=3$ and $l=1$. Recall the third derivative terms in (\ref{third_R}). By a direct computation, ${\mathcal{T}}^{(3)}_3$ is bounded by 
\begin{align}\label{tempp1}
	\big|{\mathcal{T}}^{(3)}_3 \big|=&\frac{\rho^2}{n^{5/2}} \sum_{a,B} \E\Big[ \Big(\big|(G_1 A  G_2 A^* G_1)_{Ba}\big|+\big|(G_2 A G_1 A^* G_2)_{Ba}\big| \Big) \big|{R_t}\big|^{2m-1}\Big]\nonumber\\
	&\qquad +O_\prec \Big( \frac{\rho^{3/2}}{n\eta^{3/2}\gamma}\E\big|{R_t}\big|^{2m-1} \Big)+O_\prec \Big(\frac{\rho}{n\gamma^2}\E\big|{R_t}\big|^{2m-1} \Big).
\end{align}
To bound the error terms in the last line above, we used the estimates (\ref{G_bound})-(\ref{GGG_bound}), together with that $|G_{aB}| \prec \sqrt{\frac{\rho}{n\eta}}+\one_{a+n=B}.$
Using the Cauchy-Schwarz inequality and Ward identity, the leading term on first line above is bounded by
\begin{align}
	&\frac{\rho^2}{n^{5/2}} \sum_{a,B} \E\Big[ \Big(\big|(G_1 A  G_2 A^* G_1)_{Ba}\big|+\big|(G_2 A G_1 A^* G_2)_{Ba}\big| \Big) \big|{R_t}\big|^{2m-1}\Big]\nonumber\\
	\lesssim &\frac{\rho^2}{n^{3/2}} \E \Big[ \sqrt{\frac{\<\Im G_1 A G_2 A^* \Im G_1 A G_2^* A^*\>}{\eta_1^2}}\big|{R_t}\big|^{2m-1}\Big]\prec \frac{\rho^2}{n^{3/2}\eta^2 \sqrt{\gamma}}\E\big|{R_t}\big|^{2m-1},
\end{align}
where in the last step, we also used that
\begin{align}
	\<G_1 A G_2 A^* G_1 A G_2 A^*\> =\frac{1}{n} \sum_{\mathfrak{i,j}=1}^{2n} (G_1 AG_2 A^*)_{\mathfrak{ij}} (G_1 AG_2 A^*)_{\mathfrak{ji}} \prec \frac{1}{\eta^2 \gamma}+\frac{1}{\gamma^2} \lesssim \frac{1}{\eta^2\gamma},
\end{align}
which follows from (\ref{enlaw12}), (\ref{M_term}) for $k=2$, and that $\gamma \gtrsim \frac{\eta}{\rho} \gtrsim \eta$. Therefore we obtain from (\ref{tempp1})
\begin{align}\label{T3}
	\big|{\mathcal{T}}^{(3)}_3 \big| \prec & \Big( \frac{\rho^2}{n^{3/2} \eta^2 \sqrt{\gamma}} +\frac{\rho^{3/2}}{n\eta^{3/2}\gamma} +\frac{\rho}{n\gamma^2}\Big) \E\big|{R_t}\big|^{2m-1} \nonumber\\ 
	\lesssim & \Big(1+\frac{1}{\sqrt{n}\eta}\Big) \left(  \E\big|{R_t}\big|^{2m}+O_\prec\Big( \frac{1}{\sqrt{n\eta \rho}\gamma} \Big)^{2m} \right),
\end{align}
where we also used that $\frac{\eta}{\rho} \lesssim \gamma \lesssim \rho^{-2}$ from (\ref{parameter}) and $n\eta \rho \gtrsim  n^{\epsilon}$.
Combining (\ref{T3}) with (\ref{temp11}) and (\ref{temp111}), 
 we have proved the differentiation inequality in (\ref{gft_RR}). \nc
\end{proof}

\subsection{Proof of Proposition \ref{prop_initial} without $|z_1-z_2|$-decorrelation}\label{sec:nogain}

In this part, we aim to prove the local laws for $\ga(\ii \eta_1)A \gb(\ii \eta_2)$ without gaining from the $|z_1-z_2|$-decorrelations. The proof is similar to the proofs in the previous subsections with less efforts, so we only sketch it.

\begin{proof}[Proof of Proposition \ref{prop_initial}]		
Note that the averaged local law  (\ref{initial_en}) has already been proved in \cite[Theorem 5.2]{complex_CLT}.  
Moreover, the last statement (\ref{initial_3G}) follows directly from using the bound (\ref{M_3_bound}) and that 
\begin{align}\label{initial_S}
	\Big|\big(G^{z_1}(\ii \eta_1)A_1G^{z_2}(\ii \eta_2)A_2G^{z_1}(\ii \eta_1)\big)_{\xx \yy}\Big| \lesssim \sqrt{\frac{\big(G^{z_1} A_1 \Im G^{z_2} A_1^* (G^{z_1})^*\big)_{\xx \xx} \big( \Im G^{z_1}\big)_{\yy \yy}}{\eta_1 \eta_2}} \lesssim \frac{1}{\eta_*^2}
\end{align}
by Schwarz inequality and Ward identity. 
So we will only consider the isotropic local law (\ref{initial_ave}). 
	As explained in the proof of Proposition \ref{prop_iteration}, it suffices to derive a similar differentiation inequality as in Lemma \ref{lemma_RS_gft}, \ie 
	\begin{align}\label{gft_xy}
		\frac{\dd \E |R^{xy}_t|^{2m}}{\dd t}\lesssim &\Big( 1+\frac{1}{\sqrt{n} \eta_*} \Big)
		\left(  \E \big|R^{xy}_t\big|^{2m} +O_\prec\Big(\frac{\sqrt{\rho}}{\sqrt{n} \eta^{3/2}} \Big)^{2m} \right).
	\end{align}
The proof uses the same strategy as Lemma \ref{lemma_RS_gft}. Actually it is much easier since we do not want to gain from $|z_1-z_2|$-decay. Again we also assume $|\eta_1| \sim |\eta_2| \sim \eta$ and $\rho_1 \sim \rho_2 \sim \rho$ to simplify the notations.

	We repeat the same cumulant expansions as in (\ref{cumulant_app0})-(\ref{higher_xy}). Recall the differentiation rules in (\ref{rule_xy})-(\ref{third}) and the third order terms in (\ref{third_xy}).  From \cite[Lemma D.2]{gumbel} we know
	\begin{align}\label{ineqqq}
		\big|(G_1A_1 G_2)_{\uu \vv} \big|  \prec \frac{\rho}{\eta}, \qquad \big|(G_1A_1G_2A_2G_1)_{\uu \vv} \big| \prec \frac{\rho}{\eta^2},
	\end{align}
	for any deterministic unit vectors $\uu, \vv$.  Then (\ref{estimate}) is now replaced with the following	\begin{align}\label{ineqqq0}
		\sum_{\mathfrak{j}=1}^{2n} \big|(G)_{\uu \mathbf{e}_{\mathfrak{j}} }\big|^k \prec n\Big( \sqrt{\frac{\rho}{n\eta}} \Big)^k+ 1+\sqrt{n} \one_{k=1}, \qquad \sum_{{\mathfrak{i}}=1}^{2n} \Big| (G_1 A G_2)_{ \mathbf{e}_{\mathfrak{i}} \vv} \Big|^k \prec  \frac{\rho^{k-1}}{\eta^{k+1}}+ \sqrt{\frac{n \rho}{\eta^3}}\one_{k=1},
	\end{align}
where the second inequality follows from (\ref{ineqqq}) and the Ward identity.
	
	Using (\ref{ineqqq0}) for $k=3$, the first term $\mathcal{T}^{(1,1,1)}_3$ in (\ref{third_xy}) is bounded by
	\begin{align}\label{argument}
		\big|\mathcal{T}^{(1,1,1)}_3\big|	 \lesssim& \frac{1}{n^{3/2}}\sum_{a,B} \E \Big[\Big(\Big| (G_1)_{\xx \ea} (G_1AG_2)_{\eB \yy}\Big|^{3}+\Big|(G_1AG_2)_{\xx \ea} G_{\eB \yy} \Big|^{3} \Big) \big|R^{xy}_t\big|^{2m-3}\Big]\nonumber\\
		\prec &  \Big( \frac{\rho^{7/2}}{n^2\eta^{11/2}} +\frac{\rho^2}{n^{3/2} \eta^{4}} \Big)  \E \big|R^{xy}_t\big|^{2m-3}.
	\end{align}
	For the second term $\mathcal{T}^{(1,2)}_3$ in (\ref{third_xy}), it then suffices to check the worst term in (\ref{worst_third}), \ie
	\begin{align}
	(*)=&  \frac{1}{n^{3/2}}\sum_{a,B} \E\Big[   
		(G_1)_{\xx \ea} (G_1AG_2)_{\eB \yy} (G_1)_{\xx \eB} (G_1-M_1)_{aa}  (G_1AG_2)_{\eB \yy} \big|R^{xy}_t\big|^{2m-2} \Big]\nonumber\\
		&\qquad  +O_\prec\Big(\frac{\rho^2}{n\eta^3}\Big)\E \big|R^{xy}_t\big|^{2m-2} 
	\prec  \Big( \frac{\rho^2}{n^{3/2}\eta^4} +\frac{\rho^2}{n\eta^3}\Big) \E \big|R^{xy}_t\big|^{2m-2},
	\end{align}
	where we used the first isotropic bound in (\ref{entry}), (\ref{ineqqq}), and  (\ref{ineqqq0}).  For the last term $\mathcal{T}^{(3)}_3$ in (\ref{third_xy}), we focus on the worst term as discussed in (\ref{worst}). Using (\ref{ineqqq}), (\ref{ineqqq0}) and the first isotropic bound in (\ref{entry}), such term is bounded by
	\begin{align}
		(**)=&\frac{1}{n^{3/2}} \sum_{a,B} \E\Big[  (G_1)_{\xx \ea} (G_1-M_1)_{aa} (G_1)_{BB}  (G_1AG_2)_{\eB \yy} \big(R^{xy}_t\big)^{2m-1}\Big]\nonumber\\
		&\qquad +O_\prec\Big(\frac{\rho^{5/2}}{\sqrt{n} \eta^{3/2}}\Big)  \E \big|R^{xy}_t\big|^{2m-1} \prec \Big(\frac{\rho^{3/2}}{n \eta^{5/2}} +\frac{\rho^{5/2}}{\sqrt{n} \eta^{3/2}} \Big) \E \big|R^{xy}_t\big|^{2m-1}.
	\end{align}
	 
		Suming up and using Young's inequality, all the third order terms in (\ref{third_xy}) are bounded by
	\begin{align}
		|\mathcal{T}_3| 
			\lesssim& \Big( 1+\frac{1}{\sqrt{n} \eta} \Big)
		\left(  \E \big|R^{xy}_t\big|^{2m} +O_\prec\Big(\frac{\sqrt{\rho}}{\sqrt{n} \eta^{3/2}} \Big)^{2m} \right).
	\end{align}
	Moreover, all the higher order terms $(k\geq 4)$ in (\ref{higher_xy}) can be bounded similarly with less efforts, since we gain additional $n^{-1/2}$ from the moment condition in (\ref{eq:hmb}). Hence we conclude with (\ref{gft_xy}) and finish the proof of Proposition \ref{prop_initial}.
\end{proof}

\subsection{Proof of Theorem \ref{thm:2G_edge}:  local law for $\Im G^{z_1} A \Im G^{z_2}$}\label{sec:im}
In this section, we will prove Theorem \ref{thm:2G_edge}, showing an improvement close to the edge of the spectrum when the resolvent $G$ is replaced by $\Im G$, using Proposition \ref{prop:aveflowim} and Theorem \ref{thm:2G}, Lemma \ref{lemma:3G} as technical inputs.  
\begin{proof}[Proof of Theorem \ref{thm:2G_edge}]
	From (\ref{rho}) we know that $\rho_i \gtrsim \eta_i^{1/3}$ for $|z_i| \leq 1+ n^{-1/2+\tau}$ and $\ell:=\min_{i=1}^2\{ |\eta_i| \rho_i \} \gtrsim n^{-1+\epsilon}$ with   $0<\tau< \epsilon/10$. This further implies that $\rho_*=\min_{i=1}^{2}\rho_i \gtrsim n^{-1/4-\epsilon/4}$. To prove Theorem \ref{thm:2G_edge}, we will focus on the regime $|z_1-z_2|^2 \gtrsim \rho_1 \eta_1 +\rho_2 \eta_2$,  which also implies that $\gamma \gtrsim \frac{\eta_*}{\rho^*}$  from (\ref{parameter}). For the complementary regime $|z_1-z_2|^2 \lesssim \rho_1 \eta_1+\rho_2 \eta_2$, the desired local law follows directly from Proposition \ref{prop_initial} without gaining from the  $|z_1-z_2|$-decay.

	We define the following short-hand notations
	\begin{align}\label{R_t}
		\wh{R_t}:=\<\big(\Im G_1 A_1 \Im G_2 -\wh{M}^{A}_{12}\big) A_2 \>, \qquad G_i=G^{z_i}_t(\ii \eta_i)=(G^{z_i}_t-\ii \eta_i)^{-1}, \quad i=1,2,
	\end{align}
	with the time dependent matrix $H^{z}_t$ defined in (\ref{flow}). Then we claim the analogue of Lemma \ref{lemma_RS_gft}, \ie
	\begin{align}\label{gft_im}
		\frac{\dd \E |\wh{R_t}|^{2m}}{\dd t}=&C'_n
		\left(  \E \big|\wh{R_t}\big|^{2m} +O_\prec\Big( \frac{1}{\sqrt{n\ell }}\frac{\rho_1 \rho_2}{\wh \gamma } \Big)^{2m} \right), \qquad C'_n:=1+\frac{1}{\sqrt{n} \eta_* \rho_*},
	\end{align}
	for any $|z_i|\leq 1+\tau$, $n\ell \gtrsim n^{\epsilon}$, $\gamma \gtrsim \frac{\eta_*}{\rho^*}$, and $|\rho_i| \gtrsim n^{-c_*}$ with $0\leq c_*\leq 1/4-\epsilon/4$. 

	 We  will prove (\ref{gft_im}) below. Assuming it now, similarly to the zig-zag arguments used in the proof of Proposition \ref{prop_iteration}, we  use (\ref{gft_im}) and Proposition \ref{prop:aveflowim} iteratively to prove (\ref{avlaw12_im}). In the first step, for any fixed $|z^{(1)}_i|\leq 1+\tau$, $|\eta^{(1)}_i|\gtrsim \eta^{(1)}_*:= n^{-\frac{1}{4}}$ with $\rho^{(1)}_i \sim \rho_i$,  we choose a small constant $T^{(1)} \sim 1$ so that $\mathrm{dist}(\ii \eta_i(0), \mathrm{supp}(\rho_i(0))) \gtrsim 1$ from  Lemma \ref{lem:ODEtheo}. \nc  Then the global law (\ref{eq:aveglobal}) and  Proposition \ref{prop:aveflowim} yield the local law (\ref{avlaw12_im}) at time $t=T^{(1)}$. Using (\ref{gft_im}), $\rho_*\gtrsim n^{-1/4}$, and similar arguments as in (\ref{Gronwall}), the Gaussian component $T^{(1)}\sim 1$ can be removed, \ie
	 \begin{align}
	 	\big|\wh{R_{0}}\big| \prec \frac{1}{\sqrt{n\ell}} \frac{\rho_1\rho_2}{\wh\gamma}, \qquad \qquad  |\eta_i| \gtrsim \eta^{(1)}_*=n^{-\frac{1}{4}}.
	 \end{align}
	
	In general for any $k\geq 1$, we define the scales recursively by
	\begin{align}\label{eta_k_im}
		\eta_*^{(k)}:=n^{-a_{k}}, \qquad a_{k}=a_{k-1}+\frac{1}{2}-2c_*, \qquad a_1=\frac{1}{4},
	\end{align}
	recalling that $\rho_*=n^{-c_*}$ with $0\leq c_* \leq 1/4-\epsilon/4$. Clearly $a_k$ is increasing with $a_k-a_{k-1} \geq \epsilon/2$.  Replacing Proposition \ref{prop:isoflow} with Proposition \ref{prop:aveflowim}, the $k$-th iteration step is similar to the $k$-th step explained in the proof of Proposition \ref{prop_iteration}. Using the new scales defined in (\ref{eta_k_im}), we obtain the analogue of (\ref{check}), \ie
	\begin{align}\label{check_analog}
		e^{\int_{0}^{T^{(k)}} C'_n \dd t} \lesssim T^{(k)} C'_n \lesssim  n^{-a_{k-1}+c_*} \Big(1+\frac{1}{ n^{1/2-a_{k}-c_*}}) \Big)  \lesssim 1,
	\end{align}
	where we used that $0\leq c_* \leq 1/4-\epsilon/4$.  We stop the iterations when $\eta^{(k)}_*$ reaches $\eta_*=\min_{i=1}^{2}|\eta_i|\gtrsim n^{-1+\epsilon}$. In the worst case when $z_1$ is in the bulk and $z_2$ is at the edge with $c_*=n^{-1/4-\epsilon/4}$ and $\eta_*=n^{-1+\epsilon}$, the number of iterations needed is at most $O(1/\epsilon)$.  This hence  proves Theorem 3.5 modulo the proof of (\ref{gft_im}) that we will do below. \nc
\end{proof}

	\bigskip

	\begin{proof}[Proof of (\ref{gft_im})]
		We apply It\^{o}'s formula to $|\wh{R_t}|^{2m}$ and perform the cumulant expansions as in (\ref{cumulant_app0}). Note that from (\ref{eq:3giso}) we have
	\begin{align}\label{rule_im}
		\frac{\partial \wh{R_t}}{\partial_{aB}}=-\frac{1}{2 n \ii  } &\Big((G_1A_1 \Im G_2A_2 G_1)_{Ba}-(G^*_1A_1 \Im G_2A_2 G^*_1)_{Ba}\nonumber\\
		&\qquad\qquad\qquad +(G_2A_2 \Im G_1A_1 G_2)_{Ba}-(G^*_2A_2 \Im G_1A_1 G^*_2)_{Ba}\Big) = O_\prec \Big(\frac{1}{n\eta_* \gamma}\Big).
	\end{align}
We remark that, even though some of $\Im G_i$ are preserved after taking derivatives, we do not need to gain extra smallness from $\Im G_i$ for these isotropic terms.
	In general for any $s\geq 1$, the partial derivative $\partial^s \wh{R_t}$ has a similar form as described in (\ref{deri_RR})-(\ref{third_R}), with certain $G$'s replaced with $\Im G$'s. Thus the same bound as in (\ref{deff_rule}) also applies to $\partial^s \wh{R_t}$, \ie
	\begin{align}\label{deff_rule_im}
		\Big| \frac{\partial^s \wh{R_t}}{\partial h^{s_1}_{aB}\partial h^{s_2}_{Ba}} \Big| \prec \frac{(\rho^*)^{s-1}}{n\eta_* \gamma}, \qquad s=s_1+s_2\geq 1,
	\end{align}
	where we also used $n\ell \gtrsim  n^{\epsilon}$ and $\gamma \gtrsim \frac{\eta_*}{\rho^*}$. Using (\ref{deff_rule_im}), the corresponding $k$-th order cumulant expansion term as in (\ref{higher_xy}), denoted by $\wh{\mathcal{T}}_k^{(s_1,\cdots,s_l)}$ with $s_1+\cdots+s_l=k$, is bounded by
	\begin{align}\label{K_high}
		\big|\wh{\mathcal{T}}_k^{(s_1,\cdots,s_l)}\big| \lesssim & \frac{(\rho^*)^{k}}{n^{\frac{k-4}{2}}} \Big(\frac{1}{n \eta_* \rho^* \gamma}\Big)^l \E\big| \wh{R_t}\big|^{2m-l}\nonumber\\ 
		\lesssim &
		\frac{1}{n^{\frac{k-4}{2}} (\rho_*)^l}   \Big( \frac{1}{\sqrt{n\ell}} \Big)^{l} \left( \E\big| \wh{R_t}\big|^{2m} +\Big(\frac{1}{\sqrt{n\ell }}\frac{\rho_1 \rho_2}{\wh \gamma}\Big)^{2m} \right).
	\end{align}
	where we also used that $n\eta_*\rho^* \geq n\ell \gtrsim n^{\epsilon}$, $\rho^*\lesssim 1$, and $k\geq l$. Using that $\rho_*\gtrsim n^{-1/4}$, the above naive estimates are already good enough to prove (\ref{gft_im}), for the terms with $5\leq k\leq 7$ and $l\leq 2k-8$, as well as all the high order terms with $k\geq 8$. For the remaining terms, we need the finer estimates below.

	If there exists at least one $s_i$ such that $s_i=1$, then, using (\ref{rule_im}), (\ref{deff_rule_im}), the Cauchy-Schwarz inequality and the Ward identity, we obtain
	\begin{align}
		\big|\wh{\mathcal{T}}_k^{(s_1,\cdots,s_l)}\big| \prec  & \frac{(\rho^*)^{k-1}}{n^{\frac{k+2}{2}}} \Big(\frac{1}{n \eta_* \rho^* \gamma}\Big)^{l-1}  \sum_{a,B}  \E\Big[ \Big(\big|(G_1 A_1 \Im G_2 A_2 G_1)_{Ba}\big|+\big|(G_2 A_1 \Im G_1 A_2 G_2)_{Ba}\big| \Big)\big|\wh{R_t}\big|^{2m-l}\Big] \nonumber\\
		\leq & \frac{(\rho^*)^{k-1}}{n^{\frac{k-1}{2}} \eta_*} \Big(\frac{1}{n \eta_* \rho^* \gamma}\Big)^{l-1} \E\Big[ \sqrt{\<\Im G_1 A_1 \Im G_2 A_2 \Im G_1 A_2^* \Im G_2 A_1^*\>}  \big|\wh{R_t}\big|^{2m-l}\Big]\nonumber\\
		\lesssim & \frac{(\rho^*)^{k-1}}{n^{\frac{k-2}{2}} \eta_*} \Big(\frac{1}{n \eta_* \rho^* \gamma}\Big)^{l-1}  \E\Big[\<\Im G_1 A_1 \Im G_2 A_2\> |\wh{R_t}|^{2m-l} \Big],
	\end{align}
	where in the last step we also used the following inequality  (from Eq. (\ref{eq:redina})) 
	\begin{align}\label{traceab}
		\<\Im G_1 A_1 \Im G_2 A_2 \Im G_1 A_2^* \Im G_2 A_1^*\> \leq n\<\Im G_1 A_1 \Im G_2 A_2\>^2.
	\end{align}
	Recall from (\ref{R_t}) and (\ref{Im_M_bound}) that 
	$$|\<\Im G_1 A_1 \Im G_2 A_2\>| \prec |\wh{R_t}|+\frac{\rho_1\rho_2}{\wh\gamma}.$$
	Thus we conclude that, if there exists at least one  $s_i=1$, then
	\begin{align}\label{s=1}
		\big|\wh{\mathcal{T}}_k^{(s_1,\cdots,s_l)}\big| \prec  &	  \frac{(\rho^*)^{k-1}}{n^{\frac{k-2}{2}} \eta_*} \Big(\frac{1}{n \eta_* \rho^* \gamma}\Big)^{l-1}   \Big( \E|\wh{R_t}|^{2m-l+1} +\frac{\rho_1\rho_2}{\wh \gamma }\E|\wh{R_t}|^{2m-l}\Big)\nonumber\\
		\lesssim& 
		\frac{1}{n^{\frac{k-2}{2}} \eta_* (\rho_*)^{l-1}}   \Big( \frac{1}{\sqrt{n\ell}} \Big)^{l-1} \left( \E\big| \wh{R_t}\big|^{2m} +\Big(\frac{1}{\sqrt{n\ell }}\frac{\rho_1\rho_2}{\wh \gamma}\Big)^{2m} \right).
	\end{align}
	We remark that if there exist at least two indices such that $s_i=s_j=1$, then using the same strategy, we gain a bit more, \ie
	\begin{align}\label{ss=1}
		\big|\wh{\mathcal{T}}_k^{(s_1,\cdots,s_l)}\big|
		\leq  & \frac{(\rho^*)^{k-2}}{n^{\frac{k}{2}} (\eta_*)^2} \Big(\frac{1}{n \eta_*  \rho^*\gamma}\Big)^{l-2}  \E\Big[\<\Im G_1 A_1 \Im G_2 A_2\>^2 |\wh{R_t}|^{2m-l} \Big]\nonumber\\
		\lesssim& 
		\frac{1}{n^{\frac{k}{2}} (\eta_*)^2 (\rho_*)^{l-2}}   \Big( \frac{1}{\sqrt{n\ell}} \Big)^{l-2} \left( \E\big| \wh{R_t}\big|^{2m} +O_\prec\Big(\frac{1}{\sqrt{n\ell }}\frac{\rho_1\rho_2}{\wh \gamma}\Big)^{2m} \right).
	\end{align}
	Using that $\rho_* \gtrsim n^{-1/4}$, it is straightforward to check that the above finer bounds (\ref{s=1}) and (\ref{ss=1}) are good enough to prove (\ref{gft_im}), for terms with $5\leq k\leq 7$, $l\leq 2k-8$, as well as the third and fourth order terms: $\wh{\mathcal{T}}^{(1,1,1)}_3$, $\wh{\mathcal{T}}^{(1,2)}_3$, $\wh{\mathcal{T}}^{(1,1,1,1)}_4$,$\wh{\mathcal{T}}^{(1,1,2)}_4$, $\wh{\mathcal{T}}^{(1,3)}_4$. It then suffices to check the remaining terms, \ie $\wh{\mathcal{T}}^{(3)}_3$, $\wh{\mathcal{T}}^{(4)}_4$, $\wh{\mathcal{T}}^{(2,2)}_4$. 
	
	We start with the third order term $\wh{\mathcal{T}}^{(3)}_3$ with $k=3$ and $l=1$.  Recall the third derivative terms as in (\ref{third_R}). Except for the terms with $(G_1 A_1 \Im G_2 A_2 G_1)_{aB}$ that can be bounded similarly using (\ref{s=1}),  the most critical and representative terms are given by
	\begin{align}
		T_{I}:=&\frac{1}{n^{5/2}} \sum_{a,B} \E\Big[ (G_1A_1 G_2A_2 G_1)_{aa} (G_1)_{BB} (G_1)_{aB}  \big(\wh{R_t}\big)^{2m-1}\Big],\\
		T_{II}:=&\frac{1}{n^{5/2}} \sum_{a,B} \E\Big[ (G_1A_1G_2)_{aa} (G_2A_2G_1)_{BB} (G_1)_{aB}  \big(\wh{R_t}\big)^{2m-1}\Big],
	\end{align}
	and all the other terms can be bounded similarly. Using the isotropic local law (\ref{entrywise}) and (\ref{M_term}),
	\begin{align} 
		|T_{I}| \prec &\frac{1}{n^{5/2}} \sum_{a,B} \E\Big[ (G_1A_1 G_2A_2 G_1-M_{121})_{aa} (G_1-M_1)_{BB} (G_1-M_1)_{aB}  \big(\wh{R_t}\big)^{2m-1}\Big]+\frac{\rho^*}{n\eta_* \gamma } \E\big|\wh{R_t}\big|^{2m-1} \nonumber\\
		& \prec \Big(\frac{\rho^*}{n^{3/2} (\eta_*)^2 \gamma}+\frac{\rho^*}{n\eta_* \gamma} \Big) \E\big|\wh{R_t}\big|^{2m-1};\label{third_im1}\\
		|T_{II}| \prec &\frac{1}{n^{5/2}} \sum_{a,B} \E\Big[ (G_1A_1G_2-M_{12})_{aa} (G_2A_2G_1-M_{12})_{BB} (G_1-M_1)_{aB}  \big(\wh{R_t}\big)^{2m-1}\Big] +\frac{1}{(n\gamma)^{3/2}\eta_*}\E\big|\wh{R_t}\big|^{2m-1}\nonumber\\
		&\prec \left(\frac{\sqrt{\rho^*}}{n^2 (\eta_*)^{5/2} \gamma}+\frac{1}{(n\gamma)^{3/2}\eta_*}\right) \E\big|\wh{R_t}\big|^{2m-1},\label{third_im2}	
	\end{align}
	where we also used (\ref{M_bound}), (\ref{M_3_bound}) and (\ref{M_term}) to bound the $M$-terms. Therefore, we obtain that
	\begin{align}
		\big|\wh{\mathcal{T}}^{(3)}_3\big| 
		\prec &  \Big(\frac{\rho^*}{n^{3/2} (\eta_*)^2 \gamma}+\frac{\rho^*}{n\eta_*\gamma } +\frac{\sqrt{\rho^*}}{n^2 (\eta_*)^{5/2} \gamma}+\frac{1}{(n\gamma)^{3/2}\eta_*}\Big) \E\big|\wh{R_t}\big|^{2m-1}\nonumber\\
		\lesssim &\frac{1}{\sqrt{n\ell}}\Big( 1+\frac{1}{\sqrt{n} \eta_* \rho_* } \Big) \left( \E\big| \wh{R_t}\big|^{2m} +O_\prec\Big(\frac{1}{\sqrt{n\ell }}\frac{\rho_1 \rho_2}{\wh \gamma}\Big)^{2m} \right).
	\end{align}
	using that $\wh\gamma/\gamma\lesssim 1$, $\gamma \gtrsim \frac{\eta_*}{\rho^*}$ and $n\ell \gg 1$. One could obtain a similar bound for $\wh{\mathcal{T}}^{(4)}_4$ and $\wh{\mathcal{T}}^{(2,2)}_4$ using (\ref{enlaw12}) and (\ref{eq:3giso}) easily, since we gain additional $n^{-1/2}$ from the fourth order cumulants.  This finishes the proof of (\ref{gft_im}).\nc	
\end{proof}

\nc

\appendix

\section{Proof of several lemmas and estimates in Section \ref{sec:zig}}\label{app:lemma_M}

We start with proving Lemma \ref{lemma_M}, whose proof relies on the following Lemma~\ref{lemma_beta}. 
Recall  the structure of the $M$ matrix from \eqref{Mmatrix} and the notations~\eqref{eq:defF}.
We use $m_j= m^{z_j}(\ii \eta_j)$ and $u_j=u^{z_j}(\ii \eta_j)$, $j=1,2$.  
 It was proven in~\cite[Appendix B]{mesoCLT} that the two--body stability operator $\mathcal{B}_{12}$ defined in (\ref{eq:defstabop}) has two non-trivial (small) eigenvalues 
\begin{equation}
	\label{eq:dedfevalues}
	\beta_\pm=\beta_\pm(z_1,\eta_1, z_2,\eta_2):=1-u_1u_2\Re[z_1\overline{z_2}]\pm\sqrt{s}, \qquad s:=m_1^2m_2^2-u_1^2u_2^2(\Im[z_1\overline{z_2}])^2,
\end{equation}
with the standard complex square root taken in the branch cut $\C \setminus (-\infty,0]$, while the remaining eigenvalues are trivially equal to one. The corresponding (right) eigenvectors of $\beta\pm$ are given by
\begin{align}\label{eigenR}
\B_{12}[R_\pm]=\beta_\pm R_\pm, \qquad 
R_{\pm}:=\begin{pmatrix}
	-u_1 u_2 \Re z_1 \bar z_2 \pm \sqrt{s} &  z_1u_1m_2+
	\frac{m_1^2z_2u_2m_2}{\ii u_1 u_2 \Im z_1 \bar z_2 \mp \sqrt{s}}  \\
	\bar z_2 u_2 m_1+\frac{m_2^2 \bar z_1 u_1m_1}{\ii u_1 u_2 \Im z_1 \bar z_2 \mp \sqrt{s}}  & \frac{m_1m_2}{\ii u_1 u_2 \Im z_1 \bar z_2 \mp \sqrt{s}}(-u_1 u_2 \Re z_1 \bar z_2 \pm \sqrt{s})
\end{pmatrix},
\end{align}
and $\B_{12}[F^{(*)}]=F^{(*)}$.  Note that $\B_{12}$ is not self-adjoint, and its adjoint operator is given by
\begin{align}\label{B_star}
	\B^*_{12}:=1-\mathcal{S}[(M^{z_1})^* \cdot (M^{z_2})^*],
\end{align}
with corresponding (left) eigenvectors given by
\begin{align}\label{eigen_B_12}
	\B^*_{12}[L^*_{\pm}]=&\overline{\beta_\pm} L^*_{\pm},\qquad\qquad L_{\pm}:=\begin{pmatrix}
		\frac{\ii u_1u_2 \Im z_1 \bar z_2 \mp \sqrt{s}}{m_1m_2} &  0  \\
		0  & 1
	\end{pmatrix}.
\end{align}
\begin{lemma}\label{lemma_beta}
	Fix any $|z_i|< 10$ and $0< |\eta_i|<10~(i=1,2)$. On the imaginary axis, both the product $\beta_+\beta_-$ and sum $\beta_++\beta_-$ are positive and
	\begin{align}
		 & 0< \beta_+\beta_- \sim \wh\gamma= |z_1-z_2|^2 +\rho_1 |\eta_1|+ \rho_2 |\eta_2|+\Big(\frac{\eta_1}{\rho_1}\Big)^2+\Big(\frac{\eta_2}{\rho_2}\Big)^2,\label{prod_formula}\\
		 & 0<\beta_++\beta_- \sim |z_1-z_2|^2+\rho_1^2 +\rho_2^2+ \frac{|\eta_1|}{\rho_1} + \frac{|\eta_2|}{\rho_2}.\label{sum_formula}
	\end{align}
Furthermore, $\beta_++\beta_--\beta_+\beta_-$ is also postive and
	\begin{align}\label{bb_formula}
		0< \beta_++\beta_--\beta_+\beta_- \sim  \rho_1^2 +\rho_2^2+ \frac{\eta_1}{\rho_1} + \frac{\eta_2}{\rho_2}.
	\end{align}
Morover, $\beta_+-\beta_-$ is either real-valued or purely imaginary and
	\begin{align}\label{diff_formula}
		0<\Re(\beta_+-\beta_-) \lesssim \rho_1\rho_2, \qquad 0<\Im(\beta_+-\beta_-) \lesssim |z_1-z_2|.
	\end{align}
Finally, we have
\begin{align}\label{beta_formula}
	\min\{|\beta_\pm |\} \gtrsim \gamma= \frac{|z_1-z_2|^2 +\rho_1 |\eta_1|+ \rho_2 |\eta_2|+\Big(\frac{\eta_1}{\rho_1}\Big)^2+\Big(\frac{\eta_2}{\rho_2}\Big)^2}{ |z_1-z_2|+\rho_1^2 + \frac{|\eta_1|}{\rho_1}+\rho_2^2+\frac{|\eta_2|}{\rho_2} }.
\end{align}
\end{lemma}

\begin{proof}[Proof of Lemma \ref{lemma_beta}]
We start with proving the upper bound of $\beta_+\beta_-$ in (\ref{prod_formula}).	Using that $2\Re(z_1 \ov{z_2})=|z_1|^2+|z_2|^2-|z_1-z_2|^2$ and that $m^2=-u+u^2 |z|^2$ from~(\ref{m_function}), we have
\begin{align}\label{product}
\beta_+\beta_-=&1-2u_1u_2 \Re(z_1 \ov{z_2})+u_1^2u_2^2 |z_1|^2 |z_2|^2-m_1^2m_2^2\nonumber\\
=&1-u_1 u_2 \big( |z_1|^2+|z_2|^2-|z_1-z_2|^2\big)+u_1^2u_2^2 |z_1|^2 |z_2|^2-m_1^2m_2^2\nonumber\\
=&1-u_1 u_2 \big( 1+ |z_1|^2(1-u_1)+|z_2|^2(1-u_2)-|z_1-z_2|^2\big)\nonumber\\
=& u_1 u_2 |z_1-z_2|^2-m_1^2 \frac{u_2}{u_1} (1-u_1)-m_2^2 \frac{u_1}{u_2} (1-u_2)+(1-u_1)(1-u_2)
\end{align}
From (\ref{rho}) we know $\rho \gtrsim |\eta| $ and hence
$$1-u=\frac{|\eta|}{|\eta|+|m|}  \sim \frac{|\eta|}{\rho}.$$
Thus we conclude that
\begin{align}\label{bbpp}
	 \beta_+\beta_- \sim |z_1-z_2|^2 +\rho_1 |\eta_1|+ \rho_2 |\eta_2|+\Big|\frac{\eta_1\eta_2}{\rho_1\rho_2}\Big| 
\end{align}
This proves the upper bound in (\ref{prod_formula}) for any $z$ by a simple  Cauchy-Schwarz inequality. To obtain the comparable lower bound in (\ref{prod_formula}), we split the discussion into two cases. For $|z|<1$, it follows directly from (\ref{bbpp}) and that $\rho^3 \gtrsim |\eta|$ from (\ref{rho}). For the complementary regime $|z|>1$, since $\rho^3 \lesssim |\eta|$, the desired lower bound in (\ref{prod_formula}) was already obtained in Appendix A~(Lemma A.1) of the gumbel paper.

We continue to prove the remaining estimates. Similarly to (\ref{product}) we have
\begin{align}\label{sum}
	\beta_++\beta_- = 2-2u_1u_2\Re[z_1\overline{z_2}]= &  u_1 u_2 |z_1-z_2|^2  -\frac{u_2}{u_1} m_1^2-\frac{u_1}{u_2} m_2^2+ 1-u_1 +1-u_2\nonumber\\
	\sim & |z_1-z_2|^2+\rho_1^2 + \frac{|\eta_1|}{\rho_1}+\rho_2^2 + \frac{|\eta_2|}{\rho_2}.
\end{align}
This proves (\ref{sum_formula}). Combining (\ref{product}) with (\ref{sum}), we have
$$ \beta_++\beta_--\beta_+\beta_-=1-u_1u_2-u_2m_1^2-u_1 m_2^2 \sim \rho_1^2 + \frac{|\eta_1|}{\rho_1}+\rho_2^2 + \frac{|\eta_2|}{\rho_2},$$
which proves (\ref{bb_formula}). Similarly, we have $\beta_+-\beta_-=2\sqrt{m_1^2m_2^2-u_1^2u_2^2(\Im[z_1\overline{z_2}])^2}$ in the standard branch cut $\C\setminus (-\infty,0]$. Note that $m_1^2m_2^2-u_1^2u_2^2(\Im[z_1\overline{z_2}])^2$ is real valued, which implies that $\beta_+-\beta_-$ is either real or purely imaginary. More precisely, we have
\begin{align}
	\Re(\beta_+-\beta_-) =&2 \Re \Big( \sqrt{m_1^2m_2^2-u_1^2u_2^2(\Im[z_1\overline{z_2}])^2}\Big) \leq 2 \rho_1\rho_2,\nonumber\\
\Im(\beta_+-\beta_-) =&2 \Im \Big( \sqrt{m_1^2m_2^2-u_1^2u_2^2(\Im[z_1\overline{z_2}])^2}\Big) \leq 2 u_1 u_2 |\Im[z_1\overline{z_2}]| \lesssim |z_1-z_2|,
\end{align}
which prove (\ref{diff_formula}). 

Finally it still remains to prove (\ref{beta_formula}). Combining (\ref{bb_formula}) with (\ref{prod_formula}), we obtain that
\begin{align}
	\frac{1}{\beta_-}+\frac{1}{\beta_+}-1=\frac{\beta_++\beta_--\beta_+\beta_-}{\beta_+\beta_-} \lesssim \frac{ \rho_1^2 + \frac{|\eta_1|}{\rho_1}+\rho_2^2 + \frac{|\eta_2|}{\rho_2}}{|z_1-z_2|^2 +\rho_1 |\eta_1|+ \rho_2 |\eta_2|+\Big(\frac{\eta_1}{\rho_1}\Big)^2+\Big(\frac{\eta_2}{\rho_2}\Big)^2}.
\end{align}
Similarly, combining (\ref{diff_formula}) with (\ref{prod_formula}), we obtain that
\begin{align}
	\Big|\frac{1}{\beta_-}-\frac{1}{\beta_+}\Big|=	\Big|\frac{\beta_+-\beta_-}{\beta_+\beta_-}	\Big| \lesssim \frac{ \max\{ \rho_1\rho_2, |z_1-z_2|\}}{|z_1-z_2|^2 +\rho_1 |\eta_1|+ \rho_2 |\eta_2|+\Big(\frac{\eta_1}{\rho_1}\Big)^2+\Big(\frac{\eta_2}{\rho_2}\Big)^2}.
\end{align} 
Therefore, we obtain that
\begin{align}
	\Big|\frac{1}{\beta_\pm}\Big| \lesssim \Big|\frac{1}{\beta_-}+\frac{1}{\beta_+}\Big|+\Big|\frac{1}{\beta_-}-\frac{1}{\beta_+}\Big|\lesssim \frac{ |z_1-z_2|+\rho_1^2 + \frac{|\eta_1|}{\rho_1}+\rho_2^2+\frac{|\eta_2|}{\rho_2} }{|z_1-z_2|^2 +\rho_1 |\eta_1|+ \rho_2 |\eta_2|+\Big(\frac{\eta_1}{\rho_1}\Big)^2+\Big(\frac{\eta_2}{\rho_2}\Big)^2},
\end{align}
which finishes the proof of Lemma \ref{lemma_beta}.
\end{proof}

\begin{proof}[Proof of Lemma \ref{lemma_M}]	
	Recall from (\ref{eq:defM12}) that 
	\begin{align}
		\label{app:defM12}
		M_{12}^A= \mathcal{B}^{-1}_{12}\Big[M^{z_1}(\ii \eta_1) A M^{z_2}(\ii \eta_2)\Big].
	\end{align}
Then (\ref{M_bound}) directly follows from  (\ref{beta_formula}), \ie
	\begin{align}\label{app_M_12}
		\|M^{A}_{12}\| \lesssim \|\mathcal{B}^{-1}_{12}\| \leq \frac{1}{\min\{|\beta_\pm |\}} \lesssim \frac{1}{\gamma}.
	\end{align}
Moreover, $\wh{M}_{12}^A$  is a fourfold linear combination of $M^A_{12}(\pm \eta_1, z_1, \pm \eta_2, z_2)$ using the identity $\Im G=\frac{1}{2\ii}(G-G^*)$ twice. To prove  (\ref{Im_M_bound}), it suffices to show that
	\begin{align}\label{000}
		\left|\left\<\mathcal{B}^{-1}_{12}\Big[M^{z_1}(\ii \eta_1) A_1 M^{z_2}(\ii \eta_2)\Big] A_2\right\>-\left\<\mathcal{B}^{-1}_{12}\Big[M^{z_1}(\ii \eta_1) A_1 M^{z_2}(-\ii \eta_2)\Big] A_2\right\>\right| \lesssim \frac{\rho_1\rho_2}{\wh \gamma}.
	\end{align}
We will focus on the worst case when $A_1,A_2$ are along the bad directions $L_\pm$. Note that $L_\pm$ are linear combinations of $E_\pm$. It then reduces to prove (\ref{000}) for $A_1,A_2 \in \{E_+, E_-\}$. By a direct computation from (\ref{eigen_B_12}), we obtain
 \begin{align}\label{eigen_E}
 	\big(\B^*_{12}\big)^{-1}[E_\pm]=
 	\begin{pmatrix}
 		\frac{1- z_1 u_1 \bar z_2 u_2 \pm m_1m_2}{\beta_-\beta_+}   &  0 \\
 		0  & \frac{m_1m_2  \pm(1- \bar z_1 u_1   z_2 u_2)}{\beta_-\beta_+}
 	\end{pmatrix}.
 \end{align}
Hence using (\ref{eigen_E}) and (\ref{Mmatrix}), we obtain that
	\begin{align}
		\left\<\mathcal{B}^{-1}_{12}\Big[M^{z_1}(\ii \eta_1) E_\pm M^{z_2}(\ii \eta_2)\Big] E_\pm\right\>=&\left\< \big(\B^*_{12}\big)^{-1}[E_\pm], M^{z_1}(\ii \eta_1) E_\pm M^{z_2}(\ii \eta_2)\right\>\nonumber\\
		=& \pm \frac{m_1^2m_2^2 \pm m_1m_2+\Re[z_1 \ov{z_2}] u_1u_2-|z_1|^2 |z_2|^2 u_1^2 u_2^2 }{\beta_+ \beta_-};\label{111}\\
		\left\<\mathcal{B}^{-1}_{12}\Big[M^{z_1}(\ii \eta_1) E_- M^{z_2}(\ii \eta_2)\Big] E_+\right\>=&-\left\<\mathcal{B}^{-1}_{12}\Big[M^{z_1}(\ii \eta_1) E_+ M^{z_2}(\ii \eta_2)\Big] E_-\right\>=\frac{ -\ii \Im [z_1 \ov{z_2}] u_1u_2}{\beta_+ \beta_-}.\label{222}
	\end{align}
	Using that $u(-\ii \eta)=u(\ii \eta)$ and $m(-\ii \eta)=-m(\ii \eta)$, we obtain the following cancellation
	\begin{align}\label{improve_im}
		\left|\left\<\mathcal{B}^{-1}_{12}\Big[M^{z_1}(\ii \eta_1) A_1 M^{z_2}(\ii \eta_2)\Big] A_2\right\>-\left\<\mathcal{B}^{-1}_{12}\Big[M^{z_1}(\ii \eta_1) A_1 M^{z_2}(-\ii \eta_2)\Big] A_2\right\>\right| \lesssim \frac{\rho_1\rho_2}{\beta_+\beta_-} \lesssim \frac{\rho_1\rho_2}{\wh \gamma},
	\end{align}
	for any $A_1,A_2 \in \{E_{+}, E_{-}\}$, which proves (\ref{Im_M_bound}).

	 We next prove the last statement (\ref{M_3_bound}).
	 Recall from (\ref{eq:defM121}) and (\ref{cov}) that
	\begin{align}
		M_{121}^{A_1,A_2}=&\mathcal{B}_{11}^{-1}\left[M_1A_1M_{21}^{A_2}+M_1 \Big( \langle M_{12}^{A_1} E_+\rangle E_+-\langle M_{12}^{A_1} E_-\rangle E_- \Big) M_{21}^{A_2}\right].
	\end{align}
		From \cite[Eq. (3.8)]{complex_CLT}, we know $\mathcal{B}^{*}_{11}$ has two small eigenvalues 
	\begin{align}\label{beta_0}
		\beta_*:=\beta_-(z,\eta, z,\eta) \sim \frac{\eta}{\Im m}, \qquad \beta:=\beta_+(z,\eta, z,\eta) \sim \frac{\eta}{\Im m}+(\Im m)^2,
	\end{align}
	with the left eigenvectors given by $E_-$ and $E_+$, respectively. It then suffices to test the above matrix against these two directions $\{E_+, E_-\}$, \ie
	\begin{align}\label{temp}
		\Big \langle M_{121}^{A_1,A_2} E_\pm \Big \rangle=&\frac{\big\langle M_{12}^{A_1} E_+\big\rangle    \big\langle  E_\pm M_1 E_+ M_{21}^{A_2}  \big\rangle -\big\langle M_{12}^{A_1} E_-\big\rangle    \big\langle E_\pm M_1 E_- M_{21}^{A_2}  \big\rangle+ \big\langle E_\pm M_1 A_1 M_{21}^{A_2} \big\rangle}{\beta_\pm(z_1,\eta_1, z_1,\eta_1)}.
	\end{align}
Note that using (\ref{beta_0}) and (\ref{app_M_12}) the last term above is easily bounded by 
\begin{align}\label{naive_M}
\left|\frac{1}{\beta_{(*)}} \big\langle E_\pm M_1 A_1 M_{21}^{A_2} \big\rangle \right| \lesssim \frac{\rho_1}{\eta_1 \gamma}. 
\end{align}
For the remaining terms, using (\ref{app:defM12}), we further write
\begin{align}
	&\big\langle M_{12}^{A_1} E_+\big\rangle   \big\langle  E_\pm M_1 E_+ M_{21}^{A_2}  \big\rangle -\big\langle M_{12}^{A_1} E_-\big\rangle    \big\langle E_\pm M_1 E_- M_{21}^{A_2}  \big\rangle \nonumber\\
	=& \big\langle (\mathcal{B}^*_{12})^{-1}[E^*_+], M_1 A_1 M_2\big\rangle    \big\langle  (\mathcal{B}^*_{21})^{-1}[E^*_+ M^*_1 E^*_\pm], M_2 A_2 M_1  \big\rangle \nonumber\\
	&\qquad -\big\langle (\mathcal{B}^*_{12})^{-1}[E^*_-], M_1 A_1 M_2 \big\rangle    \big\langle (\mathcal{B}^*_{21})^{-1}[E^*_- M^*_1 E^*_\pm], M_2 A_2 M_1  \big\rangle .
\end{align}
 Using $M_1=m_1 E_+-z_1 u_1 F-\ov{z_1} u_1 F^*$, the matrices of the form $E_\pm^* M_1^* E^*_\pm$ can be written as linear combiniations of $E_\pm$, $F$ and $F^*$. Then we introduce the following estimates.  To simplify the notations, we may assume $\rho\sim \rho_1 \sim \rho_2$ and $\eta \sim \eta_1 \sim \eta_2$. The same upper bound estimates below also apply to general $\eta_i$ and $\rho_i$, with $\rho$ and $\eta$ replaced by $\rho^*$ and $\eta_*$.   
\begin{align}\label{eigen_EE}
	\big(\B^*_{12}\big)^{-1}[E_\pm]=
	\begin{pmatrix}
		\frac{1- z_1 \bar z_2 \pm m_1m_2+O(\eta/\rho)}{\beta_-\beta_+}   &  0 \\
		0  & \frac{m_1m_2  \pm(1- \bar z_1   z_2 +O(\eta/\rho))}{\beta_-\beta_+}
	\end{pmatrix},
\end{align}
\begin{align}\label{eigen_F}
	\big(\B^*_{12}\big)^{-1}[F]
	=& 	\begin{pmatrix}
		\frac{m_2 (\ov{z_1}-\ov{z_2})+O(\eta)}{\beta_-\beta_+}  &  1 \\
		0  & \frac{m_1 (\ov{z_2}-\ov{z_1})+O(\eta)}{\beta_-\beta_+}	\end{pmatrix},
\end{align} 
\begin{align}\label{eigen_F_star}
	\big(\B^*_{12}\big)^{-1}[F^*]
	=& 	\begin{pmatrix}
		\frac{m_1 (z_2 -z_1)+O(\eta)}{\beta_-\beta_+}  &  0 \\
		1  & \frac{m_2 (z_1-z_2)+O(\eta)}{\beta_-\beta_+}	\end{pmatrix},
\end{align}
which follow directly from (\ref{B_star})-(\ref{eigen_B_12}) and that $u=1+O(\eta/\rho)$. By direct computations using (\ref{eigen_EE})-(\ref{eigen_F_star}) and (\ref{beta_0}), we obtain that
\begin{align}
	&\left|\frac{\big\langle M_{12}^{A_1} E_+\big\rangle    \big\langle  E_+ M_1 E_+ M_{21}^{A_2}  \big\rangle - \big\langle M_{12}^{A_1} E_-\big\rangle    \big\langle E_+ M_1 E_- M_{21}^{A_2}  \big\rangle}{\beta}  \right| \lesssim \frac{1}{\rho^2}\frac{\rho^5+\rho|z_1-z_2|^2}{(\wh \gamma)^2} \lesssim \frac{\rho^2}{\eta \wh \gamma},\nonumber\\
	&\left|\frac{\big\langle M_{12}^{A_1} E_+\big\rangle    \big\langle  E_- M_1 E_+ M_{21}^{A_2}  \big\rangle - \big\langle M_{12}^{A_1} E_-\big\rangle    \big\langle E_- M_1 E_- M_{21}^{A_2}  \big\rangle }{\beta_*}   \right| \lesssim \frac{\rho}{\eta} \frac{\rho|z_1-z_2|^2}{(\wh\gamma)^2} \lesssim \frac{\rho^2}{\eta \wh \gamma},
\end{align} 
for $|z|\leq 1$, where we also used $\rho^3 \gtrsim \eta$, $\wh\gamma=\beta_-\beta_+ \gtrsim |z_1-z_2|^2+\rho\eta$,  and 
$$m^2=|z|^2-1+O(\rho \eta), \qquad 2\Re(z_1 \ov{z_2})=|z_1|^2+|z_2|^2-|z_1-z_2|^2, \qquad |\Im(z_1 \ov{z_2})| \sim |z_1-z_2|.$$
Plugging the above estimates and (\ref{naive_M}) into (\ref{temp}) we obtain that
\begin{align}
	\Big \langle M_{121}^{A_1,A_2} E_\pm \Big \rangle \lesssim \frac{(\rho^*)^2}{\eta_* \wh\gamma}+\frac{\rho_1}{\eta_1 \gamma} \lesssim \frac{1}{\eta_*\gamma}.
\end{align}
It hence proves (\ref{M_3_bound}) for $|z|\leq 1$. The proof for the complementary regime $1<|z|\leq 10$ is similar and actually easier using that $\rho^3 \leq \eta$, so we omit the details.  We hence finish the proof of Lemma \ref{lemma_M}. 
\end{proof}

\begin{proof}[ Proof of (\ref{eq:explcompdet}) for $\eta_1\eta_2<0$]
	Note that for $|z_i| \leq 1-c$, from (\ref{eq:dedfevalues}) and (\ref{eigen_B_12})  we know that, $\beta_+\sim 1$, $\beta_- \sim \gamma$, and $L_-$ is the only bad direction. Recall from \cite[Eq. (B.6)]{mesoCLT} that
	\begin{align}\label{ev_zz}
		E_\pm=\big(1+O(|z_1-z_2|)\big)L_{\pm\sigma}+O(|z_1-z_2|)L_{\mp\sigma}, \qquad \sigma:=\sgn(\eta_1\eta_2).
	\end{align}
	Decomposing the matrix $A$ as $A=\<AE_+\>E_++\<AE_-\>E_-+A_c$, we obtain that
	\begin{align}\label{app_M2}
		&	\big|\langle M_{12}^{A} E_{-\sigma}\rangle\big|\lesssim  \big|\langle M_{12}^{E_+}E_{-\sigma}\rangle\big|+\big|\langle M_{12}^{E_-}E_{-\sigma} \rangle\big|+\big|\langle M_{12}^{A_c} E_{-\sigma}\rangle\big| \lesssim \big|\langle M_{12}^{E_{-\sigma}}E_{-\sigma}\rangle\big|+\frac{|z_1-z_2|}{ \gamma},\nonumber\\
		&\big|\langle M_{12}^{A} E_{\sigma} \rangle\big| \lesssim \big|\langle M_{12}^{E_+} E_{\sigma} \rangle\big|+\big|\langle M_{12}^{E_-} E_{\sigma} \rangle\big|+\big|\langle M_{12}^{A_c} E_{\sigma} \rangle\big|\lesssim \frac{|z_1-z_2|}{\gamma}.
	\end{align}
	This proves (\ref{eq:explcompdet}) for $\eta_1\eta_2<0$.  Similar upper bounds also hold for $\eta_1\eta_2>0$ with switched signs.
\end{proof}

\begin{proof}[Proof of Lemma \ref{lemma_meta}]
	The first relation for $M_{12}$ was proved in Eq. (A.26) in the Appendix A.5 of \cite{mesoCLT}. The  second relation for $M_{121}$ then follows by a so-called meta argument \cite{CHNR, CEHS23}, using recursive Dyson equations for chains of product of resolvents, see e.g. \cite[Lemma D.1]{CEHS23}.
\end{proof}

\end{document}